\documentclass[a4paper,10pt]{amsart}
\usepackage[a4paper]{geometry}
\usepackage{enumerate}
\usepackage[all, knot]{xy}
\usepackage{tikz}
\usepackage{wrapfig}
\usetikzlibrary{matrix,arrows,decorations.pathreplacing}
\usepackage{color}
\usetikzlibrary{calc,shapes} 
\xyoption{rotate}
\usepackage{amsthm,amsmath,amssymb,amsfonts}
\usepackage{subcaption}
\usepackage{array}
\usepackage[disable]{todonotes} 
\usepackage[backend=bibtex, firstinits=true,doi=false,url=false]{biblatex}
\usepackage[pdftex, citecolor=black, urlcolor=black,
  linkcolor=black, colorlinks=true, bookmarksopen=true,bookmarksdepth=3]{hyperref}
\usepackage{imakeidx}
\newtheorem{thm}{Theorem}[section]
\newtheorem{lemma}[thm]{Lemma}
\newtheorem{cor}[thm]{Corollary}
\newtheorem{prop}[thm]{Proposition}

\theoremstyle{definition}
\newtheorem{defi}[thm]{Definition}

\newtheorem{rmk}[thm]{Remark}

\tikzset{verta/.style={shape=regular polygon,regular polygon sides=3,shape border rotate=180, inner sep=1pt, draw }}
\tikzset{verto/.style={shape=diamond, draw,inner sep=1.2pt }}
\tikzset{vertb/.style={shape=regular polygon,regular polygon sides=3,shape border rotate=-90, inner sep=1pt,draw }}
\tikzset{vert/.style={shape=circle, draw }}
\tikzset{vertf/.style={shape=circle, fill=white, inner sep=1pt, draw,scale=0.7}}
\tikzset{vertn/.style={shape=circle, fill=white, inner sep=1pt, draw}}
\tikzset{vertg/.style={shape=rectangle,inner sep=1.5pt, draw }}
\tikzset{vertgo/.style={shape=ellipse,inner sep=1.5pt, draw }}
\tikzset{lea/.style={fill=white,inner sep=1pt}}
\tikzset{leag/.style={fill=white, rounded corners=2pt,inner sep=1pt}}
\tikzset{edg/.style={draw,line width=0.2mm}}
\tikzset{scall/.style={every node/.style={transform shape}}}
\newcommand{\outcheckpt}[4]
{
\foreach \x in {1,...,#2}
    {
     \node (out#1\x) at ($(#1)+(#3*2*\x-#3*#2-#3,#4)$) {};
    }
}
\newcommand{\incheckpt}[4]
{
\foreach \x in {1,...,#2}
    {
     \node (in#1\x) at ($(#1)+(#3*2*\x-#3*#2-#3,#4)$) {};
    }
}
\newcommand{\outports}[3]
{
\foreach \x in {1,...,#2}
    {
     \node (outp\x) at ($(#1)+(#3*2*\x-#3*#2-#3,0)$) {};
    }
}
\newcommand{\inports}[3]
{
\foreach \x in {1,...,#2}
    {
     \node (inp\x) at ($(#1)+(#3*2*\x-#3*#2-#3,0)$) {};
    }
}
\newcommand{\connectnodes}[3]
{
\foreach \x / \y / \z in {#3}
 \draw (#1) .. controls (out#1\x) and (in#2\y) .. node[lea] {\z} (#2);
}
\newcommand{\connectnodess}[3]
{
\foreach \x / \y in {#3}
 \draw (#1) .. controls (out#1\x) and (in#2\y) .. (#2);
}
\newcommand{\connectntoop}[2]
{
\foreach \x / \y / \z in {#2}
 \draw(#1) .. controls (out#1\x) .. node[lea] {\z} (outp\y) ;
}
\newcommand{\connectntoops}[2]
{
\foreach \x / \y in {#2}
 \draw(#1) .. controls (out#1\x) .. (outp\y) ;
}
\newcommand{\connectntoinp}[2]
{
\foreach \x / \y / \z in {#2}
 \draw  (inp\x) .. controls (in#1\y) .. node[lea] {\z} (#1) ;
}
\newcommand{\connectntoinps}[2]
{
\foreach \x / \y in {#2}
 \draw  (inp\x) .. controls (in#1\y) .. (#1) ;
}
\newcommand{\grnode}[5]
{
\foreach \x / \y / \s / \t in {#2}
    {
     \node[vertg] (fn#1\x) at ($(#1)+(#4*2*\x-#4*#3-#4,#5)$) {$ \y$};
     \node (in#1\x) at ($(#1)+(#4*2*\x-#4*#3-#4,2*#5)$) {};
     \draw[line width=0.7mm] (#1) -- (fn#1\x) node [midway,leag] {\footnotesize $ \t$}; 
     \draw[line width=0.7mm] (fn#1\x) -- (in#1\x) node [midway,leag] {\footnotesize $ \s$};
    }
}

\newcommand{\grout}[3]
{
     \node (out#1) at ($(#1)-(0,#3)$) {};
     \draw[line width=0.7mm] (#1) -- (out#1) node [midway,leag] {\footnotesize $#2$}; 
}

\newcommand{\Enr}[1]{#1 \text{-}}
\newcommand{\Cat}{\mathsf{Cat}}
\newcommand{\Catadj}{\mathsf{Cat}_{\mathrm{adj}}}
\newcommand{\ECat}[1]{#1 \text{-}\Cat}
\newcommand{\ECatfc}[2]{\ECat{#1}_{#2}}
\newcommand{\Oper}{\mathsf{Oper}}
\newcommand{\OPER}{\mathsf{OPER}}
\newcommand{\EOper}[1]{#1 \text{-} \Oper}
\newcommand{\EOPER}[1]{#1 \text{-} \OPER}
\newcommand{\EOperfc}[2]{\EOper{#1}_{#2}}
\newcommand{\Prop}{\mathsf{Prop}}
\newcommand{\PROP}{\mathsf{PROP}}
\newcommand{\EProp}[1]{#1 \text{-} \Prop}
\newcommand{\EPROP}[1]{#1 \text{-} \PROP}
\newcommand{\EPropfc}[2]{\EProp{#1}_{#2}}
\newcommand{\Cl}[1]{\mathbf{cl}(#1)}
\newcommand{\Opop}[1]{\mathsf{Op}_{#1}}
\newcommand{\Opcat}[1]{\mathsf{Cat}_{#1}}
\newcommand{\fOpprop}{\mathsf{PROP}}
\newcommand{\Opprop}[1]{\mathsf{PROP}_{#1}}
\newcommand{\Opcfprop}[1]{\mathsf{cfPROP}_{#1}}
\newcommand{\Opafprop}[1]{\mathsf{afPROP}_{#1}}

\newcommand{\EMGraph}[1]{#1 \text{-} \mathsf{MultiGraph}}
\newcommand{\EGraph}[1]{#1 \text{-} \mathsf{Graph}}

\newcommand{\N}{\mathbb{N}}
\newcommand{\id}{\mathrm{id}}
\newcommand{\Set}{\mathcal{S}et}

\newcommand{\HoC}[1]{\mathrm{Ho}(#1)}
\newcommand{\cell}[1]{#1 \text{-}\mathrm{cell}}
\newcommand{\inj}[1]{#1 \text{-}\mathrm{inj}}

\newcommand{\Mm}{\mathcal{V}}
\newcommand{\Un}{\mathbb{I}}
\newcommand{\arr}[2]{#1 \rightarrow #2}
\newcommand{\arro}[3]{#2 \!\overset{#1}\longrightarrow \! #3}

\newcommand{\lmorp}[3]{#1 \colon #2 \longrightarrow ~#3}
\newcommand{\morp}[3]{#1 \colon #2 \!\rightarrow\!\!\! ~#3}
\newcommand{\ntra}[3]{#1 \colon #2 \Longrightarrow #3}
\newcommand{\gfun}[5]{
\begin{array}{r@{}c@{\hspace{0.1cm}}c@{\hspace{0.1cm}}l}
#1 \colon & #2 & \longrightarrow & #3 \\
             & #4 & \longmapsto & #5\\
\end{array}
}

\newcommand{\adjpair}[4]{#1 \colon #3 \rightleftarrows #4 \colon #2}
\newcommand{\indadjpair}[3]{\adjpair{{#1}_!}{{#1}^*}{#2}{#3}}
\newcommand{\diagc}[9]{\xymatrix{#1 \ar @{} [dr] |{#9} \ar[d]_{#5} \ar[r]^-{#6} &  #2 \ar[d]^-{#7}\\
          #3 \ar[r]_-{#8} & #4}}
\newcommand{\diagcn}[9]{\xymatrix{#1 \ar[d]_{#5} \ar[r]^-{#6} &  #2 \ar@{}[dl]^(.3){}="a"^(.7){}="b" \ar@2{->}_-{#9} "a";"b" \ar[d]^-{#7}\\
          #3 \ar[r]_-{#8} & #4}}
\newcommand{\coeq}[4]{\xymatrix{#1 \ar@<0.7ex>[r]^-{#3} \ar@<-0.7ex>[r]_-{#4} &  #2}}
\newcommand{\coeqv}[4]{\xymatrix{#1 \ar@<0.7ex>[d]^-{#3} \ar@<-0.7ex>[d]_-{#4} \\  #2}}

\newcommand{\verx}[1]{\mathrm{vert}(#1)}
\newcommand{\card}[1]{|#1|}
\newcommand{\arity}[1]{a(#1)}

\newcommand{\opc}[1]{#1^{\mathrm{op}}}

\newcommand{\Sqofcf}{\mathrm{Sign}}

\newcommand{\Sqofc}[1]{\Sqofcf(#1)}

\newcommand{\ECfProp}[1]{#1 \text{-} \mathsf{cfProp}}
\newcommand{\EAfProp}[1]{#1 \text{-} \mathsf{afProp}}

\newcommand{\Ob}[1]{\mathbf{ob}(#1)}

\newcommand{\Clf}{\mathbf{cl}}
\newcommand{\trf}[1]{#1^{u}}
\newcommand{\trc}[1]{#1_{u}}

\newcommand{\Alg}[2]{\mathrm{Alg}_{#1}(#2)}
\newcommand{\fAlg}{\mathrm{Alg}}

\newcommand{\fSet}{\mathsf{Fin}}

\newcommand{\fcar}[1]{\overline{#1}} 
\newcommand{\Gr}{\mathsf{Gr}} 
\newcommand{\resd}[1]{\mathrm{res}(#1)}
\newcommand{\Val}[1]{\mathrm{Val}(#1)}
\newcommand{\fVal}[1]{\mathbf{Bi}(#1)}

\newcommand{\cval}[1]{\mathbf{val}(#1)}
\newcommand{\ToBi}[2]{\AlgF{#1}{#2}}

  \newcommand{\fsm}[1]{k_!#1}
  \newcommand{\FSm}[1]{w_!#1}
  \newcommand{\adjsm}[1]{\widetilde{#1}}
\newcommand{\Cor}[1]{\mathrm{Cor}(#1)}

  \newcommand{\BVt}{\underset{BV} \otimes}
\newcommand{\GraK}{\mathfrak{G}\mathrm{r}}
  \newcommand{\Fey}{\mathfrak{F}\mathrm{eyn}}
  \newcommand{\EFey}[1]{\Enr{#1}\Fey}
\newcommand{\otw}[2]{\mathrm{tw}(#1,#2)}
\newcommand{\tw}[2]{\mathrm{tw}(#1,#2)}
\newcommand{\ModCat}{\mathrm{ModCat}}
\newcommand{\disf}[1]{\Cl{#1}}

\newcommand{\inva}[1]{{#1}_{\mathrm{in}}}
 
\newcommand{\outva}[1]{{#1}_{\mathrm{out}}}

\newcommand{\ioval}[2]{
(#1 ; #2)}
\newcommand{\iovalv}[2]{
\left[\!
    \begin{array}{c}
      #1 \\
      #2
    \end{array}
  \!\right]}
\newcommand{\iport}[1]{\mathbf{in}(#1)}
\newcommand{\oport}[1]{\mathbf{out}(#1)}
\newcommand{\port}[1]{\mathbf{port}(#1)}

\newcommand{\resop}[2]{{#1}|_{#2}}
\newcommand{\dPo}{\mathbf{P}}
  \newcommand{\Str}[1]{\mathsf{sq}(#1)}
  \newcommand{\covf}[3]{{#1}_{{#2}^{-1}(#3)}}
\newcommand{\cie}[2]{\mathbf{cie}(#1,#2)}
\newcommand{\ror}[2]{#1 |_{#2}}

  \newcommand{\ope}{\mathcal}
  \newcommand{\prp}{\mathcal}
  \newcommand{\cat}{\mathcal}
  \newcommand{\ofam}{\mathbf}
  \newcommand{\Mf}{\mathcal{F}}
  \newcommand{\Mfg}{\mathcal{G}}
  \newcommand{\Mop}{\ope{O}}
  \newcommand{\End}[1]{\mathrm{End}_{#1}} 
  \newcommand{\Endsm}[1]{\mathrm{End}_{#1}} 
\newcommand{\IHom}[3]{\mathrm{Hom}_{#1}(#2,#3)} 

\newcommand{\CAT}{\mathbf{CAT}} 
\newcommand{\CATadj}{\CAT_{\mathrm{adj}}} 

\newcommand{\str}{\mathbf}
\newcommand{\EndP}[1]{\mathsf{End}_{#1}} 
\newcommand{\coa}{coa }
\newcommand{\Aut}{\mathrm{Aut}}
  \newcommand{\FSmf}{T}
  
  \newcommand{\Foff}{S}
  
\newcommand{\block}[1]{\mathbf{b}(#1)}
  \newcommand{\preFey}{\mathsf{pre}\mathfrak{F}}
  \newcommand{\EpreFey}[1]{#1 \text{-} \preFey}
  \newcommand{\LKan}[1]{\mathsf{Lan}_{#1}}
    \newcommand{\Wcolim}[3]{#2 \star #3}
    \newcommand{\LaxM}[2]{\mathrm{LaxMon}(#1,#2)}  
\newcommand{\stpt}[1]{\mathsf{st}(#1)}
\newcommand{\ept}[1]{\mathsf{end}(#1)}

  \newcommand{\alf}[2]{#1[#2]}
  \newcommand{\AlgF}[2]{\Alg{#1}{#2}}
  
  \newcommand{\gc}{\textstyle\int}
  \newcommand{\gco}{\textstyle\int}

  \newcommand{\smxylabel}[1]{{\makebox[.02\textwidth][c]{$#1$}}}

  \newcommand{\Wlim}[3]{\{#2,#3\}}

\newcommand{\comcat}[2]{{#1}\downarrow {#2}} 
\newcommand{\fSigma}{\mathbf{\Sigma}}
\newcommand{\hof}[2]{{#1}|_{#2}}
\newcommand{\fbv}[3]{\gco (\hof{#1}{#3})}
  \newcolumntype{M}[1]{ >{\centering\arraybackslash} m{#1} }
\newcounter{tikzl}
\newcommand{\grnodes}[5]
{
\node[vertgo,fill=white] (#1-lab) at (#1) {#5};
\foreach \x in {1,...,#2}
    {
     \node (#1-\x) at ($(#1)+(#3*\x-0.5*#2*#3-0.5*#3,#4)$) {};
     \draw[line width=0.7mm] (#1-lab) -- (#1-\x.center); 
    }
}
\newcommand{\grouts}[3]
{
     \coordinate (out#1) at ($(#1)-(0,#2)$);
     \draw[line width=0.7mm] (#1-lab) -- (out#1) node [midway,lea] {$#3$}; 
}
\newcommand{\permg}[2]
{
\foreach \x / \y in {#1}
    {
     \coordinate (\x-up) at ($(\x)+(0,#2)$);
     \draw[line width=0.7mm] (\x-up.center) -- (\y.center); 
    }
}
\newcommand{\permgl}[2]
{
	\foreach \x / \y /\z in {#1}
	{
		\node[lea] (\y-l) at (\y) {\z};   
		\coordinate (\x-up) at ($(\x)+(0,#2)$);
		\draw[line width=0.7mm] (\x-up.center) -- (\y); 
	}
}
\newcommand{\permgrl}[2]
{
	\foreach \x / \y /\z in {#1}
	{   
		\coordinate (\x-up) at ($(\x)+(0,#2)$);
		\draw[line width=0.7mm] (\x-up.center) -- (\y) node [midway,lea] {\z}; 
	}
}
\newcommand{\centerp}[4]
 {
 	\foreach \x in {1,...,#2}
 	    \coordinate (#1-\x) at ($(#1)+(#3*\x-0.5*#2*#3-0.5*#3,#4)$);
 }
\newcommand{\upcopyp}[3]
{
	\foreach \x in {#1}
	{
		\coordinate (\x#3) at ($(\x)+(0,#2)$);
	}
} 
\newcommand{\upcopy}[2]
{
 \upcopyp{#1}{#2}{-up}
}

\newcommand{\conx}[5]
{
	\foreach \x / \y in {#1}
	{
		\draw[line width=0.7mm] (#2\x#3) -- (#4\y#5);
	}
}

\newcommand{\groupn}[2]
{
    \setcounter{tikzl}{0};
    \foreach \x in {#1}
    {
        \addtocounter{tikzl}{1};
        \coordinate (#2\the\value{tikzl}) at (\x);
    }
}
\numberwithin{equation}{subsection}

\title{The Dwyer-Kan model structure for enriched coloured PROPs}
\author{Giovanni Caviglia}
\address{Radboud Universiteit Nijmegen, Institute for Mathematics, Astrophysics, and Particle Physics\\ Heyendaalseweg 135\\ 6525 AJ Nijmegen\\ the Netherlands}
\email{g.caviglia@math.ru.nl}


\begin{document} 
    
\begin{abstract}
    Coloured PROPs are a generalisation of coloured operads.
    In this article, we prove the existence of a Dwyer-Kan model structure on the category of small  coloured PROPs enriched in a (sufficiently nice) monoidal model category $\mathcal{V}$. This model structure extends the known Dwyer-Kan model structures on the categories of (small) $\Mm$-enriched categories and $\Mm$-enriched coloured operads. 
\end{abstract}    

\maketitle

\setcounter{tocdepth}{1}
\tableofcontents
\section{Introduction}

PROPs were invented by Mac Lane and Adams (see \cite{McL65} for the original definition) and provide an effective way to encode algebraic structures with operations with multiple inputs and multiple outputs in symmetric monoidal categories.

In homotopy theory, (topological) PROPs (called \emph{categories of operators with permutations} in \cite{BV68}) were used by Boardman and Vogt in the study of homotopy invariance of algebraic structures
in topological spaces \cite{BV73}; in their work they mainly restrict themselves to PROPs whose set of operations is generated
by those with only one output; these special PROPs (also called \emph{categories of operators in standard form} in \cite{BV68})
correspond to the well-known notion of operad.

Operads were defined by May for the formulation of his recognition principle for iterated loop spaces \cite{Ma72} and 
have been proved to be an effective tool to study the homotopy theories of algebraic structures; see for example 
\cite{BV73}, Rezk \cite{Re96} for simplicial operads, Markl \cite{Ma04} for dg-operads and the works of Berger and Moerdijk \cite{BM03}, \cite{BM05},\cite{BM06} for the more general case; one of the insights of this theory is that homotopy invariant structures are described by cofibrant operads (in a suitable sense). 

Many algebraic structures such as (commutative) monoids, Lie algebras, Poisson algebras, etc. can be described by operads;
still other important structures such as (Frobenius) bialgebras and Hopf algebras can only be modeled by general 
PROPs. 

Coloured PROPs (and operads), i.e. PROPs with multiple objects, are useful to describe (and study the homotopy theory of) more involved algebraic structures such as
morphisms and diagrams between algebras over (uncoloured) PROPs (see \cite{BM05}) and their resolutions; the theories of PROPs and operads themselves can also be described via coloured operads (\cite{C14} and \S\ \ref{sec.opprop}). 

As for operads, PROPs can be enriched in any symmetric monoidal category $\Mm$.
If $\Mm$ has a (sufficiently nice) model structure, the category of $\Mm$-enriched $C$-coloured PROPs (i.e. with a fixed set of colours $C$) inherits a (semi-)model structure (see \cite{Fr10}). 

In analogy with operads, it is therefore natural to investigate, for example, weather cofibrant PROPs have homotopy invariant algebras. 

The homotopy invariance of algebras over enriched (coloured) PROPs in the context of model categories
 was studied by Fresse in \cite{Fr10} and Johnson and Yau in \cite{JY09}. 
The higher degree of complexity of the algebraic structures presented by PROPs makes their study more difficult
with respect to operadic theories; for example, there is not always a free construction for this kind of algebras, thus
there is no way to transfer the model structure from the ground category to the category of algebras in general.  
In the dg-context, Yalin (\cite{Ya14}) was able to circumvent this lack of model structures, proving a homotopy invariance
result at the level of simplicial localizations of the categories of algebras. 

As we mention, under suitable conditions on $\Mm$, the category of $\Mm$-enriched $C$-coloured PROPs, that will be denoted by
$\EPropfc{\Mm}{C}$, admits a model structure. However, in this article we are interested in $\EProp{\Mm}$, the category of all (small) $\Mm$-enriched coloured PROPs, where the set of colours is allowed to vary. There are natural inclusions of the category of (small) $\Mm$-enriched categories and of the category of (small) $\Mm$-enriched coloured operads in $\EProp{\Mm}$:
\[
 \ECat{\Mm} \longrightarrow \EOper{\Mm} \longrightarrow \EProp{\Mm}.
\]

\noindent The $2$-categorical structure on $\ECat{\Mm}$ (and $\EOper{\Mm}$) extends to a $2$-categorical structure on $\EProp{\Mm}$ in such
a way that the above inclusions become inclusions of $2$-categories. We thus get a notion of equivalence between coloured 
PROPs which is weaker than the isomorphism relation. 
When $\Mm$ is a monoidal model category the homotopical notion corresponding to an equivalence of $\Mm$-enriched PROPs 
is the one of Dwyer-Kan weak equivalence ( def. \ref{defi.DKequi}).

Dwyer-Kan weak equivalences between simplicial categories were first introduced by Dwyer and Kan in their seminal work on simplicial localization (\cite{DK80}).
Bergner (\cite{Be07}) proved that Dwyer-Kan weak equivalences are the weak equivalences for a model structure on the category of simplicial categories that models $\infty$-categories.
 
The same result for $\Mm$-enriched categories was proven by Lurie \cite[A.3]{Lu09}, Berger and Moerdijk \cite{BM12} and Muro \cite{Mu14} for 
$\Mm$ an arbitrary monoidal model category (satisfying certain hypotheses). 
 
Cisinski and Moerdijk \cite{CM11} and Robertson \cite{Ro11} independently proved that simplicial coloured operads can be endowed with a model structure where 
the weak equivalences are the Dwyer-Kan weak equivalences. In \cite{C14} the author addressed the same problem for coloured operads
enriched in an arbitrary monoidal model category. 

The main goal of this paper, achieved in Theorem \ref{main.thm}, is to prove that under suitable conditions on a combinatorial 
monoidal model category $\Mm$, there exists a combinatorial model structure on $\EProp{\Mm}$ in which the weak equivalences are the Dwyer-Kan weak equivalences; the same result holds also for $\ECfProp{\Mm}$ (resp. $\EAfProp{\Mm}$), the category of constant-free (resp. augmentation-free) coloured $\Mm$-PROPs (\S \ref{sec2constantfree}).

Examples of model categories $\Mm$ for which the theorem applies are: simplicial sets (with the Kan-Quillen model structure),
chain complexes over a field of characteristic zero (with the projective model structure) and symmetric spectra with the positive stable model structure.

As in \cite{BM12} we also show that, under more restrictive hypotheses on $\Mm$, the fibrant objects for our model structure
coincide with the locally fibrant one. 

In the case $\Mm=\Set$ (with the trivial model structure) the Dwyer-Kan model structure on $\EProp{\Set}$ has equivalences of PROPs
(in the $2$-categorical sense) as weak equivalences (in analogy with the ``folk'' model structure on the category of small categories).

The case $\Mm=\mathrm{sSet}$, the category of simplicial sets with the Kan-Quillen model structure, was investigated by Hackney and Robertson in \cite{HR14}.

As in \cite{C14}, the proof of Theorem \ref{main.thm} relies on the Interval Cofibrancy Theorem of \cite{BM12} and its analogue in \cite[Theorem 7.13, 7.14]{Mu14} and a careful analysis of push-outs in the category of PROPs (theorem \ref{main.po}).

Another crucial point is that, even though $\EProp{\Mm}$ (as $\EOper{\Mm}$, $\ECat{\Mm}$) can not be described as the category of algebras for an operad, it is the total category of a bifibration over $\Set$ and each fiber of this bifibration is the category of algebra for a certain operad. We call such bifibrations \emph{operadic bifibrations}. To study them, we introduce the notion of operadic family and a kind of operadic Grothendieck construction. The author believes that these tools can be exploited in the study of other ``coloured'' structures similar to PROPs, categories and operads such as PROs, properads, and wheeled PROPs.

Lastly, we would like to bring the reader's attention to the following issue: the operad describing $C$-coloured
PROPs is not $\Sigma$-free, therefore, even if the unit of $\Mm$ is cofibrant, it will not be a $\Sigma$-cofibrant operad (seen as
an operad in $\Mm$). In general, only for a $\Sigma$-cofibrant operad $\mathcal{O}$ the homotopy theory of strict algebras is equivalent to the homotopy theory of homotopy algebras (i.e. algebras for a cofibrant replacement of $\mathcal{O}$, \emph{cf.} theorem 4.4 \cite{BM03}); this implies that the model structure on $\EPropfc{\Mm}{C}$ (and $\EProp{\Mm}$) might not present the homotopy theory of homotopy PROPs. This might not be a problem when $\Mm$ is a symmetric flat monoidal model category such as chain complexes over a field of characteristic $0$ or symmetric spectra with the positive model structure, 
since the rectification result of Pavlov and Scholbach \cite[Theorem 9.3.6 and \S\ 7]{PS14} guarantees
that the homotopy theories of homotopy PROPs and strict PROPs are equivalent. However this condition is not satisfied,
for example, by $\mathrm{sSet}$ (with the Kan-Quillen model structure).

In the general case, depending on the applications, one might want to restrict himself to constant-free PROPs or augmentation-free PROPs; the operads encoding these structure are $\Sigma$-free, therefore the above problem is solved in these cases. 

\subsubsection*{Organization of the paper:}
In \S\ \ref{sec.enrprop} we recall the definitions of enriched coloured PROP, bicollection, and morphisms between them.
In \S\ \ref{secoperbif} we introduce operadic families, i.e. families of operads parametrized by a category, that will
help us to define $\EProp{\Mm}$ as the total category of a special kind of bifibration over $\Set$, i.e. an operadic bifibration. To study diagram in operadic bifibrations we introduce a Grothendieck construction for operads, that permits to describe those diagrams as algebras for certain operads.

In \S\ \ref{sechomfib} few facts about the transferred model structure on $\EPropfc{\Mm}{C}$ are recalled.

Section \ref{sec.main} is the central section of the paper: the Dwyer-Kan model structure for enriched coloured PROPs is introduced
and sufficient conditions for its existence are given in Theorem \ref{main.thm}.
In \S\ \ref{sec.fibob} it is proven that, when $\Mm$ is right proper and the unit is cofibrant, the fibrant objects in $\EProp{\Mm}$
coincides with the locally fibrant ones.

The last sections of the paper contain the most technical parts of the paper. In \S\ \ref{sec.opprop} we carefully 
describe the operad for $C$-coloured PROPs; the composition law of this operad is based on graph insertion; the definition of graph that we use (borrowed from \cite{Ko14}) and graph insertion are recalled in appendix \ref{ch:graph}.

Our main result about push-outs of PROPs is proved in Section \ref{sec.poalongop}; operadic families and the operadic Grothendieck are used to efficiently describe the combinatorics of push-outs in $\EProp{\Mm}$.

In \S\ \ref{sec.locpresop} we study filtered colimits in operadic families. The results therein are used to show that if $\Mm$ is locally presentable, then $\EOper{\Mm}$ is locally presentable.  

Adjunctions between the categories of algebras induced by a morphism of operads are extensively used through-out the paper. In Appendix \ref{sec: alg as Kan extensions} we give a characterisation of the left adjoints of these adjunctions as left Kan extensions. This is closely related to the theory of Feynman categories of Ward and Kaufmann (\cite{KW14}); in fact we show that the notion of coloured operad and Feynman category are equivalent in a certain sense (\S \ref{sec:fey cat}).

\subsection*{Acknowledgment} The author would like to thank Ieke Moerdijk for his constant support and guidance and Dimitri Ara and Javier Guti\'errez for many insightful discussions around this work.

\subsection{Notation and conventions}\label{sec:Malgebras}

In general we will mostly ignore set-theoretical size-issues and we shall assume that we are working in a fixed universe of sets; a set will be called \emph{small} if it belong to this universe.

The category of (small) sets will be denoted by $\Set$.
The full subcategory of $\Set$ spanned by the finite sets will be denoted by $\fSet$.

For every non-negative integer $n\in\N$ let $\fcar{n}$ be $\{1,\dots,n\}$, the finite cardinal of order $n$.
Given a finite set $C$ we will denote by $\card{C}\in \N$ its cardinality.

For every $n\in \N$ the symbol $\Sigma_n$ will denote the group of automorphisms of $\fcar{n}$ (i.e. the $n$-permutation group).\\
We denote by $\fSigma$ \emph{finite permutation groupoid} the \index{finite permutation groupoid}, i.e. the maximal subgroupoid of $\Set$ that has $\{\fcar{n} \mid n\in \N\}$ as set of objects;
clearly $\fSigma\cong \coprod_{n\in \N} \Sigma_n$.

The term ``operad'' will always stand for \emph{coloured symmetric operads} (\emph{cf.} \cite{BM05}, also called symmetric multicategories
in the literature) if not specified otherwise.

The $2$-category of (locally small) categories will be denoted by $\CAT$, while $\Cat$ will denote the ($2$-)category of small categories.
Similarly $\OPER$ will stand for the $2$-category of (locally small) operads, while $\Oper$ will denote the $2$-category of small operads.

Given a bicomplete closed symmetric monoidal category $(\Mm,\otimes,\Un_{\Mm})$ (\cite{McL98}) and a $\Mm$-category $\cat{W}$ (i.e. a category enriched in $\Mm$), for every $a,b\in \cat{W}$ we will denote
by $\IHom{\Mm}{a}{b}\in \Mm$ the hom-object from $a$ to $b$. 

A \emph{cocomplete symmetric $\Mm$-algebra} will be for us a cocomplete symmetric monoidal $\Mm$-category $\cat{W}$ (\cite[\S B.5]{AJ13}) together with a strong symmetric monoidal $\Mm$-functor $\morp{z}{\Mm}{\cat{W}}$.

\begin{rmk}\label{rmk:setalg}
Every cocomplete symmetric monoidal category $(\cat{W},\otimes,\Un)$ with monoidal product that commutes with coproducts in each variable is a cocomplete symmetric $\Set$-algebra (where $\Set$ is taken with the cartesian monoidal structure)  via the functor $\morp{-\cdot \Un}{\Set}{\Mm}$ that sends $S\in \Set$ to the coproduct $\coprod_S \Un$.
\end{rmk}
\noindent We assume the reader to be familiar with the basic theory of model categories and operads (and algebras over them); for the former \cite{Ho99} should contain all the definitions and results that we will need, for the latter the reader is referred to \cite{BM03}, \cite{BM05}; see also \cite{MSS02} for a more broad survey on the theory of operads.

Section \ref{secoperbif} and appendix \ref{sec: alg as Kan extensions} make a mild use of (enriched) weighted colimits and limits, we refer the reader to the book of Kelly \cite{Ke82} for the notation.

\section{Enriched coloured PROPs}\label{sec.enrprop}

For the whole section $(\Mm,\otimes,\Un)$ will be a bicomplete closed symmetric monoidal category.
 
To define (enriched) coloured PROPs, we will proceed gradually starting from the definition of uncoloured non-enriched PROP.  

Uncoloured (or one-coloured) PROPs (\emph{cf.} \cite{McL65}, \cite{La04}, \cite{Mk08}) are symmetric (strict) monoidal categories whose monoid of objects
is freely generated by one object; it is custom to identify the set of objects of an uncoloured PROP with $\N$.
 
Given a symmetric monoidal category $\cat{M}$ and a PROP $\prp{P}$, one can define the
category of algebras of $\prp{P}$ in $\Mm$, denoted by $\Alg{\prp{P}}{\Mm}$, as the category of strong monoidal functor from $\prp{P}$ to $\Mm$.

Since $\prp{P}$ is freely generated on one object $1$ such a functor $\morp{A}{\prp{P}}{\cat{M}}$ is determined (up to isomorphisms) on objects by the value $A(1)$, in fact $A(n)\cong A(1)^{\otimes n}$.
Each $f\in \prp{P}(n,m)$ determines a morphism 
\[
\arro{A(f)}{A(1)^{\otimes n}}{A(1)^{\otimes m}}
\]
that can be regarded as an operation with $n$ inputs and $m$ outputs defined on $A(1)$. These operations on $A(1)$
are subject to constrains dictated by the composition in $\prp{P}$.

(Uncoloured) operads can be seen as particular PROPs; in fact, they are nothing but PROPs in which the morphisms are freely generated (via the monoidal product) by the ones with target $1$.

PROPs enriched in $\Mm$ (or $\Mm$-PROPs) are defined as symmetric monoidal $\Mm$-category freely generated by one object.

As for operads (and algebraic theories) PROPs come also in a $C$-coloured (or multi-sorted) version: for every set $C$, a $C$-coloured $\Mm$-PROPs is a symmetric monoidal $\Mm$-category with monoid of objects is freely generated by $C$.

Algebras for coloured PROPs are defined as expected: every $C$-coloured $\Mm$-enriched PROP $\prp{P}$ and every symmetric $\Mm$-enriched monoidal model category $\mathcal{M}$ the category of algebras of $\prp{P}$ in $\mathcal{M}$ can be defined as the category of strong monoidal $\Mm$-functor from $\prp{P}$ to $\mathcal{M}$ (\emph{cf.} \S \ref{sec:algprop}).\\
Even though this is one of the most important and motivational aspects of the theory, we will not be concerned with algebras over PROPs in this paper; we refer the reader to \cite{Fr10} and \cite{JY09} for more details on the theory of algebras and their homotopy theory.

As for (enriched) categories and coloured operads (\emph{cf.} \cite{C14}),  there is a category 
$\EProp{\Mm}$ whose objects are all the coloured $\Mm$-PROPs with all possible set of objects.
The functor of colours $\morp{\Clf}{\EProp{\Mm}}{\Set}$ is a bifibration over $\Set$ (\S\ \ref{secoperbif}); for each $C\in \Set$ the 
fiber over $C$ is isomorphic to $\EPropfc{\Mm}{C}$, the category of $C$-coloured $\Mm$-enriched PROPs.

\subsection{Valences and bicollections}\label{sec:valences and bicollections}  
For every set $C\in \Set$ let $(\Str{C},\ast,[])$ be the free monoid generated by $C$.
The underlying set $\Str{C}$ is isomorphic to the set of finite ordered sequences of elements of $C$, 
that is the set of couples $(n,\str{c})$ where $n\in \N$ and $\morp{\str{c}}{\fcar{n}}{C}$ is a function. An element $(n,\str{c})$ of $\Str{C}$ can also be represented
as an array $(c_1,\dots,c_n)$ where $c_i=\str{c}(i)$ for every $i\in \fcar{n}$; the integer $n$ is called the \emph{length}\index{length!of a sequence} of $(n,\str{c})$ and denoted by
$\card{(n,\str{c})}$. When not relevant for the discussion we will omit the length and denote $(n,\str{c})$ simply by $\str{c}$.
If we think of $\N$ as the discrete subcategory of $\Set$ spanned by the sets $\fcar{n}$ for every $n\in \N$ and we think of $C$ as an
object of $\Set$ then $\Str{C}$ is isomorphic to the (discrete) comma category $\comcat{\N}{C}$.

The monoid structure on $\Str{C}$ is given by the concatenation of sequences, more explicitly
for every couple of function $\morp{\str{f}}{\fcar{n}}{C}$, $\morp{\str{g}}{\fcar{m}}{C}$ we define 
$\morp{\str{f}\ast \str{g}}{\fcar{n+m}}{C}$ in the following way:
\[
 (\str{f}\ast \str{g})(i)=
\begin{cases}
 \str{f}(i) & \text{if } i\leq n \\
 \str{g}(i-n) & \text{otherwise;}
\end{cases}
\]
the unit $[]$ is the unique sequence of length $0$.

We denote by $\Sigma\Str{C}$ the comma category $\comcat{\Sigma}{C}$; $\Sigma\Str{C}$ is actually a groupoid and it will be called the \emph{$C$-sequences permutation groupoid }\index{permutation groupoid! $C$-sequences}: its set of objects
is isomorphic to $\Str{C}$ and for every $\sigma\in \Sigma_n$ there is an isomorphism $\morp{\tilde{\sigma}}{(n,\str{c}\sigma)}{(n,\str{c})}$.

A \emph{$C$-valence} is an element of $\Val{C}=\Str{C}\times \Str{C}$. Given $(n,\str{c}),(m,\str{d})\in \Str{C}$
the $C$-valence $\mathbf{v}=((n,\str{c}),(m,\str{d}))$ will be represented in different ways
\[
 \mathbf{v}=\ioval{\inva{\mathbf{v}}}{\outva{\mathbf{v}}}=\ioval{\mathbf{c}}{\mathbf{d}}=\iovalv{c_1,\dots,c_n}{d_1,\dots,d_m}
\]

For every two $C$-valences $\mathbf{v},\mathbf{u}\in \Val{C}$ let $\mathbf{v}\ast\mathbf{u}\in \Val{C}$ be the $C$-valence
represented by $\ioval{\inva{\mathbf{v}}\ast\inva{\mathbf{u}}}{\outva{\mathbf{v}}\ast\outva{\mathbf{u}}}$. 

\begin{defi}
The \emph{category of $C$-bicollection} in $\Mm$ is $\Mm^{\Val{C}}$, the category of functors from $\Val{C}$ (seen as a discrete category) to $\Mm$.
\end{defi}

\noindent A \emph{$C$-bicollection $\mathbf{P}$ in $\Mm$}  is thus a collection $\{\mathbf{P}(\mathbf{a};\mathbf{b})\}_{\ioval{\mathbf{a}}{\mathbf{b}}\in \Val{C}}$ of objects in $\Mm$.

We can also define the \emph{permutation groupoid of $C$-valences}\index{permutation groupoid! of $C$-valences} $\Sigma \Val{C}$ as $\opc{\Sigma~\Str{C}}~\times~ \Sigma\Str{C}$;
the set of objects of $\Sigma \Val{C}$ is $\Val{C}$; for every $\sigma\in \Sigma_n$ and $\tau \in \Sigma_m$ and every $\mathbf{a},\mathbf{b}\in \Sqofc{C}$ such that $\card{\mathbf{a}}=n$, $\card{\mathbf{b}}=m$
there is a morphism
\[
 \lmorp{(\sigma,\tau)}{\ioval{\mathbf{a}}{\mathbf{b}}}{\ioval{\mathbf{a}\sigma}{\mathbf{b}\tau^{-1}}}.
\]

\begin{defi}
The \emph{category of $C$-coloured symmetric $\Mm$-bicollection}\index{bicollection!symmetric} in $\Mm$ is the functor category $\Mm^{\Sigma\Val{C}}$.
\end{defi} 

\subsection{Coloured PROPs}
Coloured PROPs can be defined in a compact and conceptual way as follows. 

\begin{defi}\label{def:propcomp}
A \emph{coloured $\Mm$-PROP}\index{PROP!coloured} (or a PROP enriched in $\Mm$) is a couple $(C,\prp{P})$ where $C$ is a set and $\prp{P}$ 
a symmetric strict monoidal $\Mm$-category $(\prp{P},\boxtimes, \Un)$ (\emph{cf.} \cite{McL98}) 
such that $\Ob{\prp{P}}=\Str{C}$ and the induced monoid structure on objects is $(\Str{C},\ast,[])$.
The set $C$ is called the \emph{set of colours} of $\prp{P}$.\\
A \emph{morphism of $\Mm$-PROP}\index{morphism!of $\Mm$-PROPs} $\morp{F}{(C,\prp{P})}{(D,\prp{R})}$ is given by:
\begin{itemize}
 \item[-] a function $\morp{f}{C}{D}$;
 \item[-] a strict monoidal $\Mm$-functor $\morp{F}{\prp{P}}{\prp{R}}$ that coincides with
$\morp{\Str{f}}{\Str{C}}{\Str{D}}$ at the level of objects.
\end{itemize}
Composition of morphisms is defined in the evident way; 
The category of coloured $\Mm$-PROPs and morphisms between them will be denoted by $\EPROP{\Mm}$; a PROP is \emph{small} if its set of objects is small; the category of small coloured PROPs will be denoted by $\EProp{\Mm}$.
\end{defi}

\begin{rmk}\label{rmk:flex prop}
Even thought definition \ref{def:propcomp} is the classical one, asking that $\Ob{\prp{P}}$ is isomorphic to $\Str{C}$
might sound more natural than asking that $\Ob{\prp{P}}$ and $\Str{C}$ are \emph{equal}; for this, one 
can define a (coloured) $\Mm$-PROP to be a triple $(C,\prp{P},\iota)$ where $C$ is a set, $\prp{P}$ is a symmetric strict
monoidal $\Mm$-category and $\morp{\iota}{C}{\prp{P}}$ is a $\Mm$-functor such that the induced map of monoids
$\morp{\iota}{(\Str{C},\ast,[])}{(\Ob{\prp{P}},\boxtimes,\Un_{\prp{P}})}$ is an isomorphism. A morphism from $(C,\prp{P},\iota)$ to $(D,\prp{R},\iota')$ is just a couple $(f,F)$ as in definition \ref{def:propcomp}
such that $\iota' f=F \iota$.  
This definition is clearly equivalent to the above one, but it will be preferable in \S \ref{sec:fey cat}.
\end{rmk}
\noindent Unravelling definition \ref{def:propcomp} we get the following equivalent one (see \cite{Fr10}): 

\begin{defi}\label{def:prop}
 A \emph{coloured PROP enriched in $\Mm$} is given by the following data:
\begin{itemize}
\item[-] A set $C$;
\item[-] A $C$-coloured bicollection $\prp{P}$;
 \item[-] A \emph{vertical composition law}\index{composition law!vertical (PROPs)}, i.e. a morphism
\begin{equation}\label{vert.comp.prop}
 \lmorp{\circ}{\prp{P}(\mathbf{b};\mathbf{c})\otimes \prp{P}(\mathbf{a};\mathbf{b})}{\prp{P}(\mathbf{a};\mathbf{c})}
\end{equation}
for every $\mathbf{a},\mathbf{b},\mathbf{c}\in \Str{C}$
\item[-] For every $\mathbf{a}\in \Str{C}$ a morphism
\[
 \lmorp{u_{\mathbf{a}}}{\Un}{\prp{P}(c;c)}
\]
called the \emph{identity for $\mathbf{a}$}\index{identity!in PROPs};
\item[-] An \emph{horizontal composition law}\index{composition law!horizontal (PROPs)}, i.e. a morphism
\begin{equation}\label{hor.comp.prop}
 \lmorp{\boxtimes}{\prp{P}(\mathbf{a};\mathbf{b})\otimes\prp{P}(\mathbf{c};\mathbf{d})}{\prp{P}(\mathbf{a}\ast \mathbf{c};\mathbf{b}\ast \mathbf{d})}
\end{equation}
for every $\mathbf{a},\mathbf{b},\mathbf{c},\mathbf{d}\in \Str{C}$.

\item[-] For every $\mathbf{a},\mathbf{b}\in \Str{C}$ such that $\card{a}=n$ and every 
 $\sigma \in \Sigma_n$, two morphisms
\[
\lmorp{\sigma^*}{\prp{P}(\mathbf{a};\mathbf{b})}{\prp{P}(\mathbf{a}\sigma;\mathbf{b})}
\]
\[
\lmorp{\sigma_*}{\prp{P}(\mathbf{b};\mathbf{a})}{\prp{P}(\mathbf{b};\mathbf{a}\sigma^{-1})}
\]
called \emph{input-permutation}\index{input-permutation} and \emph{output-permutation}\index{output-permutation} respectively.
\end{itemize}
These data have to satisfy the following conditions:
\begin{enumerate}
 \item \emph{$\circ$ is associative and has $u_{\mathbf{a}}$ as unit} for $\mathbf{a}\in \Str{C}$. In the case $\Mm=\Set$ this translates into the formulas
\[
 (f\circ f')\circ f''=f\circ (f'\circ f''),\ \ u_{\mathbf{a}}\circ f=f,\ \ f\circ u_\mathbf{a}=f 
\]
whenever these compositions make sense;
\item \emph{$\boxtimes$ is associative with $u_{[]}$ as unit}. When $\Mm=\Set$, this is equivalent to require that
\[
 (f\boxtimes f')\boxtimes f''= f\boxtimes (f'\boxtimes f''),\ \ f\boxtimes u_{[]}=f, \ \ u_{[]}\boxtimes f=f
\]
whenever these compositions make sense;
\item \emph{the vertical composition and the horizontal composition distribute (interchange law)} one over the other. If  $\Mm=\Set$ this constrains read as
\[
 (f\boxtimes g)\circ (f'\boxtimes g')=(f\circ f')\boxtimes (g \circ g') \ \ u_\mathbf{a}\boxtimes u_\mathbf{b}=u_{\mathbf{a}\ast \mathbf{b}} 
\]
whenever the compositions above are defined;
\item The \emph{input-permutations} and the \emph{output-permutations} define a right action and a left action respectively 
\[
 \tau^*\sigma^*=(\sigma \tau)^*\ \ \tau_*\sigma_*=(\tau \sigma)_*\ \ \id^*=\id_*=\id\ \ \tau^*\sigma_*=\sigma_*\tau^*;
\]
for every $n,m\in \N$ and $(\tau,\sigma)\in \Sigma_n\times \Sigma_m$;
\item Permutations and vertical composition commute. In the case $\Mm=\Set$ this translates in the following equations
\[
 \sigma^*(g \circ f)=g\circ \sigma^*(f)\ \ \sigma_*(g \circ f)=\sigma_*(g)\circ f\ \ 
 \sigma^*(g) \circ f=g\circ \sigma_*(f)
\]
whenever these expressions make sense.
\item Permutations and horizontal composition are compatible; that is, given $\mathbf{a},\mathbf{b},\mathbf{c},\mathbf{d}\in \Str{C}$
with $\card{\mathbf{a}}=a,\card{\mathbf{b}}=b,\card{\mathbf{c}}=c,\card{\mathbf{d}}=d$ and $\alpha \in \Sigma_a$, $\beta \in \Sigma_b$,
$\gamma \in \Sigma_c$, $\delta \in \Sigma_d$ the diagram
\[
 \xymatrix{\prp{P}(\mathbf{a};\mathbf{b})\otimes\prp{P}(\mathbf{c};\mathbf{d}) \ar[r]^-{\boxtimes} \ar[d]^-{\alpha^*\beta_* \otimes \gamma^*\delta_*} & \prp{P}(\mathbf{a}\ast \mathbf{b};\mathbf{c}\ast \mathbf{d}) \ar[d]^-{(\alpha \times \beta)^*\otimes (\gamma \times \delta)_*}\\
           \prp{P}(\mathbf{a}\alpha;\mathbf{b}\beta^{-1})\otimes\prp{P}(\mathbf{c}\gamma;\mathbf{d}\delta^{-1}) \ar[r]^-{\boxtimes} & \prp{P}((\mathbf{a}\ast \mathbf{b})(\alpha \times \beta);(\mathbf{c}\ast \mathbf{d})(\gamma \times \delta)^{-1})
          }
\]
commutes.\\
Moreover the diagram
\[
 \xymatrix{\prp{P}(\mathbf{a},\mathbf{b})\otimes\prp{P}(\mathbf{c},\mathbf{d}) \ar[d] \ar[r]^-\boxtimes & \prp{P}(\mathbf{a}\ast \mathbf{c},\mathbf{b}\ast \mathbf{d})\ar[d]^-{\tau_{(a,b)}^*\tau_{(c,d)*}}\\
           \prp{P}(\mathbf{c};\mathbf{d})\otimes\prp{P}(\mathbf{a};\mathbf{b}) \ar[r]_-\boxtimes & \prp{P}(\mathbf{c}\ast \mathbf{a};\mathbf{d}\ast \mathbf{b})}
\]
has to commute; the left vertical map is the symmetric morphism in $\Mm$ and $\tau_{(a,b)}$ is the unique element of $\Sigma_{a+b}$ that restricts to two order preserving functions $\arr{\{1,\dots,a\}}{\{b+1,\dots,a+b\}}$
and $\arr{\{a+1,\dots,a+b\}}{\{1,\dots,b+1\}}$. 
\end{enumerate}
\end{defi}

\noindent Note that the two monoid structures on $\prp{P}([];[])$ defined by the vertical composition
and the horizontal composition $(\circ,u_{[]})$ and $(\square,u_{[]})$ coincide and are commutative by the Heckmann-Hilton argument.

Note that, given a function $\morp{f}{C}{D}$ and a $D$-coloured $\Mm$-PROP $\prp{Q}$ the $C$-coloured $\Mm$-bicollection 
$f^*(\prp{Q})=\{\prp{Q}(f(\str{a}),f(\str{b}))\}_{(\str{a},\str{b})\in \Val{C}}$ inherits a natural $C$-coloured $\Mm$-PROP structure. We leave to the reader to check that a morphism between coloured $\Mm$-PROPs $\arr{(C,\prp{P})}{(D,\prp{Q})}$ is given by a function $\morp{f}{C}{D}$ together with a morphism of $C$-coloured $\Mm$-bicollection $\arr{\prp{P}}{f^*(\prp{Q})}$ preserving compositions, permutations and identities in an obvious way. 

As for operads, there is a functor $\morp{\Clf}{\EProp{\Mm}}{\Set}$ associating to each PROP its set of colours.

For every set $C$ we will denote by $\EPropfc{\Mm}{C}$ the \emph{category of $C$-coloured $\Mm$-PROPs}\index{PROP!$C$-coloured}, i.e. the fiber of $\Clf$ over $C$.

Every $C$-coloured PROP $\prp{P}$ has an underlying $C$-coloured $\Mm$-bicollection (by definition) which is endowed with the structure of a symmetric $\Mm$-bicollection by the permutation morphisms of $\prp{P}$.

In \S\ \ref{sec.opprop} we will describe a $\Val{C}$-coloured operad $\Opprop{C}$ whose algebras in $\Mm$ are $\Mm$-enriched $C$-coloured PROPs. The forgetful functor from $C$-coloured $\Mm$-PROPs to $C$-coloured $\Mm$-bicollection is thus the right adjoint of a free-forgetful adjunction (\emph{cf.} \S \ref{sec:freeprop}):
\begin{equation}\label{eq:freeprop}
 \xymatrix{\Mm^{\Val{C}} \ar@<3pt>[r]^-{F_{\Opprop{C}}} & \ar@<3pt>[l]^-{U_{\Opprop{C}}} \EPropfc{\Mm}{C}} 
\end{equation}

Similarly there is a free-forgetful adjunction between the category of symmetric $C$-coloured $\Mm$-bicollections and $\EPropfc{\Mm}{C}$ (\emph{cf.} \S \ref{sec:freesigmaprop}):
\begin{equation}\label{eq:freesigmaprop}
 \xymatrix{\Mm^{\Sigma\Val{C}} \ar@<3pt>[r]^-{\eta_!} & \ar@<3pt>[l]^-{\eta_*} \EPropfc{\Mm}{C}.} 
\end{equation}

\subsubsection{The endomorphisms PROP}\label{sec:endomorphism PROP}\index{endomorphism!PROP}
As in the case of operads, the most fundamental example of $\Mm$-PROP is the one generated by a collection of objects in a symmetric
monoidal $\Mm$-category $\cat{W}$.

Given a set $C$ and a collection of objects $\mathbf{X}=\{X_c\}_{c\in C}$ indexed by $C$, the strict symmetric 
monoidal $\Mm$-category $\End{\Mm}(\mathbf{X})$ defined in \S \ref{apx:strictification} is a $C$-coloured $\Mm$-PROP
called the \emph{endomorphism $\Mm$-PROP of $\mathbf{X}$}\index{endomorphism!PROP!of a collection}.

In particular there is an $\Ob{\cat{W}}$-coloured $\Mm$-PROP $\EndP{\Mm}(\cat{W})$ (\S \ref{apx:strictification}) associated to $\cat{W}$ that we will call the \emph{endomorphism $\Mm$-PROP of $\cat{W}$}.

Endomorphism PROPs are used to define algebras over (other) PROPs (see \S \ref{sec:algprop}).

\subsection{Constant-free PROPs and augmentation-free PROPs}\label{sec2constantfree}
Operations with empty input or empty output are a source of difficulties in the combinatorics of PROPs.
For certain applications it is enough, and it might be convenient, to restrict ourselves to PROPs with no operations with empty input or empty output.

\begin{defi}
Let $C$ be a set. A $\Mm$-enriched $C$-coloured PROP $\mathbf{P}$ is \emph{constant-free} if $\mathbf{P}([],\mathbf{a})\cong \emptyset$ for every $\mathbf{a}\in \Str{C}$.

A $\Mm$-enriched $C$-coloured PROP $\mathbf{P}$ is \emph{augmentation-free} if $\mathbf{P}(\mathbf{a},[])\cong \emptyset$ for every $\mathbf{a}\in \Str{C}$.
\end{defi}

\noindent There are $\Sigma$-free operads $\Opcfprop{C}$ and $\Opafprop{C}$ whose algebras are constant-free PROPs and augmentation-free PROPs respectively (see \S\ \ref{sec.opprop}). 

\subsection{Categories, Operads and PROPs}
As we said, PROPs can be regarded as a generalisation of categories and operads.
Let $\ECat{\Mm}$ be the category of (small) $\Mm$-enriched categories let and $\EOper{\Mm}$ be the category of $\Mm$-enriched coloured operads (\emph{cf.} \cite{C14}).
There are well know fully-faithful inclusions $\morp{k_!}{\ECat{\Mm}}{\EProp{\Mm}}$ and $\morp{w_!}{\EOper{\Mm}}{\EProp{\Mm}}$
that will be recalled in \S \ref{sec:opfibcatopprop} and \ref{operandcat}.

Informally categories correspond to PROPs where all the operations are freely generated (under horizontal composition) by the operation with one input and one output; similarly, the essential image of $w_!$ is spanned by the PROPs in which the operations are 
freely generated by the operations with precisely one output.

\subsection{The \texorpdfstring{$2$}{2}-category of PROPs}
The category $\EProp{\Mm}$ can be endowed with a $2$-category structure which extends the one on $\EOper{\Mm}$ and $\ECat{\Mm}$.

\begin{defi}
    Given two morphisms of coloured $\Mm$-PROPs $f,g\colon \prp{P} \rightarrow \prp{R}$ a \emph{natural transformation from $f$ to $g$}\index{natural transformation!of PROPs}
    is a monoidal $\Mm$-natural transformation $\alpha$ between $f$ and $g$.  
\end{defi}

\noindent Note that such a $\Mm$-natural transformation is completely
    determined by the subset of its components $\{\alpha_c\mid c\in C\}$ (where we identify $C$ with the subset of elements of length $1$ of $\Str{C}$ )
    by the requirement that it is monoidal.

Natural transformations give to $\EProp{\Mm}$ the structure of a $2$-category.
The functor $w_!$ and $k_!$ (\S \ref{sec:opfibcatopprop}) extends to $2$-functors between $2$-categories in an obvious way.

The next definition is now natural:
\begin{defi}
    A morphism $\morp{f}{\prp{P}}{\prp{R}}$ of $\Mm$-PROPs is
    \begin{itemize}
        \item[-] \emph{fully-faithful}\index{fully-faithful!morphism!of PROPs} when it is fully-faithful as a $\Mm$-functor of symmetric monoidal $\Mm$-categories;
        \item[-] \emph{essentially surjective}\index{essentially surjective morphism! of PROPs} when the $\Mm$-functor $k^*(f)$ is essentially surjective.
    \end{itemize} 
\end{defi}

\noindent The following characterisation of equivalences is proved as in the case for $\Mm$-categories:
\begin{prop}
    A morphism in the $2$-category $\EProp{\Mm}$ is an equivalence if and only if it is fully-faithful and essentially surjective.  
\end{prop}

\subsection{Algebras over PROPs}\label{sec:algprop}
Given an symmetric monoidal $\Mm$-category $\cat{W}$ and a $C$-coloured $\Mm$-PROP $\prp{P}$, the \emph{category of $\prp{P}$-algebras 
}\index{algebra!for a PROP} $\Alg{\prp{P}}{\cat{W}}$, is defined to be the category of morphisms $\EPROP{\Mm}(\prp{P},\EndP{\Mm}(\cat{W}))$.

Since $\cat{W}$ is equivalent to $\EndP{\Mm}(\cat{W})$ ( \S \ref{apx:strictification}), the category $\Alg{\prp{P}}{\cat{W}}$ is equivalent to the category of strong
symmetric monoidal $\Mm$-functors from (the underlying symmetric monoidal $\Mm$-category of) $\prp{P}$ to $\cat{W}$.

Given a collection of objects $\mathbf{X}=\{X_c\}_{c\in C}$ in $\cat{W}$ a $\prp{P}$-algebra structure on it is a morphism of $C$-coloured PROPs
from $\prp{P}$ to $\End{\Mm}(\mathbf{X})$, that is a $\prp{P}$-algebra in $\cat{W}$ that factorises via the inclusion
$\arr{\EndP{\Mm}(\mathbf{X})}{\EndP{\Mm}(\cat{W})}$.

Note that when $\cat{W}$ is a cocomplete symmetric $\Mm$-algebra, an equivalent (and perhaps more familiar) definition of $\prp{P}$-algebra using the tensored structure of $\cat{W}$
is available, in fact to give a $\prp{P}$-algebra structure on $\mathbf{X}$ is the same as giving a morphism
\[
\lmorp{\alpha_{\str{c},\str{d}}}{\prp{P}(\str{c},\str{d})\otimes \bigotimes_{i\in \card{\str{c}}} X_{c_i} }{\bigotimes_{j\in \card{\str{d}}} X_{d_j}}
\]
for every $\ioval{\str{c}}{\str{d}}\in \Val{C}$, subjects to certain associativity, equivariancy and unitality constrains.

\section{Operadic bifibrations}\label{secoperbif}
We continue to fix a bicomplete closed symmetric monoidal category $(\Mm,\otimes,\Un)$.

We want to describe the category $\EProp{\Mm}$ of $\Mm$-enriched coloured PROPs as the total category 
of a bifibration; to this end we will introduce \emph{operadic bifibrations}.
A bifibration is a functor which is both a fibration (or cartesian fibration) and an opfibration 
(or cocartesian fibration); \S\ 2 of \cite{HP14} provides a nice introduction to bifibrations and their correspondence
with pseudo-functors via the Grothendieck construction (see also \cite[Section 3]{C14} for the notation); 
a more classical reference for fibered categories is \cite[Section 8]{Bo08}.

We will use the same notation and conventions used in \cite{C14}; in particular, given a bifibration $\morp{\pi}{\cat{E}}{\cat{B}}$,
we will suppose that a cartesian arrow $\phi_f$ and a cocartesian arrow $\nu_f$ over $f$ have been chosen for every morphism $f$ in $B$.
For every morphism $\morp{g}{A}{B}$ in $E$ such that $\pi(g)=f$ we will denote by $\morp{\trf{g}}{A}{f^*(B)}$
the unique morphism such that $\phi_f \trf{g}=g$ and by $\morp{\trc{g}}{f_!(A)}{B}$ the unique morphism such that
$\trc{g} \nu_f=g$.  

\subsection{The \texorpdfstring{$2$}{2}-functor of algebras}\label{sec:algebra functor} 
We continue to assume $(\Mm,\otimes,\Un)$ is a bicomplete symmetric closed monoidal category.
Let $\EOPER{\Mm}$ be the (large) $2$-category of $\Mm$-operads and let  $\EOper{\Mm}$ be the $2$-category of small coloured $\Mm$-operads.

For every cocomplete symmetric $\Mm$-algebra $\cat{W}$ (\S \ref{sec:Malgebras}) and every $O\in \EOper{\Mm}$ let $\Alg{O}{\cat{W}}$ be the category of algebras of $O$ in $\cat{W}$ (\emph{cf.} \cite{BM05});
it is well known (\emph{cf.} \cite{BM05},\cite{Fr09}) that every morphism of $\Mm$-operads $\morp{f}{O}{P}$ produces adjunction between the categories of algebras
\begin{equation}\label{eq.oper.adj}
\adjpair{f_!}{f^*}{\Alg{O}{\cat{W}}}{\Alg{P}{\cat{W}}.}
\end{equation} 

Recall that every symmetric monoidal $\Mm$-category $\cat{W}$ generates an endomorphism $\Mm$-operad $\Endsm{\Mm}(\cat{W})$: its colours are the elements of $\Endsm{\Mm}(\cat{W})(s_1,\dots,s_n;s)\cong \IHom{\Mm}{s_1\otimes\dots\otimes s_n,s}$ for every $(s_1,\dots,s_n)\in \Str{\Ob{\cat{W}}}$ and $s\in \cat{W}$; $\Endsm{\Mm}(\cat{W})$ can also be described as the $\Mm$-operad underlying the endomorphism PROP $\EndP{\Mm}(\cat{W})$ (\S \ref{sec:endomorphism PROP}).
 
For every operad $\ope{O}\in \EOper{\Mm}$ the category of algebras $\Alg{\ope{O}}{\cat{W}}$ is equivalent to category of morphisms $\EOPER{\Mm}(\ope{O},\Endsm{\Mm}(\cat{W}))$,
thus we can denote the representable $2$-functor \[\morp{\EOPER{\Mm}(-,\Endsm{\Mm}(\Mm))}{\opc{\EOper{\Mm}}}{\CAT}\] by
$\Alg{-}{\cat{W}}$.

Under this identification a morphism of $\Mm$-operad $\morp{f}{\ope{O}}{\ope{P}}$ is sent by $\Alg{-}{\cat{W}}$ to the functor $f^*$ 
described above; a description of $f_!$ is given in appendix \ref{sec: alg as Kan extensions}. 

Since $f^*$ is right adjoint, $\Alg{-}{\cat{W}}$ can be thought as a $2$-functor 
\begin{equation}\label{eq:repalg}
\morp{\Alg{-}{\cat{W}}}{\EOper{\Mm}}{\CATadj}
\end{equation}
where $\CATadj$ is the $2$-category having categories as $0$-cells, adjunctions as $1$-cells (going in the direction of the left adjoint) and natural transformations
between right adjoints as $2$-cells.

\subsection{Operadic families} 
Consider the $2$-category of (un-enriched) operads $\Oper$ and let $\mathcal{C}$ be a category; every (pseudo)functor
\[
 \lmorp{\ofam{F}}{\mathcal{C}}{\Oper}
\]
can be though as a family of operads parameterized by $\mathcal{C}$, thus we will call such objects
\emph{operadic $\mathcal{C}$-families}. A morphism of operadic $\mathcal{C}$-families is just a natural transformation
between them.

Recall that since $\Mm$ is closed, it is a cocomplete symmetric $\Set$-algebra (remark \ref{rmk:setalg}).
Given an operadic family $\morp{\ofam{F}}{\cat{C}}{\Oper}$ the composition
\[
 \cat{C}\overset{\ofam{F}}\longrightarrow \Oper \overset{\Alg{-}{\Mm}}\longrightarrow \Catadj
\]
will be denoted by $\alf{\ofam{F}}{\Mm}$.

By abuse of notation, for every morphism $f$ in $\cat{C}$ we will denote the adjunction $\alf{\ofam{F}}{\Mm}(f)$
by $(f_!,f^*)$, when $\ofam{F}$ is clear from the context.

Given such a family $\ofam{F}$, we can apply the Grothendieck construction to $\alf{\ofam{F}}{\Mm}$ (\emph{cf.} \cite{HP14}) to 
get a bifibration
\[
 \lmorp{\pi_{\ofam{F}}}{\gc \alf{\ofam{F}}{\Mm}}{\mathcal{C}};
\]
the total category $\gc \alf{\ofam{F}}{\Mm}$ will be called the \emph{category of $\ofam{F}$-algebras in $\Mm$}
and denoted by $\ToBi{\ofam{F}}{\Mm}$. 
The objects of $\ToBi{\ofam{F}}{\Mm}$ are couples $(c,X)$ where $x\in \mathcal{C}$ and $X\in \Alg{\Mf(c)}{\Mm}$.
A morphism between $(c,X)$ and $(d,Y)$ is a couple $(f,g)$ where $\morp{f}{c}{d}$ and $\morp{g}{X}{\Mf(f)^*(Y)}$ is a morphism in $\Alg{\Mf(c)}{\Mm}$. 

If $\Mm$ and $\mathcal{C}$ are bicomplete then $\ToBi{\ofam{F}}{\Mm}$ is also bicomplete and $\pi_{\Mf}$ preserves limits and
colimits (\emph{cf.} \cite[Appendix A]{C14}, \cite[Section 2.4]{HP14}).
 
The Grothendieck construction is functorial, so for every morphism of operadic $C$-families $\morp{\alpha}{\ofam{F}}{\mathcal{G}}$ we get a morphism of bifibrations
\[
 \xymatrix{\ToBi{\ofam{F}}{\Mm} \ar[dr]_{\pi_{\ofam{F}}} \ar@/^/@<1pt>[rr]^{\alpha_!}& & \ar@/^/@<1pt>[ll]^{\alpha^*}\ToBi{\mathcal{G}}{\Mm} \ar[dl]^{\pi_{\mathcal{G}}}\\
           & \mathcal{C}. &}
\]

For example let $\morp{\Clf}{\Oper}{\Oper}$ be the functor which associates to every operad $\ope{O}$ the unique discrete
operad with the same set of colours (i.e. the initial $\Cl{\ope{O}}$-coloured operad); clearly $\Alg{\Cl{\ope{O}}}{\Mm}\cong \Mm^{\Cl{\ope{O}}}$.
 
For every operadic $\mathcal{C}$-family $\ofam{F}$ the composition $\Clf\circ\ofam{F}$ defines another operadic
$\mathcal{C}$-family, the \emph{family of colours of $\ofam{F}$}, that will be denoted by $\disf{\ofam{F}}$. The total category $\ToBi{\disf{\ofam{F}}}{\Mm}$ is called the \emph{category of 
$\ofam{F}$-collection in $\Mm$}.

There is a unique morphism from $\disf{\ofam{F}}$ to $\ofam{F}$ which is level-wise the identity on colours, the associated adjunction at the level of total categories
will be denoted by
\begin{equation}\label{freeforgfib}
 \xymatrix{\ToBi{\disf{\ofam{F}}}{\Mm} \ar@<3pt>[r]^-{F_{\ofam{F}}} & \ar@<3pt>[l]^-{U_{\ofam{F}}}  \ToBi{\ofam{F}}{\Mm}};
\end{equation}
when restricted to the fiber over an object $c\in \mathcal{C}$ this gives us the usual free-forgetful adjunction
\begin{equation}\label{freeforgfib2}
 \xymatrix{\Mm^{\Cl{F(c)}}\ar@<3pt>[r]^-{F_{\ofam{F}(c)}} & \ar@<3pt>[l]^-{U_{\ofam{F}(c)}} \Alg{\ofam{F}(c)}{\Mm}}.
\end{equation}
Moreover every morphism of operadic $\mathcal{C}$-families $\morp{\alpha}{\ofam{F}}{\ofam{G}}$ produces a commutative
square of adjunctions
\[
 \xymatrix{\ToBi{\Cl{\ofam{F}}}{\Mm} \ar@<3pt>[d]^-{\ofam{F}} \ar@<3pt>[r]^-{\Cl\alpha_!}& \ar@<3pt>[d]^-{F} \ar@<3pt>[l]^-{\Cl\alpha^*} \ToBi{\Cl{\ofam{G}}}{\Mm}\\
           \ToBi{\ofam{F}}{\Mm} \ar@<3pt>[u]^-{U}\ar@<3pt>[r]^-{\alpha_!}& \ar@<3pt>[u]^-{U} \ar@<3pt>[l]^-{\alpha^*} \ToBi{\ofam{G}}{\Mm}}
\]
 
We will be mainly interested in the case in which $\mathcal{C}=\Set$ and $\ofam{F}=\fOpprop$, the operadic $\Set$-family
 \begin{equation}\label{prop.operadic}
\gfun{\fOpprop}{\Set}{\Oper}{C}{\Opprop{C}}                                            
\end{equation} 
where $\Opprop{C}$ is the operad described in \S\ \ref{sec.opprop}, whose algebras in $\Mm$ are the $C$-coloured $\Mm$-enriched PROPs; see
\S\ \ref{sec.changeofcl} for a description of $\fOpprop$ on the morphisms.

The category of $\fOpprop$-algebras in $\Mm$ is isomorphic to $\EProp{\Mm}$, the category of coloured $\Mm$-enriched PROPs. 

This is the same approach we used in \cite{C14} to describe the category of enriched categories $\ECat{\Mm}$ and that
of enriched operads $\EOper{\Mm}$; In fact, there, $\ECat{\Mm}$ and $\EOper{\Mm}$ are presented as algebras in $\Mm$ for certain
operadic $\Set$-families $\Opcat{}$ and $\Opop{}$.

\subsubsection{Categories, Operads and PROPs}\label{sec:opfibcatopprop} The inclusions of operads displayed in (\ref{inclusion.op}) are natural in $C$, 
that is they define morphisms of operadic $\Set$-families
\[
 \xymatrix{\Opcat{}\ar[r]^-{j} & \Opop{}\ar[r]^-{w} & \fOpprop};
\]
the composite $wj$ will be denoted by $k$.

The categories of $\Opcat{}$-collections and $\Opop{}$-collections are denoted by $\EGraph{\Mm}$ and $\EMGraph{\Mm}$ respectively
in
\cite{C14} (but they should not be confused with the graphs of appendix \ref{ch:graph}). The category of $\fOpprop$-collection in $\Mm$ will be called the category
of \emph{bicollections in $\Mm$} and denoted by $\fVal{\Mm}$.
Via the Grothendieck construction we get a commutative diagram of adjunctions
\begin{equation}\label{eq:diagcatopgr}
 \xymatrix{\EGraph{\Mm} \ar@<3pt>[d]^-{F_{\Opcat{}}} \ar@<3pt>[r]^-{j_!}& \ar@<3pt>[d]^-{F_{\Opop{}}} \ar@<3pt>[l]^-{j^*} \EMGraph{\Mm} \ar@<3pt>[r]^-{w_!}& \ar@<3pt>[d]^-{F_{\Opprop{}}} \ar@<3pt>[l]^-{w^*} \fVal{\Mm}\\
           \ECat{\Mm} \ar@<3pt>[u]^-{U_{\Opcat{}}}\ar@<3pt>[r]^-{j_!}& \ar@<3pt>[u]^-{U_{\Opop{}}} \ar@<3pt>[l]^-{j^*} \EOper{\Mm} \ar@<3pt>[r]^-{w_!}& \ar@<3pt>[u]^-{U_{\Opprop{}}} \ar@<3pt>[l]^-{w^*} \EProp{\Mm}}
\end{equation}
where the right adjoints $j^*$ and $w^*$ are just the obvious forgetful functors.
For every $C\in \Set$ the restriction on the fibers over $C$ of the above diagram is isomorphic to:
\[
 \xymatrix{\Mm^{C\times C} \ar@<3pt>[d]^-{F_{\Opcat{C}}} \ar@<3pt>[r]^-{j_!}& \ar@<3pt>[d]^-{F_{\Opop{C}}} \ar@<3pt>[l]^-{j^*} \Mm^{\Str{C}\times C} \ar@<3pt>[r]^-{w_!}& \ar@<3pt>[d]^-{F_{\Opprop{C}}} \ar@<3pt>[l]^-{w^*} \Mm^{\Val{C}}\\
           \ECatfc{\Mm}{C} \ar@<3pt>[u]^-{U_{\Opcat{C}}}\ar@<3pt>[r]^-{j_!}& \ar@<3pt>[u]^-{U_{\Opop{C}}} \ar@<3pt>[l]^-{j^*} \EOperfc{\Mm}{C} \ar@<3pt>[r]^-{w_!}& \ar@<3pt>[u]^-{U_{\Opprop{C}}} \ar@<3pt>[l]^-{w^*} \EPropfc{\Mm}{C}}
\]
\subsubsection{Constant-free and augmentation-free PROPs} The operads for constant-free PROPs and augmentation-free PROPs defined in \S\ \ref{sec2constantfree} define operadic $\Set$-families $\Opcfprop{}$ and $\Opafprop{}$ as well; the associated bifibered categories will be denoted by $\ECfProp{\Mm}$ and $\EAfProp{\Mm}$.
The obvious fully-faithful inclusions $\arr{\ECfProp{\Mm}}{\EProp{\Mm}}$ and $\arr{\EAfProp{\Mm}}{\EProp{\Mm}}$ are induced by the obvious morphism of operadic $\Set$-families and are therefore left adjoints.

\subsection{Diagram in operadic bifibrations}\label{diagopfam}

We now introduce a kind of Grothendieck construction for operads. This will allow us
to describe diagrams in the total category of an operadic bifibrations as algebras for a certain operad obtained from the starting
operadic family via this construction (\emph{cf.} Propositions \ref{prop.diagop} and \ref{prop.diagop2}). 

\subsubsection{A Grothendieck construction for operads}
Let $\mathcal{C}$ be a small category and let $\ofam{F}$ be an operadic $\mathcal{C}$-family.
We define a new operad $\gco \ofam{F}$ that we call the \emph{Grothendieck construction over $\ofam{F}$}:\\
the set of colours of $\gco \ofam{F}$ is the set of couples $(c,x)$ such that $c\in \mathcal{C}$ and $x\in \ofam{F}(c)$.
For every signature $((c_1,x_1),\dots,(c_n,x_n);(c,x))$ the set of operations \[\gco \ofam{F}((c_1,x_1),\dots,(c_n,x_n);(c,x))\] has for elements the couples $(\{f_i\}^n_{i=1},o)$ where
$\morp{f_i}{c_i}{c}$ is a morphism in $\mathcal{C}$ for every $i\in \fcar{n}$ and $o$ is an operation in
$\ofam{F}(c)(f_1(c_i),\dots,f_n(c_n);c)$.

To define the operadic composition in $\gco \Mf$ it is sufficient to define
the partial compositions (or $\circ_i$-operations \emph{cf.}
\cite[Section 1.3,1.7]{MSS02}).
The partial composition is defined as 
\[(\{f_i\}^n_{i=1},o)\circ_l (\{g_j\}^m_{j=1},p)= (\{h_k\}^{m+n-1}_{k=1}, o \circ_l \ofam{F}(f_l)(p))\] 
for every $l\in \fcar{n}$, where
\[
h_k=\begin{cases}
f_k & \text{if } k<l \\
f_l g_{k-l+1} & \text{if } l\leq k \leq l+m \\
f_{k-m} & \text{if } k>l+m. \\
\end{cases}
\]

Note the similarity of this construction with the Boardman-Vogt tensor product for operads (\emph{cf.} \cite[4.1]{Mo10}, \cite[\S\ 5.1]{MW07}); in
particular the rules for composition ( Figure \ref{figcomp}) resemble the interchange law in the Boardman-Vogt tensor product (see
also Remark \ref{rmkBV} below).
\begin{figure}[!ht]
\begin{center}    
\begin{tikzpicture}
 \node[vertgo] (a) at (0,0) {\footnotesize $o$};
 \grnode{a}{1/{f_1}/{(b_1,z_1)}/{(a,z_1)},2/{f_2}/{(b_2,z_2)}/{(a,z_2)},3/{f_3}/{(b_3,z_3)}/{(a,z_3)}}{3}{0.9}{1.2}
 \grout{a}{(a,w)}{1}
\end{tikzpicture}
\caption{A possible graphical representation for operations in $\gco{\Mf}$; $o$ is an operation in $\Mf(a)$, and
$f_1,f_2,f_3$ are morphisms in $\mathcal{C}$.}
\end{center}
\end{figure}
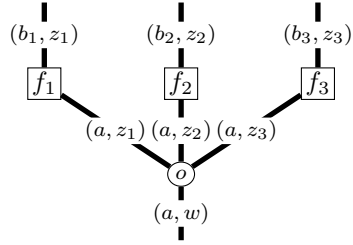

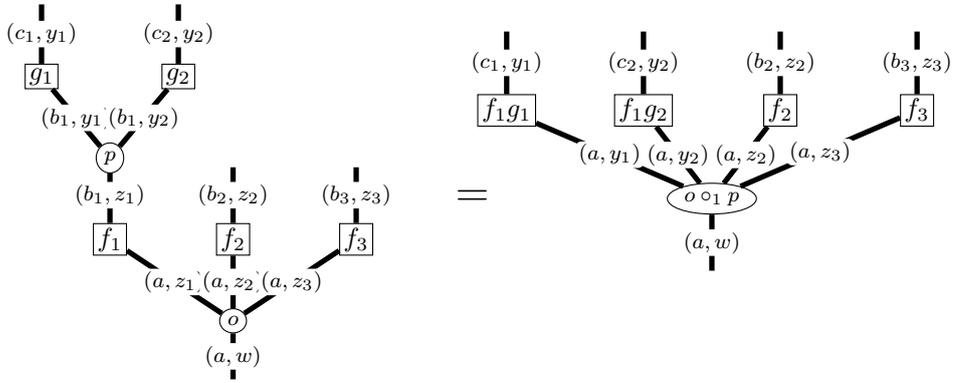
\begin{figure}[!ht]
\begin{center} 
\begin{tikzpicture}[scale=0.9]
 \node[vertgo] (a) at (0,0) {\footnotesize $o$};
 \grnode{a}{1/{f_1}/{(b_1,z_1)}/{(a,z_1)},2/{f_2}/{(b_2,z_2)}/{(a,z_2)},3/{f_3}/{(b_3,z_3)}/{(a,z_3)}}{3}{0.9}{1.2}
 \node[vertgo,fill=white] (b) at (ina1) {\footnotesize $p$};
 \grnode{b}{1/{g_1}/{(c_1,y_1)}/{(b_1,y_1)},2/{g_2}/{(c_2,y_2)}/{(b_1,y_2)}}{2}{1}{1.2} 
 \grout{a}{(a,w)}{1}
 \node[fill=white] (ug) at (3.5,1.8) {\huge $=$};
 \node[vertgo] (c) at (7,1.8) {\footnotesize $o \circ_1 p$};
 \grnode{c}{1/{f_1 g_1}/{(c_1,y_1)}/{(a,y_1)},2/{f_1 g_2}/{(c_2,y_2)}/{(a,y_2)},3/{f_2}/{(b_2,z_2)}/{(a,z_2)},4/{f_3}/{(b_3,z_3)}/{(a,z_3)}}{4}{1}{1.3}
 \grout{c}{(a,w)}{1.2}
\end{tikzpicture}
\caption{A representation of the composition  $(o,\{f_1,f_2,f_3\})\circ_{1}(p,\{g_1,g_2\})$.}\label{figcomp}
\end{center}
\end{figure}

\begin{rmk}
It can be shown that the operad $\gco \ope{F}$ is the lax colimit of $\ope{F}$ in the $2$-category $\Oper$.
More precisely $\gco \ofam{F}$ is isomorphic to $\Wlim{\cat{C}}{w}{\ofam{F}}$, the $\Cat$-weighted colimit of $\ofam{F}$ with respect to the weight 
\[ \gfun{w}{\opc{\mathcal{C}}}{\Cat}{c}{\comcat{c}{\mathcal{C}}}\]
where $\mathcal{C}$ is regarded as a $2$-category with trivial $2$-cells.

We recall that, for a (pseudo)functor $\morp{F}{\cat{C}}{\Cat}$, the (classical) Grothendieck construction $\gc F$ is also isomorphic  to $\Wlim{\cat{C}}{w}{F}$, i.e. to the lax colimit of $F$ (in $\Cat$); this justifies the name we have chosen for our construction.
\end{rmk}
\begin{prop}\label{prop.diagop}
The category of $\gco \ofam{F}$-algebras in $\Mm$ is isomorphic to the category of sections of the bifibration
\[
\lmorp{\pi_{\ofam{F}}}{\ToBi{\ofam{F}}{\Mm}}{\mathcal{C}.}
\]
\end{prop}

\begin{proof}
It is not hard to prove the statement by and explicit description of the algebras of $\gco \ofam{F}$.

A more concise proof can be given observing that the $2$-functor $\AlgF{-}{\Mm}$ is equivalent to the representable
$2$-functor $\OPER (-,\EndP{\Set}(\Mm))$ (\emph{cf.} \S \ref{sec:algebra functor}).
The proof follows from the commutation properties of lax colimits and representable functors; in fact  
\[
\Alg{\ofam{F}}{\Mm}\cong \Oper (\Wcolim{\cat{C}}{w}{\ofam{F}}, \EndP{\Set}(\Mm))\cong \Wlim{\cat{C}}{w}{\Oper (\ofam{F}, \EndP{\Set}(\Mm))} \cong \Wlim{\cat{C}}{w}{\alf{\ofam{F}}{\Mm}}
\]
where the last weighted limit is the lax limit of $\alf{\ofam{F}}{\Mm}$, which is know to be isomorphic to the category of sections of (the Grothendieck construction)
\[\morp{\pi_{\ofam{F}}}{\ToBi{\ofam{F}}{\Mm}}{\mathcal{C}}\].
\end{proof}
\noindent 
In other words the $\gco \ofam{F}$-algebras in $\Mm$ are couples $({A_c}_{c\in \mathcal{C}}, \{\alpha_i\}_{i\in \mathrm{Ar}(\mathcal{C})})$, where
$\mathrm{Ar}(\mathcal{C})$ is the set of morphisms of $\mathcal{C}$, such that:
\begin{enumerate}[i.]
\item $A_c\in \ofam{F}(c)$;
\item $\morp{\alpha_i}{\ofam{F}(i)_!(A_c)}{A_d}$ is a morphism in $\Alg{\ofam{F}(d)}{\Mm}$ for every $\morp{i}{c}{d}$;
\item $\alpha_{gf}=\alpha_g \ofam{F}(g)_!(\alpha_f)$ for every $f,g \in \mathrm{Ar}(\mathcal{C})$ that can be composed.
\end{enumerate}
By adjunction, the last two conditions (and hence the definition of $\{\alpha_i\}_{i\in \mathrm{Ar}(\mathcal{C})}$) could be replaced by:  
\begin{enumerate}[i'.]
\setcounter{enumi}{1}
\item $\morp{\alpha_i}{A_c}{\ofam{F}(i)^*(A_d)}$ is a morphism in $\Alg{\ofam{F}(c)}{\Mm}$ for every $\morp{i}{c}{d}$;
\item $\alpha_{gf}=\ofam{F}(f)^*(\alpha_g) (\alpha_f)$ for every $f,g \in \mathrm{Ar}(\mathcal{C})$ that can be composed.
\end{enumerate}  

The following is an easy consequence of the previous proposition.
\begin{prop}\label{prop.diagop2}
Let $\mathcal{C}$ category, $\ofam{F}$ an operadic $\cat{C}$-family, $I$ a small category and $\morp{D}{I}{\mathcal{C}}$ an $I$-diagram. 
The category of $\gco (\ofam{F} D)$-algebras in $\Mm$ is isomorphic to the category of $I$-diagrams $\mathbb{D}$ in $\ToBi{\ofam{F}}{\Mm}$ 
such that $\pi_{\ofam{F}}\mathbb{D}=D$.
\end{prop}

\subsubsection{Mixed diagrams of algebras}\label{sec.diagalg}

There are certain operadic families for which the Grothendieck construction takes a simpler form, namely when the 
family takes values in the category of full suboperads of a certain fixed operad.

Suppose $\mathcal{O}$ is an $D$-coloured operad and let 
$\mathcal{C}$ be a small category with set of objects $C$ (for us it will be sufficient to consider the case in which $\mathcal{C}$ is a poset).
Denote by $\mathcal{P}(D)$ the poset of subsets of $D$ ordered by inclusion; suppose a functor $\morp{f}{\mathcal{C}}{\mathcal{P}(D)}$
is given, we will call such a functor a \emph{hierarchy functor} for $\mathcal{C}$-diagrams of $\mathcal{O}$-algebras. 

For every $S\subseteq D$ let $\resop{\mathcal{O}}{S}$ be the full suboperad of $\mathcal{O}$ spanned by the colours in $S$,
and let $\morp{\xi^S}{\resop{\mathcal{O}}{S}}{\mathcal{O}}$ be the morphism of operads witnessing this inclusion.

Given a hierarchy functor $f$ we can define an operadic $\mathcal{C}$-family
\[
 \gfun{\hof{\mathcal{O}}{f}}{\mathcal{C}}{\Oper}{c}{\resop{\mathcal{O}}{f(c)}.}
\]
 
The operad $\gco{\hof{\mathcal{O}}{f}}$ has for colours all the couples
$(o,c)\in D\times C$ such that $o\in f(i)$.

For every $(o,c),(o_1,c_1),\dots, (o_n,c_n)\in \Cl{\gco\hof{\mathcal{O}}{f}}$
\[
 (\gco\hof{\mathcal{O}}{f})((o_1,c_1),\dots,(o_n,c_n); (o,c))=\mathcal{O}(o_1,\dots,o_n; o)\times \prod_{i=1}^{n} \mathcal{C}(c_i,c);
\]
compositions, symmetries and identities are defined in the evident way.

The algebras of $\gco{\hof{\mathcal{O}}{f}}$ in $\Mm$ are $\mathcal{C}$-diagrams $X$ in $\Alg{\mathcal{O}}{\Mm}$ such that 
$X(c)\cong \xi^{f(c)}_!(Y)$ for some $Y\in \Alg{\resop{\mathcal{O}}{f(c)}}{\Mm}$ for every $c\in C$.
Informally thus, we can think of them as diagram of $\ope{O}$-algebras in which certain nodes are particularly simple algebras
(i.e. algebras for a suboperad of $\ope{O}$).    

There is a natural morphism $\morp{l}{\gco{\hof{\mathcal{O}}{f}}}{\mathcal{O}}$ and the induced functor
\[
 \morp{l_!}{\Alg{\gco{\hof{\mathcal{O}}{f}}}{\Mm}}{\Alg{\mathcal{O}}{\Mm}}
\]
sends each diagram to its colimit.

\begin{rmk}\label{rmkBV}
When $f$ is the constant functor with constant value $D$ the operad $\gco \hof{\mathcal{O}}{f}$ will be denoted $\mathcal{O}\BVt \mathcal{C}$, since
it is isomorphic to the Boardman-Vogt tensor product of $\mathcal{O}$ with $\mathcal{C}$ (\emph{cf.} \cite[\S\ 4.1]{Mo10}).
 For every hierarchy functor $f$ the operad $\gco \hof{\mathcal{O}}{f}$ is a full suboperad of $\mathcal{O}\BVt \mathcal{C}$. 
\end{rmk}
\begin{rmk}\label{rmk.bvposet}
If $\mathcal{C}$ is a poset the description of $\mathcal{O}\BVt \mathcal{C}$ is even simpler, in fact, 
for every $o,o_1,\dots, o_n\in \Cl{\mathcal{O}}$ and $a,a_1,\dots,a_n\in \Ob{C}$:
\[
 (\mathcal{O}\BVt \mathcal{C})((o_1,a_1),\dots,(o_n,a_n); (o,a))=
\begin{cases}
\mathcal{O}(o_1,\dots,o_n; o) & \text{ if } a_i\leq a \ \forall i\in \fcar{n}\\ 
\emptyset & \text{otherwise} 
\end{cases}
\]
\end{rmk}

\subsubsection{Example: Push-out diagrams}
In Sections \ref{sec.podiag} and \ref{sec.poalongop} we will investigate push-outs of PROPs.
The case of interest for us will then be $\mathcal{C}=\dPo$, the poset with three elements $O,A,B$ such that $O<B$ and $O<A$
\begin{equation}\label{pocat}
 \xymatrix{O \ar[d]_{p_0} \ar[r]^{p_1} & B\\
           A & }
\end{equation}
we thus give a little bit more details for this case.
To give $\morp{f}{\dPo}{\mathcal{P}(D)}$ is equivalent to give a commutative diagram of full inclusions of operads:
\[
 \diagc{\mathcal{S}}{\ope{F}}{\mathcal{H}}{\mathcal{O}.}{h}{f}{g}{k}{}
\]
A algebra in $\Mm$ for $\fbv{\mathcal{O}}{\dPo}{f}$ is a $\dPo$-diagram in $\Alg{\mathcal{O}}{\Mm}$ of the form
\[
 \xymatrix{k_!f_!(S) \ar[d]_{g_!(m)} \ar[r]^{k_!(l)} & k_!(F)\\
           g_!(H) & .}
\]
where $S\in \Alg{\mathcal{S}}{\Mm}$,$F\in \Alg{\ope{F}}{\Mm}$, $H\in \Alg{\mathcal{H}}{\Mm}$.

\begin{rmk}
For simplicity we defined operadic families with values in (un-enriched) operads; however, given a (closed) symmetric monoidal category $\Mm$ it is possible to define operadic families with values in $\Mm$-operads and an operadic Grothendieck construction for them in a similar way; this would permit to describe more general bifibrations and diagrams in their total categories.    
\end{rmk}

\section{Homotopy theory of the fibers}\label{sechomfib}
Until this point there was no homotopy theory involved. From now one we suppose that
the monoidal category $(\Mm,\otimes,\Un)$ we are enriching in is given with a cofibrantly generated monoidal model structure with $I$ (resp. $J$) as set of generating (trivial)
cofibrations (\emph{cf.} \cite{Ho99}).

Given a (coloured) operad $\mathcal{O}$, under appropriate hypotheses, the model structure $\Mm$ can be transferred to the category of algebras $\Alg{\mathcal{O}}{\Mm}$.
Such an operad $\mathcal{O}$ was called \emph{admissible in $\Mm$} in \cite{C14}. Admissibility of (enriched) operads was investigated under various constrains on $\mathcal{O}$
and $\Mm$ in several papers, see for example \cite{BM03}, \cite{BM05}, \cite{Ha10}, \cite{Mu11}, \cite{BB13}, \cite{PS14} and \cite{WY15}. 
\subsection{Admissible operadic families}
We begin by extending the definition of admissible operad to operadic families in a natural way.
\begin{defi}
An operadic $\mathcal{C}$-family $\morp{\ofam{F}}{\mathcal{C}}{\Oper}$ is \emph{admissible} in $\Mm$ if
for every $c\in \mathcal{C}$ the operad $\ofam{F}(c)$ is admissible in $\Mm$ ( in the sense of \cite{C14}). 
\end{defi}

\noindent We are going to unravel a little bit the previous definition.
Recall that for every set $S\in \Set$ the product category $\Mm^S$ has an induced model structure, where 
cofibrations, fibrations and weak equivalences are defined level-wise (\cite{Ho99}). This model structure is cofibrantly generated;
if $\morp{\iota_s}{\Mm}{\Mm^{S}}$ is the left adjoint of the projection on the $s$-component, a set of generating (trivial)
cofibrations is given by \[I_S=\{\iota_s(i) \mid s\in S,\ i\in I \}\ (\text{resp. }
J_S=\{\iota_s(j) \mid s\in S,\ j\in J \}\ ).\]

If the $\mathcal{C}$-family $\morp{\ofam{F}}{\mathcal{C}}{\Oper}$ is \emph{admissible} in $\Mm$, for every $c\in \mathcal{C}$ the product model structure on $\Mm^{\Cl{\ofam{F}(c)}}$ can be 
transferred to $\Alg{\ofam{F}(c)}{\Mm}$ along the free-forgetful adjunction
\begin{equation}
\adjpair{F_{\ofam{F}(c)}}{U_{\ofam{F}(c)}.}{\Mm^{\Clf\ofam{F}(c)}}{\Alg{\ofam{F}(c)}{\Mm}}
\end{equation}
The transferred model structure on $\Alg{\ofam{F}(c)}{\Mm}$ is often called the \emph{projective model structure};
fibrations and weak-equivalences are preserved and reflected by $U_{\ofam{F}(c)}$ while a set of generating (trivial)
cofibrations is given by $F_{\ofam{F}(c)}(I_{\Clf\ofam{F}(c)})$ (resp. $F_{\ofam{F}(c)}(J_{\Clf\ofam{F}(c)})$).

For every $\morp{\alpha}{c}{d}$ in $\mathcal{C}$ the adjunction
\[
\indadjpair{\ofam{F}(\alpha)}{\Alg{\ofam{F}(c)}{\Mm}}{\Alg{\ofam{F}(d)}{\Mm}}
\]
is automatically a Quillen adjunction between the projective model structures.
 
In other words if $\ofam{F}$ is admissible, the functor $\fAlg \ofam{F}$ can be extended to a functor with values
in $\ModCat$, the 2-category of Model Categories and Quillen adjunctions between them
\begin{equation}
\lmorp{\alf{\ofam{F}}{\Mm}}{\mathcal{C}}{\ModCat.}
\end{equation}

\subsubsection{Local fibrations and trivial fibrations}
Suppose now that $\mathcal{C}$ is bicomplete and $\ofam{F}$ is an operadic $\mathcal{C}$-family admissible in $\Mm$.
In this section we are going to characterize the level-wise (trivial) fibrations in the total category $\ToBi{\ofam{F}}{\Mm}$ 
via the right lifting property with respect to a certain set of maps.

\begin{defi}\label{def:localstuff}
A morphism $\morp{f}{X}{Y}$ in $\ToBi{\ofam{F}}{\Mm}$ is a \emph{local fibration (weak equivalence)} if
for $\morp{\trf{f}}{X}{f^*(Y)}$ is a fibration (weak equivalence) in the projective model structure on
$\Alg{\Mm}{\ofam{F}(\pi_{\ofam{F}}X)}$. 
\end{defi}

\begin{rmk}
A morphism $\morp{f}{X}{Y}$ in $\ToBi{\ofam{F}}{\Mm}$ is a local fibration (weak equivalence) if and only if
$U_{\ofam{F}}(f)$ is a local fibration (weak equivalence) in $\ToBi{\disf{\ofam{F}}}{\Mm}$.
\end{rmk}

\begin{defi}\label{def.weakin}
Let $\ofam{F}$ be an operadic $\mathcal{C}$-family. A \emph{weakly initial set of colours for $\ofam{F}$} is a set $G$ of colours of $\gco \ofam{F}$ such that for every colour $(c,x)\in \gco \ofam{F}$ there exist $(g,y)\in G$ and a morphism $\morp{f}{g}{c}$ such that $\ofam{F}(f)(y)=x$. 
\end{defi}

\begin{prop}\label{rlp.int}
Let $\mathcal{C}$ be a bicomplete category and let $\ofam{F}$ be an operadic $\mathcal{C}$-family with a weakly initial set
of colours $G$ and admissible in $\Mm$. A morphism $\morp{f}{(x,X)}{(y,Y)}$ in $\ToBi{\ofam{F}}{\Mm}$ is a local (trivial) fibration
if and only if it has the right lifting property with respect to:
\[
 \mathcal{J}_{loc}=\{ (c,F_{\ofam{F}(c)}\iota_s(j)) \mid (c,s)\in G, j\in J\}
\]
(resp. 
\[
 \mathcal{I}_{loc}=\{ (c,F_{\ofam{F}(c)}\iota_s(i)) \mid (c,s)\in G, i\in I\}\text{ )}
\]  
\end{prop}

\begin{proof}
Note that, given $\morp{j}{A}{B}$ in $J$, $(c,s)\in G$, a commutative diagram 
\begin{equation}\label{rlp.one}
 \diagc{(c,F_{\ofam{F}(c)}\iota_s(A))}{(x,X)}{(c,F_{\ofam{F}(c)}\iota_s(B))}{(y,Y)}{(c,F_{\ofam{F}(c)}\iota_s(j))}{a}{f}{b}{}
\end{equation}
admits a diagonal filler if and only if the diagram
\begin{equation}\label{rlp.two}
 \diagc{a_!F_{\ofam{F}(c)}\iota_s(A)}{X}{b_!F_{\ofam{F}(c)}\iota_s(B)}{f^*(Y)}{a_!F_{\ofam{F}(c)}\iota_s(i)}{\trc{a}}{\trf{f}}{\trc{b}}{}
\end{equation}
has a diagonal filler.
If $f$ is a local fibration then diagram (\ref{rlp.two}) has a diagonal filler since $a_!$, $F_{\ofam{F}(c)}$ and $\iota_s$ are
left Quillen functors.

Conversely suppose that $f$ has the right lifting property with respect to every map in $\mathcal{J}_{loc}$.
To prove that $f$ is a local fibration it is sufficient to show that $\morp{f_t}{X(t)}{f^*(Y)(t)}$ is a fibration for every $t\in \Cl{\ofam{F}(x)}$;
in other words we have to prove that for every $t\in \Cl{\ofam{F}(x)}$ and every $j\in J$, every diagram of the form
\begin{equation}\label{rlp.three}
 \diagc{A}{X(t)}{B}{f^*(Y)(t)}{i}{a}{f_t}{b}{}
\end{equation} 
admits a diagonal filler.

Pick an element $(c,s)\in G$ and a morphism $\morp{a}{c}{x}$ such that $\ofam{F}(a)(s)=t$. The diagram (\ref{rlp.three}) has a diagonal filler
if and only if (\ref{rlp.two}) admits a diagonal filler, which is true by the hypothesis we made on $f$. Thus the statement is proved.
The proof for local trivial fibrations is similar.
\end{proof}

\subsection{The integral model structure}
There is a way to get a model structure on the total category of an operadic bifibration directly from the model structures on 
the fibers.

Suppose that $\mathcal{C}$ is bicomplete and $\ofam{F}$ is an operadic $\mathcal{C}$- parametrization admissible in $\Mm$.
We can endow $\mathcal{C}$ with the \emph{trivial model structure} where the classes of cofibrations and fibrations coincides with the whole class of morphisms and weak equivalences are the isomorphisms.

In this way the functor $\alf{\ofam{F}}{\Mm}$ becomes a relative proper functor in the sense of \cite[def. 3.3, 3.6]{HP14}; as a particular case of \cite[Theorem 3.9]{HP14} we get the following:
result
\begin{prop}\label{prop.int.mod}(\cite{HP14}) Let $\ofam{F}$ be an operadic $\mathcal{C}$-parametrization. The category $\ToBi{\ofam{F}}{\Mm}$ admits a model structure such that a morphism $\morp{f}{(c,X)}{(d,Y)}$ in $\ToBi{\ofam{F}}{\Mm}$ is
\begin{itemize}
\item[-] a \textit{weak equivalence} if and only if it is a local weak equivalence and $\pi_{\ofam{F}}(f)$ is an isomorphism.
\item[-] a \textit{fibration} if and only if $\morp{\trf{f}}{X}{f^*(Y)}$ is
a local fibration;
\item[-] a \textit{fibration} if and only if $\morp{\trc{f}}{f_!(X)}{Y}$ is
a cofibration in $\Alg{\ofam{F}(d)}{\Mm}$.
\end{itemize}
\end{prop}
The model structure on $\ToBi{\ofam{F}}{\Mm}$ defined in Proposition \ref{prop.int.mod} will be called the
\emph{integral model structure}. 

For our purposes the integral model structure on $\ToBi{\Opprop{}}{\Mm}$ has too many weak equivalences (and cofibrations);
in fact, as we are going to explain, in the Dwyer-Kan model structure the weak equivalences must be the local weak equivalences
which are essentially surjective in a homotopical sense.

The identity functor is neither a left or right Quillen functor between the integral and the Dwyer-Kan model structure on 
$\EProp{\Mm}$ (and the same holds for $\ECat{\Mm}$ and $\EOper{\Mm}$).

\section{The Dwyer-Kan model structure on enriched PROPs}\label{sec.main}
We continue to fix a cofibrantly generated monoidal model category $(\Mm,\otimes,\Un)$.

We now return to our main goal which is to prove the existence of a sensible model structure on $\EProp{\Mm}=\ToBi{\Opprop{}}{\Mm}$ (\emph{cf.} (\ref{prop.operadic})), the category
of coloured $\Mm$-enriched PROPs, at least under certain conditions on $\Mm$. 
 
First of all we want to ensure that $\Opprop{}$ is admissible in $\Mm$.
\begin{defi}
A cofibrantly generated monoidal model category $\Mm$ \emph{admits transfer for PROPs} if
$\Opprop{}$ is admissible (in $\Mm$).
\end{defi}

\noindent $\Opprop{}$ is certainly admissible in $\Mm$ if every operad is admissible in $\Mm$. Sufficient conditions for this to happen are
given in \cite{BM03} (see also \cite{BM05}), \cite{Ha10} and more recently in \cite{PS14} and \cite{WY15}. In particular, we would like to cite the following result:
\begin{thm}(\textit{Pavlov, Scholbach \cite[Theorem 9.2.11]{PS14}})
If $\Mm$ is a combinatorial, pretty small, symmetric $i$-monoidal model category the category every coloured operad is admissible
in $\Mm$. 
\end{thm}
\noindent
Monoidal model categories of interest satisfying the above hypothesis are:
\begin{itemize}
 \item[-] Simplicial Sets with the Kan-Quillen model structure;
 \item[-] (Unbounded) chain complexes over a ring $R\supseteq \mathbb{Q}$, with the projective model structure;
 \item[-] Symmetric spectra (in Simplicial sets) with the positive stable model structure 
(and more generally symmetric spectra in other monoidal model category, \emph{cf.} \cite{PS14},\cite{PS142});
\end{itemize}

We refer the reader to \S\ 7 of \emph{loc.cit.} for other examples and the definitions of symmetric $i$-monoidal and pretty small.
     
The operadic $\Set$-family $\Opprop{}$ is finitary (def. \ref{def.finfam}), this can be checked by and explicit
description of filtered colimits in $\Oper$, see for example \cite[Appendix A]{C14}.

The family $\Opprop{}$ also has a weakly initial set of colours (def. \ref{def.weakin}) that we are now going to 
describe; 
for every $(m,n)\in \N\times \N$ let $b^m_n$ be the $\fcar{n+m}$-valence $\ioval{m+1,\dots,n+m}{1,\dots,m}$, a weakly initial set of colours
for $\Opprop{}$ is
\[
 \mathcal{B}=\{(\fcar{n+m}, b^m_n) \mid (n,m)\in \N\times \N \}.
\]
For every $(m,n)\in \N\times \N$ we are going to denote by $\morp{B^m_n}{\Mm}{\EPropfc{\Mm}{\fcar{n+m}}}$ 
the composition $F_{\Opprop{C}}\iota_{b^m_n}$.

Suppose that $\Opprop{}$ is admissible in $\Mm$; according to Proposition \ref{rlp.int} the
local (trivial) fibrations in $\EProp{\Mm}$ are characterized by the right lifting property with respect to:
\[
 \mathcal{J}_{loc}=\{ B^m_n(j) \mid (m,n)\in \N\times \N, j\in J\}
\]
(resp. 
\[
 \mathcal{I}_{loc}=\{ B^m_n(i) \mid (m,n)\in \N\times \N, i\in I\}\text{ ).}
\] 
 
\subsection{The Dwyer-Kan model structure}
Let $(\Mm,\otimes, \Un)$ be a monoidal model category and let $\HoC{\Mm}$ be its homotopy category (\emph{cf.} \cite{Ho99}). The functor of ``path components''
\[
 \gfun{\pi_0}{(\Mm,\otimes,\Un)}{(\Set,\times, \ast)}{X}{\HoC{\Mm}(\Un,X)}
\]
is symmetric monoidal and therefore induces functors
\[
 \lmorp{\pi_0}{\ECat{\Mm}}{\ECat{\Set},}
\]
\[
 \lmorp{\pi_0}{\EProp{\Mm}}{\EProp{\Set}.}
\]

We recall that a morphism $\morp{f}{X}{Y}$ in $\EProp{\Mm}$ is a \emph{local weak equivalence} if $\morp{\trf{f}}{X}{f^*(Y)}$ is
a weak equivalence $\EPropfc{\Mm}{\Cl{X}}$ (def. \ref{def:localstuff}).

\begin{defi}\label{defi.DKequi}
A morphism $\morp{f}{X}{Y}$ in in $\EProp{\Mm}$ is a \emph{Dwyer-Kan weak equivalence} if it is a local weak equivalence
and the functor $\pi_0(k^*(f))$ is essentially surjective (that is $k^*(f)$ is a Dwyer-Kan weak equivalence in $\ECat{\Mm}$
in the sense of \cite{Mu14}). 
\end{defi}

\begin{defi}
A \emph{Dwyer-Kan model structure} for $\EProp{\Mm}$ is a model structure on $\EProp{\Mm}$ in which 
the class of weak equivalences coincides with the one of Dwyer-Kan weak equivalences and the class of trivial fibrations
coincides with the class of the local trivial fibrations surjective on objects. 
\end{defi}

\noindent Note that if a Dwyer-Kan model structure for $\EProp{\Mm}$ exists, it is unique.

Our main theorem about the existence of the Dwyer-Kan model structure is the following:
\begin{thm}\label{main.thm}
Suppose that $(\Mm,\otimes,\Un)$ is a combinatorial monoidal model structure such that:
\begin{itemize}
 \item[-] it satisfies the monoid axiom;
 \item[-] the class of weak equivalences is closed under transfinite composition;
 \item[-] it admits transfer for PROPs.
\end{itemize}
then the Dwyer-Kan model structure on $\EProp{\Mm}$ exists and it is combinatorial.  
\end{thm}
\noindent
We recall that a monoidal model category satisfies the \emph{monoid axiom} if the closure of the class of trivial cofibrations
under push-outs, transfinite compositions and tensor product with an arbitrary object is contained in the class of weak equivalences.
The monoid axiom was first introduced by Schwede and Shipley in \cite{SS00} to study the transferred model structure
on monoids; 
it ensure that every non-symmetric operad is admissible in $\Mm$, see \cite{Mu11} and \cite{BB13}.

We have the following immediate Corollary:
\begin{cor}
Let $(\Mm,\otimes,\Un)$ be a combinatorial monoidal model category which is pretty small and $i$-symmetric monoidal. Then 
The Dwyer-Kan model structure on $\EProp{\Mm}$ exists and it is combinatorial. 
\end{cor}
\begin{proof}
Under these hypotheses $\Mm$ satisfies all the hypotheses of \ref{main.thm} (\emph{cf.} \cite{PS14}). 
\end{proof}
\noindent The proof of Theorem \ref{main.thm} is postponed to \S\ \ref{proof.main.thm}; before that, we need to describe the set of generating (trivial) cofibrations.

On the category $\ECfProp{\Mm}$ (resp. $\EAfProp{\Mm}$) of constant-free (augmentation-free) PROPs there is an equivalent notion
of Dwyer-Kan model structure, that is a model structure in which the weak equivalences are the Dwyer-Kan weak equivalences and 
the trivial fibrations are locally trivial fibrations surjective on colours. 
The proof of Theorem \ref{main.thm} can be adapted to prove the following analogous statement; we leave the details to the reader.

\begin{thm}\label{main.thm2}
Suppose that $(\Mm,\otimes,\Un)$ is a combinatorial monoidal model structure such that:
\begin{itemize}
 \item[-] it satisfies the monoid axiom;
 \item[-] the class of weak equivalences is closed under transfinite composition;
 \item[-] it admits transfer for constant-free (augmentation-free) PROPs .
\end{itemize}
The Dwyer-Kan model structure on $\ECfProp{\Mm}$ ($\EAfProp{\Mm}$) exists and it is combinatorial.  
\end{thm}

\subsection{Generation (trivial) cofibrations}
Let $\mathbf{0}$ be the initial object in $\EProp{\Mm}$ and let $\mathbf{1}$ be the initial object in $\EPropfc{\Mm}{\ast}$ (the category
of one-coloured PROPs) and let
$\morp{r}{\mathbf{0}}{\mathbf{1}}$ be the unique morphism between them in $\EProp{\Mm}$.
For every $X\in \EProp{\Mm}$ the Hom-set $\EProp{\Mm}(\mathbf{1},X)$ is isomorphic to $\Cl{X}$.

A morphism $\morp{X}{Y}$ in $\EProp{\Mm}$ has the right lifting property with respect to $r$ if and only if it is surjective on colours.

It follows that, if we assume that $\Opprop{}$ is admissible in $\Mm$, local trivial fibrations surjective on colours are characterized
by the right lifting property with respect to
\begin{equation}\label{DKgenco}
 \mathcal{I}_{DK}=\{r\}\cup \mathcal{I}_{loc},
\end{equation}
where $\mathcal{J}_{loc}$ is the set of morphisms defined in Proposition \ref{rlp.int}.  

The description of the set of generating trivial cofibrations is more subtle and is based on the work
of Muro \cite{Mu14} and the notion of interval introduced there that we are now going to recall.    

\subsubsection{Weak intervals}\label{sec.weakint}
We assume that $(\Mm,\otimes,\Un)$ is a monoidal model category in which the operadic family $\Opcat{}$ is admissible.

Recall that two objects in a $\Mm$-category $\cat{C}$ are \emph{homotopy equivalent} if they are isomorphic in $\pi_0(\cat{C})$.
\begin{defi}
Given a $\Mm$-enriched PROP $X\in \EProp{\Mm}$ two colours $x,y\in X$ are \emph{homotopy equivalent}
if they are isomorphic as objects of $\pi_0(k^*(X))$, i.e. if they are homotopy equivalent in the underlying $\Mm$-category of $X$. 
\end{defi}

\begin{defi}\label{def.weakint}
A \emph{weak interval} in $\ECat{\Mm}$ is a $\Mm$-enriched category with two objects $\mathcal{I}\in \ECatfc{\Mm}{\{0,1\}}$
such that $0$ and $1$ are homotopy equivalent. 

A set of weak intervals $\mathfrak{G}$ is said to be generating if for every $X\in\ECat{\Mm}$ and every pair of homotopy
equivalent objects $x,y\in X$ there exist a weak interval $G\in \mathfrak{G}$ and a morphism $\morp{f}{G}{X}$ such that 
$f(0)=x$ and $f(1)=y$.
\end{defi}

\noindent We remark that weak intervals are just called intervals in \cite{Mu14}; we use a different terminology in order to distinguish them
from the intervals of Berger and Moerdijk (\emph{cf.} \S\ \ref{sec.BMint}).

\begin{prop}(\cite{Mu14}) Suppose that $\Mm$ is a combinatorial monoidal model category satisfying the monoid axiom; there exists a generating set of weak intervals $\mathcal{G}$ for $\ECat{\Mm}$. 
\end{prop}

Consider the inclusion $\morp{i_0}{\{0\}}{\{0,1\}}$; for every weak interval $\mathbb{K}$ let
$\morp{i^{\mathbb{K}_0}}{i^*_0(\mathbb{K})}{\mathbb{K}}$ the cartesian arrow over $i_0$ with target $\mathbb{K}$;
Let $\morp{c_{\mathbb{K}}}{\widetilde{i^*_0(\mathbb{K})}}{i^*_0(\mathbb{K})}$ be a cofibrant replacement
in $\EPropfc{\Mm}{\{0\}}$; factor the morphism $\trc{(i^{\mathbb{K}}_0 c_{\mathbb{K}})}$ in $\EPropfc{\Mm}{\{0,1\}}$ into a cofibration $\morp{l_{\mathbb{K}}}{i_{0!}\widetilde{i^*_0(\mathbb{K})}}{\widetilde{\mathbb{K}}}$ followed by a weak equivalence $\morp{w}{\widetilde{\mathbb{K}}}{\mathbb{K}}$; 
We will denote by $\theta_{\mathbb{K}}$ the morphism $\morp{(i_0,l_{\mathbb{K}})}{\widetilde{i^*_0(\mathbb{K})}}{\widetilde{\mathbb{K}}}$.

\begin{equation}
\xymatrix{\widetilde{i^*_0(\mathbb{K})} \ar[r]^{\theta_{\mathbb{K}}}\ar[d]_{\sim} &
\widetilde{\mathbb{K}}\ar[d]^{\sim} \\ i^*_0(\mathbb{K}) \ar[r]^{i^{\mathbb{K}}_0} & \mathbb{K};}
\end{equation}
of course the definition of $\theta_{\mathbb{K}}$ depends on the choice of cofibrant replacement
and factorization that we made; we will assume that one such a $\theta_{\mathbb{K}}$ has been chosen for
every weak interval $\mathbb{K}$.

Given a generating set of weak intervals $\mathfrak{G}$ we will denote by $\mathcal{J}^{\Opcat{}}_{\mathfrak{G}}$ the set of morphisms
\begin{equation}\label{set.genes}
 \mathcal{J}^{\Opcat{}}_{\mathfrak{G}}=\{\theta_G \mid G\in \mathfrak{G} \};
\end{equation}
this set is part of the set of generating trivial cofibrations for the Dwyer-Kan model structure on $\ECat{\Mm}$ 
introduced in \cite{Mu14}.

Let $\mathcal{J}_{\mathfrak{G}}=k_!\mathcal{J}^{\Opcat{}}_{\mathfrak{G}}$ be the incarnation of 
$\mathcal{J}^{\Opcat{}}_{\mathfrak{G}}$ in $\EProp{\Mm}$.

We define a set of morphisms
\begin{equation}\label{DKgentco} 
\mathcal{J}_{DK}=\mathcal{J}_{loc}\cup \mathcal{J}_{\mathfrak{G}}
\end{equation}
that will serve as a set of generating trivial cofibrations for the Dwyer-Kan model structure.

\subsubsection{Berger and Moerdijk's \texorpdfstring{$\Mm$}{V}-intervals}\label{sec.BMint}
We continue to assume that $\Opcat{}$ is admissible in $\Mm$. The notion of weak interval is strictly related to the notion of $\Mm$-interval defined in \cite{BM12}.

We will denote by $\mathbb{I}\in \ECat{\Mm}{\{0,1\}}$ the $\Mm$-enriched category representing isomorphisms, that is
$\mathbb{I}(0,0)=\mathbb{I}(1,0)=\mathbb{I}(0,1)=\mathbb{I}(1,1)=\Un$; the $\Mm$-monoid $\mathbf{1}=\mathbb{I}(0,0)\in \ECatfc{\Mm}{\{0\}}$ is the $\Mm$-enriched category representing objects.

\begin{defi}(\cite{BM12})
A \emph{$\Mm$-interval} is a cofibrant object in $\ECatfc{\Mm}{\{0,1\}}$ (considered with the projective model structure) which
is weakly equivalent to $\mathbb{I}$.
\end{defi}

\begin{rmk}
Every $\Mm$-interval is a weak interval. Furthermore since $\mathbb{H}$ is cofibrant $\ECat{\Mm}{\{0,1\}}$ and
$i^*_0(\mathbb{H})$ is weakly equivalent to $\mathbf{1}$ in $\ECatfc{\Mm}{\{0\}}$,
 the morphism $\theta_{\mathbb{H}}$ defined in \S~\ref{sec.weakint}
can be chosen to be the unique morphism
\[
 \morp{\theta_{\mathbb{H}}}{\mathbf{1}}{\mathbb{H}}
\]
such that $\theta_{\mathbb{H}}(0)=0$.
\end{rmk}
\noindent It is clear that every $\Mm$-interval is a weak interval.

\begin{defi}
A set of $\Mm$-interval $\mathfrak{M}$ is \emph{generating} if every $\Mm$-interval is a retract of a trivial extension on an element
of $\mathfrak{M}$ (\emph{cf.} \cite[def. 1.11]{BM12}).  
\end{defi}

\begin{prop}(\cite{BM12})
If $\Mm$ is combinatorial a generating set of $\Mm$-intervals always exists. 
\end{prop}

\begin{prop}(\cite{BM12}) 
Suppose that $\Mm$ is a right proper monoidal model category in which the unit is cofibrant and the operadic $\Set$-family $\Opcat{}$ is admissible.
Then a generating set of $\Mm$-interval $\mathfrak{M}$ is a generating set of intervals in the sense of def. \ref{def.weakint}.    
\end{prop}
\begin{proof}
This is basically a consequence of Proposition 2.24 and Lemma 2.10 in \cite{BM12} as remarked there (proof of Proposition 2.20):
if $\Mm$ is right proper and $\Opcat{}$ is admissible then given $X\in \ECat{\Mm}$, two objects $x,y\in X$ are homotopy equivalent
if and only if there exist a $\Mm$-interval $\mathbb{H}$ and a morphism $\morp{f}{\mathbb{H}}{X}$ such that $f(0)=x$ and $f(1)=y$.  
\end{proof}
\noindent 
It follows that if $\Mm$ is a combinatorial right proper monoidal model category in which the unit is cofibrant an $\Opcat{}$ is admissible
then the set (\ref{set.genes}) can be chosen to be:
\begin{equation}\label{genBMint}
 \mathcal{J}^{\Opcat{}}_{\mathfrak{M}}=\{ \arro{\theta_{\mathbb{H}}}{\mathbf{1}}{\mathbb{H}} \mid \mathbb{H}\in\mathfrak{M}\}
\end{equation}
where $\mathfrak{M}$ is a generating set of $\Mm$-intervals.

\subsection{Proof of Theorem \ref{main.thm}} \label{proof.main.thm}
Under the hypothesis of Theorem \ref{main.thm} the category $\EProp{\Mm}$ is locally presentable according to Corollary \ref{loc.pres.par}. 
In particular $\EProp{\Mm}$ is bicomplete and the set of morphisms $\mathcal{I}_{DK}$ and $\mathcal{J}_{DK}$ in (\ref{DKgenco}) and (\ref{DKgentco})
both admit the small object argument.
Therefore, according to Theorem 2.1.19 \cite{Ho99}, to prove the existence of the Dwyer-Kan model structure it is sufficient to check
that:
\begin{enumerate}[1)]
 \item\label{main1} The class $\mathcal{W}_{DK}$ has the 2-out-of-3 property and it is closed under retracts;
 \item\label{main2} $\inj{\mathcal{I}_{DK}}= \mathcal{W}_{DK} \cap \inj{\mathcal{J}_{DK}}$;
 \item\label{main3} $\cell{\mathcal{J}_{DK}}\subseteq \mathcal{W}_{DK}$.
\end{enumerate}

Point \ref{main1}) follows from the closure properties of the classes of weak equivalences in $\Mm$ and of essentially surjective functors
in $\Cat$.

Recall that $f\in \inj{\mathcal{I}_{DK}}$ if and only if it is a local trivial fibration surjective on colours;
it is clear that $\inj{\mathcal{I}_{DK}}\subseteq \mathcal{W}_{DK} \cap \inj{\mathcal{J}_{DK}}$.
Observe that $f\in \mathcal{W}_{DK} \cap \inj{\mathcal{J}_{DK}}$ if and only if it is a local trivial fibration
and $k^*(f)$ is a trivial fibration in the Dwyer-Kan model structure on $\ECat{\Mm}$ (\emph{cf.} \cite{Mu14}); 
in $\ECat{\Mm}$ the trivial fibrations
are the local trivial fibrations which are surjective on colours. It follows that $f\in \mathcal{W}_{DK} \cap \inj{\mathcal{J}_{DK}}$
if and only if $f$ is a local trivial fibration surjective on colours, that is if and only if $f\in \inj{\mathcal{I}_{DK}}$. 

We are now left to prove point \ref{main3}). Since we have supposed that the class of weak equivalences in $\Mm$ is closed under
transfinite composition the class of Dwyer-Kan weak equivalences is closed under transfinite composition.
It is therefore sufficient to prove that for every $k\in \mathcal{J}_{DK}$ and every push-out diagram
\begin{equation}\label{diagDK}
 \diagc{A}{X}{B}{Y,}{k}{f}{h}{g}{}
\end{equation}
if $f$ is a Dwyer-Kan weak equivalence then $h$ is a Dwyer-Kan weak equivalence.

If $k\in \mathcal{J}_{loc}$ then the following diagram
\begin{equation}\label{diagDK2} 
\diagc{f_!(A)}{X}{f_!(B)}{Y}{k}{\trc{f}}{h}{\trc{g}}{}
\end{equation}
is a push-out diagram in the fiber $\EPropfc{\Mm}{\Cl{A}}$. In particular $h$ is an isomorphism on objects and it is a local weak equivalence
since $f_!$ is a left Quillen functor and $k$ is a trivial cofibration in the projective model structure on $\EPropfc{\Mm}{\Cl{A}}$;
thus $h$ is a Dwyer-Kan weak equivalence.

Suppose now that $k\in \mathcal{J}_{\mathfrak{G}}$, then $k=k_! (\theta_G)$ for some weak interval $G\in {\mathfrak{G}}$.

We can decompose $\theta_G$ as $\psi\phi$ where $\phi$ is in $\ECatfc{\Mm}{\{0\}}$ and $\psi$ is fully faithful.
Diagram (\ref{diagDK}) is decomposed in two push-out squares:
\begin{equation}
 \xymatrix{
           k_! \widetilde{i^*_0(G)} \ar[d]_-{k_!(\phi)} \ar[r]^-{f} &  X\ar[d]^-{\phi'}\\
             w_! j_! i^*_0(\widetilde{G}) \ar[d]_-{k_! (\psi)} \ar[r] & X'\ar[d]^{\psi'}\\
             w_! j_! \widetilde{G} \ar[r]_-{g} & Y.}
\end{equation}
In \cite[Theorem 7.13, 7.14]{Mu14} Muro proved that $\phi$ is a trivial cofibration in $\ECatfc{\Mm}{\{0\}}$, thus $\phi$ is a weak equivalence in $\ECatfc{\Mm}{\{0\}}$.
Since $\psi$ is a fully faithful and injective on objects, $j_!(\psi)$ is also fully faithful and injective on colours; it follows by 
Theorem \ref{main.po} that $j_!(\psi)$ is a fully faithful and injective on colours.
This prove that $h=\psi' \phi'$ is a local weak equivalence. 

To conclude we only have to prove that $\pi_0(h)$ is essentially surjective;
this follows immediately from the fact that $0$ and $1$ are equivalent in $\pi_0(\widetilde{G})$ and the unique colour not in the
image of $h$ is $g(1)$. This completes the proof of Theorem \ref{main.thm}.

\begin{rmk}
    It is clear that, under the hypotheses of Theorem \ref{main.thm} on $\Mm$ the Dwyer-Kan model structure on $\ECat{\Mm}$ and $\EProp{\Mm}$ (\emph{cf.} \cite{Mu14} and \cite{C14}) and the adjunctions in the bottom row of diagram (\ref{eq:diagcatopgr}) are Quillen equivalences between the respective Dwyer-Kan model structures. 
\end{rmk}

\subsection{Fibrant objects}\label{sec.fibob}
Suppose $\Mm$ is a combinatorial monoidal model category satisfying the hypothesis of Theorem \ref{main.thm}; then the Dwyer-Kan 
model structure on $\EProp{\Mm}$ exists and the chosen set of generating trivial cofibrations contains $\mathcal{J}_{loc}$, so every
fibration in this model structure is a local fibration; the converse is not true since in general fibrations in the Dwyer-Kan model
structure has the right lifting property with respect to $\mathcal{J}_{\mathfrak{G}}$ (for a chosen generating set of weak intervals
$\mathfrak{M}$). However, as done in \cite{BM12}, it can be proved that if $\Mm$ is right proper and has cofibrant unit 
the class of Dwyer-Kan fibrant objects coincides with the one of locally fibrant objects.

\begin{prop}
Suppose that $\Mm$ satisfies the hypothesis of Theorem \ref{main.thm} an furthermore that it is right proper and with cofibrant unit,
then a $\Mm$-PROP $X\in \EProp{\Mm}$ is fibrant in the Dwyer-Kan model structure if and only if it is locally fibrant. 
\end{prop}
 
\begin{proof}
Clearly every Dwyer-Kan fibrant object is a locally fibrant object.
Suppose $X$ is locally fibrant. Under our hypothesis the set of generating trivial cofibrations for the Dwyer-Kan model structure
can be chosen to be $\mathcal{J}_{loc}\cup k_!(\mathcal{J}^{\Opcat{}}_{\mathfrak{M}})$ where $\mathfrak{M}$ is a generating set of
$\Mm$-intervals and $\mathcal{J}^{\Opcat{}}_{\mathfrak{M}}$ is as in (\ref{genBMint}). 
Since $X$ is locally fibrant, it has the right lifting property with respect to $\mathcal{J}_{loc}$; by adjunction
$X$ has the right lifting property with respect to $k_!(\mathcal{J}^{\Opcat{}}_{\mathfrak{M}})$ if and only if $k^*(X)$ has the
right lifting property with respect to $\mathcal{J}^{\Opcat{}}_{\mathfrak{M}}$, and this is indeed the case by \cite[Lemma 2.10]{BM12}
since $k^*(X)$ is locally fibrant (in $\ECat{\Mm}$).  
\end{proof}
\noindent 
Under these conditions it is also easy to prove that $\EProp{\Mm}$ is right proper.

\begin{prop}
Suppose that $\Mm$ satisfies the hypothesis of Theorem \ref{main.thm} an furthermore that it is right proper,
then the Dwyer-Kan model structure on $\EProp{\Mm}$ is right proper.  
\end{prop}

\begin{proof}
The proof is almost identical to the one of \cite[Proposition 5.3]{C14}. 
\end{proof}

\section{The operads for PROPs}\label{sec.opprop}
Fixed a set $C$, the algebraic structure of a $C$-coloured PROP is defined by a  bicollection of operations with two inputs and
one output, i.e. the vertical and the horizontal composition laws, a collection of operations with no input and one output, i.e.
the identities (see \ref{def:prop}) and a collection of morphisms with one input and one output, i.e. the permutation morphisms
of the underlying bicollection.
This suggests that the category of $C$-coloured PROPs $\EPropfc{\Mm}{C}$ can be presented as the category of algebras for a certain operad $\Opprop{C}$;
this is indeed the case and the aim of this section is to describe $\Opprop{C}$.   

The operad $\Opprop{C}$ is a $\Val{C}$-coloured $\Set$-operad. 
For every $n\in \N$ and $s_1,\dots,s_n,s_0\in \Val{C}$ the 
set of operations $\Opprop{C}(s_1,\dots,s_n; s_0)$ is defined to be the set of all isomorphism classes of completely ordered acyclic (\coa) $C$-graphs 
of arity $(s_1,\dots,s_n; s_0)$ (\ref{apx:completely ordered graph}, def. \ref{def:arity of a graph}).

For every $G\in \Opprop{C}(s_1,\dots,s_n; s_0)$, every $i\in \fcar{n}$ and every $H\in \Opprop{C}(t_1,\dots,t_m; s_i)$ the partial composition $G\circ_i H$ is
defined as the insertion $G\circ_{\sigma_G(i)} H$ (\S \ref{multins}), which is indeed an element of 
\[
\Opprop{C}(s_1,\dots,s_{i-1},t_1,\dots,t_m,s_{i+1},\dots,s_n; s_0). 
\]
For every $s\in \Val{C}$ the unit in $\Opprop{C}(s;s)$ is given by the unique c.o. $C$-corolla with $C$-valence $(s;s)$ and trivial twist.

For every $n\in \N$, every $\gamma\in \Sigma_n$ and  $s_1,\dots,s_n,s_0\in \Val{C}$, the symmetric action
\[
 \lmorp{\gamma^*}{\Opprop{C}(s_1,\dots,s_n; s_0)}{\Opprop{C}(s_{\sigma(1)},\dots,s_{\sigma(n)}; s_0)}
\]
sends $G$ to $\gamma^*G$ (def. \ref{defpermgr}).\\
The underlying category  $j^*(\Opcfprop{C})$ is equivalent to $\Sigma\Val{C}$ (\S \ref{sec:valences and bicollections}).

Constant-free and augmentation-free $C$-coloured $\Mm$-PROPs (\S \ref{sec2constantfree}) are algebras for two full suboperads of $\Opprop{C}$, denoted by $\Opcfprop{C}$
and $\Opafprop{C}$; $\Opcfprop{C}$ is spanned by the $C$-valences with non-empty input and $\Opafprop{C}$ is spanned by the $C$-valences
with non-empty output.

\subsubsection{Comments on \texorpdfstring{$\Opprop{C}$}{PROP}}

We think that the fact that the operad $\Opprop{C}$ has $C$-coloured PROPs as algebras is well known. 
In any case we would like to provide some justifications about why this is true. 

The following argument is basically a reformulation of Chapter 2 of \cite{JS91} for enriched symmetric monoidal categories (
called \emph{symmetric tensor categories} in \emph{loc.cit.}).

One of the main result of Joyal and Street in \cite{JS91} is Theorem 2.2 (and Corollary 2.3), that can be reformulated in the following way:
given a (un-enriched) symmetric monoidal category $\mathcal{M}$ and an ordered $\Ob{\mathcal{M}}$-graph $G$
(there called an \emph{anchored progressive polarised graph}) it is possible to define a value function
\begin{equation}
 \morp{\mathbf{V_G}}{\prod_{v\in N_G} \mathcal{M}(\iport{v},\oport{v})}{\mathcal{M}(\iport{G},\oport{G})}
\end{equation} 
using composition, tensor product and symmetries in $\mathcal{M}$; this function only depends on the isomorphism class of $G$ in the
sense that for every isomorphism of ordered $\Ob{\mathcal{M}}$-graphs $\morp{f}{G}{H}$ there is a commutative triangle
\[
 \xymatrix{\prod_{v\in N_G} \mathcal{M}(\iport{v},\oport{v}) \ar[d]_-{s_f} \ar[r]^-{V_G} & \mathcal{M}(\iport{G},\oport{G})\\
           \prod_{u\in N_H} \mathcal{M}(\iport{u},\oport{u})\ar[ur]_-{V_H} &  
          }
\]
where $s_{f}$ is the symmetric morphism in $\mathcal{M}$ induced by the restriction of $f$ on nodes.

In \cite{JS91} graphs are defined as particular one dimensional topological spaces, but it should be clear that the result
still holds with our definition of graphs.

Using the same methods of \cite{JS91} (decomposition of graphs in elementary graphs) 
it is possible (and not difficult) to generalise the above result to the enriched context, that is: given a $\Mm$-enriched symmetric
monoidal category $\cat{W}$ and an ordered $\Ob{\mathcal{W}}$-graph $G$ there is a morphism (in $\Mm$)
\begin{equation}
 \morp{\mathbf{V_G}}{\bigotimes_{v\in N_G} \IHom{\mathcal{W}}{\iport{v}}{\oport{v}}}{\IHom{\mathcal{W}}{\iport{G}}{\oport{G}}}
\end{equation}
which is invariant under isomorphisms of ordered $\Ob{\mathcal{M}}$-
graphs in the above sense.

Let $C$ be a fixed of colours, recall that $C$-coloured PROPs are just symmetric monoidal category with set of objects
freely generated by $C$.
 
Consider the obvious forgetful functor $\morp{U}{\EPropfc{\Mm}{C}}{\Mm^\Val{C}}$.
It admits a left adjoint $F$ that we are now going to describe.
For every $\prp{P}\in \Mm^\Val{C}$ and every $(\mathbf{a},\mathbf{b})\in \Val{C}$ 
\begin{equation}\label{eq:firstfreeprop}
 F\prp{P}(\mathbf{a},\mathbf{b})=\varinjlim_{G\in \mathbb{G}_{\text{ord},(\mathbf{a},\mathbf{b})}} \bigotimes_{v\in N_G} \prp{P}(\iport{v},\oport{v})
\end{equation}
where $\mathbb{G}_{\text{ord},(\mathbf{a},\mathbf{b})}$ is the groupoid of ordered acyclic $C$-graphs with $C$-valence $(\mathbf{a},\mathbf{b})$ and isomorphisms between them.

For every $\str{a},\str{b},\str{c}\in \Str{C}$, insertion over an appropriate untwisted
 vertical composition (def. \ref{def:graph comp}) defines a functor
\[
 \arr{\mathbb{G}_{\text{ord},(\mathbf{b},\mathbf{c})} \times \mathbb{G}_{\text{ord},(\mathbf{a},\mathbf{b})}}{\mathbb{G}_{\text{ord},(\mathbf{a}, \mathbf{c})}}
\] 
that is used to define the composition in $F\prp{P}$.
In the same way, for every $\str{a},\str{b},\str{c}, \str{d} \in \Str{C}$, insertion over an appropriate untwisted 
 horizontal composition gives
\[
 \arr{\mathbb{G}_{\text{ord},(\mathbf{a},\mathbf{b})} \times \mathbb{G}_{\text{ord},(\mathbf{c},\mathbf{d})}}{\mathbb{G}_{\text{ord},(\mathbf{a}\ast \mathbf{c},\mathbf{b}\ast \mathbf{d})}}
\]
that is used to define the tensor product in $F\prp{P}$. 

Permutation morphisms and identities in $F\prp{P}$ correspond to the components of an appropriate ordered $C$-graph with no nodes $S_{\alpha}$.
\[
 \xymatrix{\Un \ar[r]^-{\cong} & \bigotimes_{v\in S_{\alpha}} \IHom{\cat{W}}{\iport{v}}{\oport{v}} \ar[r] & F\prp{P}(\iport{S_{\alpha}},\oport{S_{\alpha}})}
\]

The functor $F$ is defined on morphisms in the evident way.

This construction $F\prp{P}$ is the enriched analogous of the free symmetric tensor category on a bicollection (a 
\emph{tensor scheme} in \emph{loc.cit.}) constructed in \cite[Theorem 2.3]{JS91}.

Given a $C$-PROP $\prp{R}$ The natural bijection
\[
 \lmorp{\phi}{\Mm^{\Val{C}}(\prp{P},U\prp{R})}{\EPropfc{\Mm}{C}(F\prp{P},\prp{R})}
\]
is defined in the following way.

For every $k\in \Mm^{\Val{C}}(\prp{P},U\mathbf{V})$, $\phi(k)$ is the unique morphism of $C$ coloured PROPs that
makes the following diagram commute
\begin{equation}
 \xymatrix@C=25pt{ \bigotimes_{v\in N_G} \prp{P}(\iport{v},\oport{v}) \ar[d] \ar[r]^-{\bigotimes k} & \bigotimes_{v\in N_G} \prp{R}(\iport{v},\oport{v}) \ar[d]^-{V_G}  \\
           F\prp{P}(\mathbf{a},\mathbf{b}) \ar[r]^-{\phi(k)} & \prp{R}(\mathbf{a},\mathbf{b}) }
\end{equation}  
for every $(\mathbf{a},\mathbf{b})\in \Val{C}$ and every $G\in \mathbb{G}_{\text{ord},(\mathbf{a},\mathbf{b})}$. 
It is routine to show that $\phi$ is bijective.

The functor $U$ is monadic in the sense that it satisfies one of the equivalent conditions of Beck's theorem
\cite[\S\ VI.7, Theorem 1]{McL98}.

An explicit description of the monad on $\Mm^{\Val{C}}$ associated to the operad $\Opprop{C}$ (\emph{c.f.} \S
\ref{sec:freeprop}) shows that it is isomorphic to $UF$; in other words $\EPropfc{\Mm}{C}$ is isomorphic to the category of $\Opprop{C}$-algebras in $\Mm$.

\subsubsection{Freeness of the permutation action}
Recall that every $C$-coloured operad $\ope{O}$ has and underlying symmetric $C$-coloured collection, i.e. an element of $\Mm^{\opc{\Sigma\Str{C}}\times C}$. The following definition is classical.

\begin{defi}
 A $C$-coloured operad $\ope{O}$ is \emph{$\Sigma$-free} if its underlying symmetric $\Mm$-collection
is free, i.e. if $\ope{O}(\str{s})$ is a free as and $\Aut (\str{s})$-set for every $\str{s}\in \opc{\Sigma\Str{C}}\times {C}$
(where $\Aut (\str{s})\subset \Sigma_{\card{\str{s}}}$ denotes the automorphism group of $\str{s}$).
\end{defi}

\noindent The operad $\Opprop{C}$ is not $\Sigma$-free. For example,
let $o\in \Opprop{C}(([],[]),([],[]);([],[]))$ be the operation represented by the graph in Figure \ref{fixedpoint};
 $\tau^*(o)=o$ for every $\tau\in \Sigma_2$. 
\begin{figure}[!ht]
\begin{tikzpicture}
\node[vert] (a) at (-1,0) {1};
\node[vert] (b) at (1,0) {2}; 
\end{tikzpicture}
\caption{}\label{fixedpoint} 
\end{figure}

As another example, the permutation $(1\,2)(3\,4)\in \Sigma_4$ has the operation represented by the graph in Figure \ref{fixedpoint2}
as fixed point. 

\begin{figure}[!ht]
\begin{tikzpicture}
\node[vert] (a) at (0,1) {1};
\node[vert] (b) at (0,2) {2};
\node[vert] (c) at (-1.2,0) {3};
\node[vert] (d) at (1.2,0) {4};

\outcheckpt{a}{2}{0.5}{-0.6}
\outcheckpt{b}{2}{0.5}{-0.3}
\incheckpt{c}{2}{0.3}{0.2}
\incheckpt{d}{2}{0.3}{0.2}
\connectnodes{a}{c}{2/2/b}
\connectnodes{a}{d}{1/1/a}
\connectnodes{b}{c}{1/1/a}
\connectnodes{b}{d}{2/2/b}
\end{tikzpicture}
\caption{}\label{fixedpoint2} 
\end{figure}

However the permutation groups act freely on the operations represented by completely ordered $C$-graph with no connected 
components of valence $(0,0)$ by Proposition \ref{trivialauto}.

\begin{prop}
The operads $\Opcfprop{C}$ and $\Opafprop{C}$ are $\Sigma$-free. 
\end{prop}
\begin{proof}
All the operations in $\Opcfprop{C}$ (resp. $\Opafprop{C}$) are represented by graphs whose nodes are connected to a port. 
Proposition \ref{trivialauto} implies that a completely ordered $C$-graph can not be isomorphic to a graph with the same underlying
ordered $C$-graph and a different node order; therefore the actions of the permutation groups in 
$\Opcfprop{C}$ and $\Opafprop{C}$ are free. 
\end{proof}

\subsection{PROPs, Operads and Categories}\label{operandcat}

The operad $\Opprop{C}$ has two important full suboperads:
\begin{equation}\label{inclusion.op}
 \xymatrix{\Opcat{C}\ar@<-1.5pt>[r]^{j} &  \Opop{C}\ar@<-1.5pt>[r]^{w} & \Opprop{C}.}
\end{equation}

$\Opop{C}$ is the full suboperad spanned by the $C$-valences with exactly one output;
 as a consequence, its operations are all the c.o. $C$-graphs which are trees; the category of $\Opop{C}$-algebras
in $\Mm$ is isomorphic to the category of $C$-coloured $\Mm$-enriched operads $\EOperfc{\Mm}{C}$.\\
$\Opcat{C}$ is the full suboperad spanned by the $C$-valences with exactly one output and one input; as a consequence, its
 operations are all the c.o. linear $C$-graphs (i.e. non-empty \coa graphs in which all the nodes
has valence $(1,1)$). the category of $\Opcat{C}$-algebras
in $\Mm$ is isomorphic to $\ECatfc{\Mm}{C}$ the category of $\Mm$-enriched categories with set of objects equal to $C$.
The composite $wj$ will be denoted by $k$. 

The inclusions $j$ and $w$ produce adjunctions between the categories of algebras:
\[
\xymatrix{
            \ECatfc{\Mm}{C} 
            \ar@<3pt>[r]^-{j_!}&  \ar@<3pt>[l]^-{j^*} \EOperfc{\Mm}{C} \ar@<3pt>[r]^-{w_!}&  \ar@<3pt>[l]^-{w^*} \EPropfc{\Mm}{C}
          }
\] 
and by composition
\[
 \xymatrix{\ECatfc{\Mm}{C} 
            \ar@<3pt>[r]^-{k_!}&  \ar@<3pt>[l]^-{k^*} \EPropfc{\Mm}{C};}
\]
the left adjoints $j_!$ and $w_!$ are both fully faithful (\emph{cf.}  corollary \ref{cor:ff}).

The underlying category of the operad $\EPropfc{\Mm}{C}$ is isomorphic to $\Sigma\Val{C}$, i.e. the permutation groupoid of $C$-valences; seen $\Sigma\Val{C}$ as an operads, its category of algebras in $\Mm$ is equivalent to the functor category $\Mm^{\Sigma\Val{C}}$, i.e. the category of symmetric $\Val{C}$-bicollections. Similarly the discrete operad $\disf{\EPropfc{\Mm}{C}}\cong \Val{C}$ has $\Mm$-bicollections as algebras (in $\Mm$). The inclusions:
\[
 \xymatrix{\Val{C} \ar@<-1.5pt>[r]^{} &  \Sigma\Val{C} \ar@<-1.5pt>[r]^{w} & \Opprop{C}.}
\]
induce adjunctions (\ref{eq:freeprop}) and (\ref{eq:freesigmaprop}).

\subsubsection{Free PROP generated by an operad} For every $C$-coloured $\Mm$-operad $\ope{O}$ the free $C$-coloured $\Mm$-PROP $w_!\ope{O}$ can be described by the following formula:

\begin{equation}\label{eq:eq free prop oper}
w_!(\ope{O})\ioval{\mathbf{c}}{\mathbf{d}}=\coprod_{\morp{f}{\fcar{n}}{\fcar{m}}} \bigotimes_{i=1}^{m} \ope{O}\ioval{\covf{\mathbf{c}}{f}{i}}{d_i}.
\end{equation}

In the case $\Mm=\Set$ thus for every $\mathbf{c}=(n,c), \mathbf{d}=(m,d)\in \Str{C}$ the elements of $w_!\ope{O}(\mathbf{c},\mathbf{d})$ 
are couples $(f,\{o_i\}^{m}_{i=1})$
where $\morp{f}{\fcar{n}}{\fcar{m}}$ is a function of sets and $o_i\in \ope{O}(\covf{\mathbf{c}}{f}{i},d_i)$ for every $i\in \fcar{m}$; such a couple can be represented as a forest of one-node trees labelled by operations of $\ope{O}$ with inports shuffled by $\omega_f$,the unshuffling of $f$ (Figure \ref{fig:free prop ope}).  
\begin{center}
    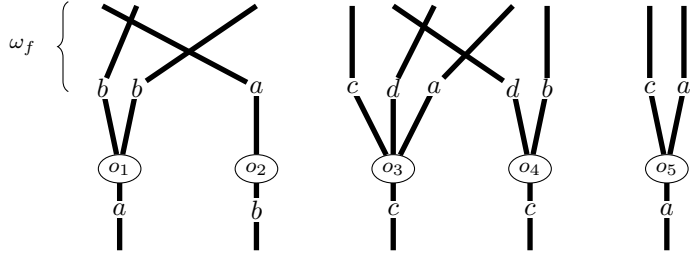
\begin{figure}[!ht]
        \begin{tikzpicture}[scale=0.9]
        \foreach \x in {1,...,3}
        {
            \coordinate (a-\x-up) at ($(0,0)+(-2*3*0.5+\x*2,0)$);	
        }
        \foreach \x in {1,...,2}
        {
            \coordinate (b-\x-up) at ($(5,0)+(-2*2*0.5+\x*2,0)$);	
        }	
        \permgrl{a-1-up/a-1-up/$a$,
            a-2-up/a-2-up/$b$,
            a-3-up/a-3-up/$c$,
            b-1-up/b-1-up/$c$,
            b-2-up/b-2-up/$a$}
        {1.2}
        \grnodes{a-1-up-up}{2}{0.5}{1.2}{\footnotesize $o_1$}
        \grnodes{a-2-up-up}{1}{0.5}{1.2}{\footnotesize $o_2$}
        \grnodes{a-3-up-up}{3}{0.6}{1.2}{\footnotesize $o_3$}
        \grnodes{b-1-up-up}{2}{0.5}{1.2}{\footnotesize $o_4$}
        \grnodes{b-2-up-up}{2}{0.5}{1.2}{\footnotesize $o_5$}
        \permgl{a-1-up-up-1/a-2-up-up-1/$a$,
            a-1-up-up-2/a-1-up-up-1/$b$,
            a-2-up-up-1/a-1-up-up-2/$b$,
            a-3-up-up-1/a-3-up-up-1/$c$,
            a-3-up-up-2/b-1-up-up-1/$d$,
            a-3-up-up-3/a-3-up-up-2/$d$,
            b-1-up-up-1/a-3-up-up-3/$a$,
            b-1-up-up-2/b-1-up-up-2/$b$,
            b-2-up-up-1/b-2-up-up-1/$c$,
            b-2-up-up-2/b-2-up-up-2/$a$}{1.2}
        \node (p1) at ($(a-1-up-up-1)+(-0.5,-0.2)$) {}; 
        \node (p2) at ($(p1|- a-1-up-up-1-up)+(0,0.2)$) {};
        \draw [decorate,decoration={brace,amplitude=4pt},xshift=-4pt,yshift=0pt]
        (p1) -- (p2) node [black,midway,xshift=-17pt] {\footnotesize
            $\omega_f$};
        \end{tikzpicture}
        \caption{Presentation of $(f,\{o_i\}^5_{i=1})\in w_!\ope{O}(a,b,b,c,d,d,a,b,c,a;a,b,c,c,a)$}\label{fig:free prop ope}
    \end{figure}
\end{center}
Given a morphism $\morp{g}{\fcar{n}}{\fcar{m}}$ and every $i\in \fcar{m}$ let $\morp{g^{\sharp}_i}{l}{n}$ be the unique order preserving injective function
whose image is $g^{-1}(i)$; for every $\mathbf{c}=(n,c)\in \Str{C}$ define $\covf{\mathbf{c}}{g}{i}=(l,cf^{\sharp})$.
For example if $\morp{g}{\fcar{3}}{\fcar{2}}$ is such that $g(1)=f(3)=1$ and $g(2)=2$ and $\mathbf{c}=(c_1,c_2,c_3)$
then $\covf{\mathbf{c}}{g}{1}=(c_1,c_3)$ and $\covf{\mathbf{c}}{g}{2}=(c_2)$.

Given $\mathbf{e}=(l,e)\in \Str{C}$ and morphisms $(f,\{o_i\}^{m}_{i=1})\in \FSm{\Mop}(\mathbf{c},\mathbf{d})$ and 
$(g,\{p_i\}^{e}_{i=1})\in \FSm{\Mop}(\mathbf{d},\mathbf{e})$ their composition is defined as
\[
(gf,\{\omega_i^*(p_i \circ (o_{g_i^{\sharp}(1)},\dots,o_{g_i^{\sharp}(k_i)}))\}_{i=1}^{e})
\]
where $k_i=\card{f^{-1}(i)}$ for every $i\in \fcar{m}$ and
$\omega_i\in \Sigma_{\card{(gf)^{-1}(i)}}$ is the permutation that makes sure that the operation has the right valence (and it depends only on $f$ and $g$); more precisely $\omega_i$ is the  the unshuffling
of $f$ restricted to $(gf)^{-1}(i)$ (\S \ref{not:total orders and unshuffles}).

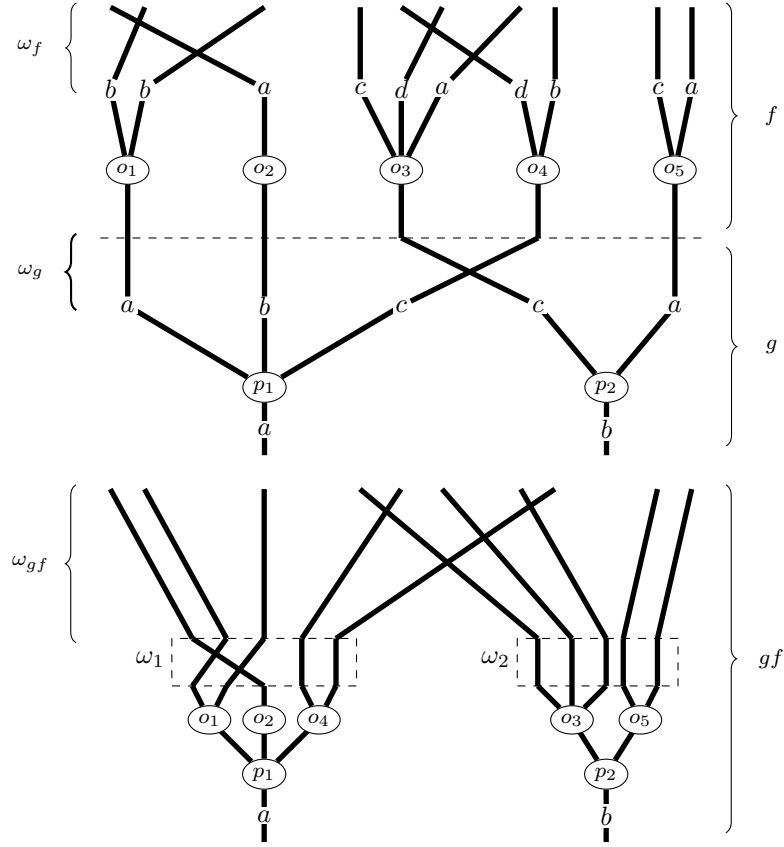
\begin{figure}[!ht]
    \begin{tikzpicture}[scale=0.9]
    \coordinate (a) at (0,0);
    \grnodes{a}{3}{2}{1.2}{\footnotesize $p_1$}
    \grouts{a}{1}{a}
    \coordinate (b) at (5,0);
    \grnodes{b}{2}{2}{1.2}{\footnotesize $p_2$}
    \grouts{b}{1}{b}
    \permgl{a-1/a-1/$a$,
        a-2/a-2/$b$,
        b-1/a-3/$c$,
        b-2/b-2/$a$,
        a-3/b-1/$c$}
    {1}
    \permg{a-1-up/a-1-up,
        a-2-up/a-2-up,
        a-3-up/a-3-up,
        b-1-up/b-1-up,
        b-2-up/b-2-up}
    {1}
    \grnodes{a-1-up-up}{2}{0.5}{1.2}{\footnotesize $o_1$}
    \grnodes{a-2-up-up}{1}{0.5}{1.2}{\footnotesize $o_2$}
    \grnodes{a-3-up-up}{3}{0.6}{1.2}{\footnotesize $o_3$}
    \grnodes{b-1-up-up}{2}{0.5}{1.2}{\footnotesize $o_4$}
    \grnodes{b-2-up-up}{2}{0.5}{1.2}{\footnotesize $o_5$}
    \permgl{a-1-up-up-1/a-2-up-up-1/$a$,
        a-1-up-up-2/a-1-up-up-1/$b$,
        a-2-up-up-1/a-1-up-up-2/$b$,
        a-3-up-up-1/a-3-up-up-1/$c$,
        a-3-up-up-2/b-1-up-up-1/$d$,
        a-3-up-up-3/a-3-up-up-2/$d$,
        b-1-up-up-1/a-3-up-up-3/$a$,
        b-1-up-up-2/b-1-up-up-2/$b$,
        b-2-up-up-1/b-2-up-up-1/$c$,
        b-2-up-up-2/b-2-up-up-2/$a$}{1.2}
    \node (p1) at ($(a-1-up-up-1)+(-0.5,-0.2)$) {}; 
    \node (p2) at ($(p1|- a-1-up-up-1-up)+(0,0.2)$) {};
    \draw [decorate,decoration={brace,amplitude=4pt},xshift=-4pt,yshift=0pt]
    (p1) -- (p2) node [black,midway,xshift=-17pt] {\footnotesize
        $\omega_f$};
    \node (q2) at ($(p1|- a-1)+(0,-0.2)$) {};
    \node (q1) at ($(p1|- a-1-up)+(-0,0.2)$) {};
    \draw [thick,decorate,decoration={brace,amplitude=4pt},xshift=-4pt,yshift=0pt]
    (q2) -- (q1) node [black,midway,xshift=-17pt] {\footnotesize
        $\omega_g$};
    \draw[dashed] ($(a-1-up)+(-0.4,0)$) -- ($(b-2-up)+(0.4,0)$); 
    \node (s1) at ($(b-2-up-up-2-up)+(0.5,0.2)$) {}; 
    \node (s2) at ($(s1|- a-1-up)+(0,0)$) {};
    \draw [decorate,decoration={brace,amplitude=4pt},xshift=-4pt,yshift=0pt]
    (s1) -- (s2) node [black,midway,xshift=17pt] {\footnotesize
        $f$};
    \node (t1) at ($(s1|- a-1-up)+(0,0)$) {}; 
    \node (t2) at ($(s1|- outa)+(0,0)$) {};
    \draw [decorate,decoration={brace,amplitude=4pt},xshift=-4pt,yshift=0pt]
    (t1) -- (t2) node [black,midway,xshift=17pt] {\footnotesize
        $g$};	    
    \end{tikzpicture}
    \begin{tikzpicture}[scale=0.9]
    \coordinate (a) at (0,0);
    \grnodes{a}{3}{0.8}{0.8}{\footnotesize $p_1$}
    \grouts{a}{1}{a}
    \coordinate (b) at (5,0);
    \grnodes{b}{2}{1}{0.8}{\footnotesize $p_2$}
    \grouts{b}{1}{b}
    \coordinate (ga) at (a);
    \centerp{ga}{3}{2}{1.2}
    \coordinate (gb) at (b);
    \centerp{gb}{2}{2}{1.2}
    \upcopy{ga-1,
        ga-2,
        gb-1,
        gb-2,
        ga-3}
    {2}
    \centerp{ga-1-up}{2}{0.5}{1}
    \centerp{ga-2-up}{1}{0.5}{1}
    \centerp{ga-3-up}{3}{0.6}{1}
    \centerp{gb-1-up}{2}{0.5}{1}
    \centerp{gb-2-up}{2}{0.5}{1}
    \upcopy{ga-1-up-1,
        ga-1-up-2,
        ga-2-up-1,
        ga-3-up-1,
        ga-3-up-2,
        ga-3-up-3,
        gb-1-up-1,
        gb-1-up-2,
        gb-2-up-1,
        gb-2-up-2}{1.2}
    \grnodes{a-1}{2}{0.5}{0.5}{\footnotesize $o_1$}
    \grnodes{a-2}{1}{0.5}{0.5}{\footnotesize $o_2$}
    \grnodes{a-3}{2}{0.5}{0.5}{\footnotesize $o_4$}
    \grnodes{b-1}{3}{0.5}{0.5}{\footnotesize $o_3$}
    \grnodes{b-2}{2}{0.5}{0.5}{\footnotesize $o_5$}
    \groupn{a-1-1,a-1-2,a-2-1,a-3-1,a-3-2}{u-}
    \upcopy{u-1,u-2,u-3,u-4,u-5}{0.7}
    \conx{1/3,2/1,3/2,4/4,5/5}{u-}{-up}{u-}{}
    \groupn{b-1-1,b-1-2,b-1-3,b-2-1,b-2-2}{v-}
    \upcopy{v-1,v-2,v-3,v-4,v-5}{0.7}
    \conx{1/1,2/2,3/3,4/4,5/5}{v-}{-up}{v-}{}
    \conx{a-1-up-1/1,
        a-1-up-2/2,
        a-2-up-1/3,
        a-3-up-2/4,
        b-1-up-2/5}
    {g}{}{u-}{-up.center}
    \conx{
        a-3-up-1/1,
        a-3-up-3/2,
        b-1-up-1/3,
        b-2-up-1/4,
        b-2-up-2/5}
    {g}{}{v-}{-up.center}
    \draw[dashed] ($(u-1-up)+(-0.3,0)$) rectangle ($(u-5)+(0.3,0)$) node[pos=0,xshift=-8pt,yshift=-8pt] {$\omega_1$};
    \draw[dashed] ($(v-1-up)+(-0.3,0)$) rectangle ($(v-5)+(0.3,0)$) node[pos=0,xshift=-8pt,yshift=-8pt] {$\omega_2$};
    \node (s1) at ($(gb-2-up-2)+(0.5,0.2)$) {}; 
    \node (s2) at ($(s1|- outa)+(0,0)$) {};
    \draw [decorate,decoration={brace,amplitude=4pt},xshift=-4pt,yshift=0pt]
    (s1) -- (s2) node [black,midway,xshift=17pt] {\footnotesize
        $gf$};
    \node (p1) at ($(ga-1-up-1)+(-0.5,0.2)$) {}; 
    \node (p2) at ($(p1|- u-1-up)+(0,-0.2)$) {};
    \draw [decorate,decoration={brace,amplitude=4pt},xshift=-4pt,yshift=0pt]
    (p2) -- (p1) node [black,midway,xshift=-17pt] {\footnotesize
        $\omega_{gf}$};
    \end{tikzpicture}
    \caption{Graphical presentation of the composition of two morphisms $(f,\{o_i\}^5_{i=1})$, $(g,\{p_j\}^2_{j=1})$ in $w_!\ope{O}$}
\end{figure}

The identity of $\mathbf{c}$ is given by $(\id_n,\{\id_{c_i}\}^n_{i=0})$.

The tensor product is defined on objects as $(n,c)\boxtimes (m,d)=(n+m,c\ast d)$ and on morphisms in the obvious way.

Let $\mathsf{Com}$ be the terminal coloured $\Set$-operad, then $\FSm{Com}$ is isomorphic to $\Gamma$ (i.e. the skeletal category of the category of finite sets).

The construction of the PROP $w_!\ope{O}$ is not new and appears in several places in the literature; for example, a detailed description of this construction (in the one-coloured case) appears in \cite[\S 5.4.1]{LV12} and \cite[\S 2]{Fr15draft}.

For convenience we also rewrite formula (\ref{eq:eq free prop oper}) in the case in which $\ope{O}$ is generated by a $\Mm$-category,
i.e. when $\ope{O}\cong j_!(\cat{C})$ for some $\cat{C}\in \ECat{\Mm}$:
\begin{equation}\label{eq:free prop cat}
k_!\cat{C}(\str{c},\str{d})\cong w_!j_!\cat{C}(\str{c},\str{d})\cong \coprod_{f\in \fSigma(n,m)} \bigotimes_{i\in \fcar{m}} \cat{C}(c_{f^{-1}(i)}, d_i)
\end{equation}
for every $\str{c},\str{d}\in \Str{C}$ of length $n$ and $m$ respectively. 
Note that (the underlying symmetric monoidal $\Mm$-category of) $k_!\cat{C}$ can be characterised as the \emph{free symmetric monoidal category} generated by $\cat{C}$.

\subsection{Some Operadic Adjunctions}\label{sec.operadj}
In appendix \ref{sec: alg as Kan extensions} we give a description of the adjunction between the categories of algebras induced by a morphism $f$ of operads; in particular we described the left adjoint as a left Kan extension along $w_!f$.
 Here we focus on examples in which the target of $f$ is $\Opprop{C}$;
as usual for left Kan extensions, the description of $f_!$ will
involve colimits indexed by the comma category $\comcat{w_!f}{\mathbf{v}}$ for some $\mathbf{v}\in \FSm{\Opprop{C}}$. 

Since $f_!$ is strong monoidal, it value on the objects is determined (up to isomorphisms) by its value on $\mathbf{v}\in\Val{C}\subset \Cl{w_!\Opprop{C}}$.

To better understand how the objects of these comma categories look like we are going to describe the slice categories over objects of $\FSm{\Opprop{C}}$. We then proceed with the description of the PROP associated to an operadic Grothendieck construction.
Finally, we describe the adjunctions of algebras which are relevant for us.

\subsubsection{Slice in \texorpdfstring{$w_!\Opprop{C}$}{PROPs}}\label{sec:sliceprop}
Fix a set $C$ and a $C$-valence $\str{s}$; the comma category $\comcat{(w_!\Opprop{C})}{{\str{s}}}$ can be described in the following way:\\
\textbf{Objects:} are all the \coa $C$-graphs with $C$-valence $\str{s}$;\\
\textbf{Morphisms:} suppose two \coa $C$-graphs $G$, $H$ with residue $\str{s}$ are given; an element of the hom-set $(\FSm{\Opprop{C}}/{\str{s}})(G,H)$
is a couple $(\{K_i\}_{i=1}^m, \alpha)$, where  $\{K_i\}_{i=1}^m$ is a sequence of \coa $C$-graphs insertable over $H$ and $\alpha$ is a permutation for its insertion over $H$ (def. \ref{permins})  such that $G\cong H \circ_{\alpha} (K_1,\dots,K_m)$ (\S \ref{multins}).

In Figure \ref{graphslice} for example, $G$ is the graph on the left, $x,y$ are the nodes marked by $1$ and $2$, $H^G_{x,y}$ is the target and $K^G_{xy}$ is the graph in the first box from the left; $\morp{f}{\fcar{3}}{\fcar{2}}$ is such that $f(3)=2$, $f(1)=f(2)=1$.

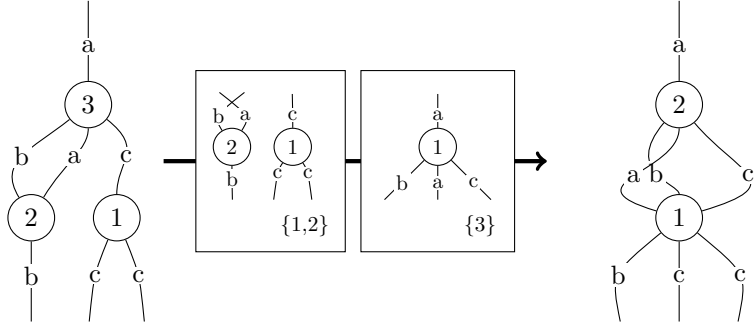
\begin{figure}
    \begin{tikzpicture}
    \node (s) at (-4,0) {
        \begin{tikzpicture}[scale=1.5]
        \node[vert] (a) at (0,0.5) {3};
        \node[vert] (b) at  (-0.5,-0.5) {2};
        \node[vert] (c) at (0.25,-0.5) {1};
        \inports{0,1.5}{1}{0};
        \outports{0,-1.5}{3}{0.25};
        \incheckpt{a}{1}{0}{0.5};
        \incheckpt{c}{1}{0}{0.5};
        \incheckpt{b}{2}{0.3}{0.5};
        \outcheckpt{a}{3}{0.25}{-0.4};
        \outcheckpt{b}{1}{0}{-0.5};
        \outcheckpt{c}{2}{0.2}{-0.5};
        \connectntoinp{a}{1/1/a};
        \connectntoop{b}{1/1/b};
        \connectntoop{c}{1/2/c,2/3/c};
        \connectnodes{a}{b}{1/1/b,2/2/a};
        \connectnodes{a}{c}{3/1/c};
        \end{tikzpicture} };
    
    \node (t) at (4,0) {
        \begin{tikzpicture}[scale=1.5]
        \node[vert] (a) at (0,0.5) {2};
        \node[vert] (b) at  (0,-0.5) {1};
        \inports{0,1.5}{1}{0};
        \outports{0,-1.5}{3}{0.25};
        \incheckpt{a}{1}{0}{0.5};
        \incheckpt{b}{3}{0.5}{0.3};
        \outcheckpt{a}{3}{0.25}{-0.6};
        \outcheckpt{b}{3}{0.3}{-0.5};
        \connectntoinp{a}{1/1/a};
        \connectntoop{b}{1/1/b,2/2/c,3/3/c};
        \connectnodes{a}{b}{1/2/b,2/1/a,3/3/c};
        \end{tikzpicture} };
    \draw[->,ultra thick] (s.east) -- 
    node{
        \begin{tikzpicture}
        \node[fill=white,draw,-] at (-1,0) {
            \begin{tikzpicture}[scale=0.8, every node/.style={transform shape}]
            \inports{-0.17,1}{3}{0.33};
            \outports{0.17,-1}{3}{0.33};
            \node[vert] (b) at  (-0.5,0) {2};
            \node[vert] (c) at (0.5,0) {1};
            \incheckpt{b}{2}{0.3}{0.5};
            \outcheckpt{b}{1}{0.25}{-0.5};
            \incheckpt{c}{1}{0.3}{0.5};
            \outcheckpt{c}{2}{0.25}{-0.4};
            \connectntoinp{b}{1/2/a,2/1/b};
            \connectntoop{b}{1/1/b};
            \connectntoinp{c}{3/1/c};
            \connectntoop{c}{1/2/c,2/3/c};
            \node at (0.7,-1.3) {\{1,2\}};
            \end{tikzpicture}
        };
        \node[fill=white,draw,-] at (1.2,0) {
            \begin{tikzpicture}[scale=0.8, every node/.style={transform shape}]
            \inports{0,1}{1}{0.5};
            \outports{0,-1}{3}{0.5};
            \node[vert] (a) at  (0,0) {1};
            \incheckpt{a}{1}{0}{0.5};
            \outcheckpt{a}{3}{0.3}{-0.6};
            \connectntoinp{a}{1/1/a};
            \connectntoop{a}{1/1/b,2/2/a,3/3/c};
            \node at (0.7,-1.3) {\{3\}};
            \end{tikzpicture}
        };
    \end{tikzpicture}
} 
(t.west);
\end{tikzpicture}
\caption{Example of a morphism in $\comcat{w_!{\Opprop{C}}}{\str{s}}$; the graph in the $n$-th box from the left
    is inserted on the $n$-th node of the target. The numbers on the lower-right corner of the $n$-th box indicate
    which nodes of the source are sent to the $n$-th node of the target.}\label{graphslice}
\end{figure}

As another example let $G$ be a \coa $C$-graph of $C$-valence $\str{s}$ and $x,y$ be two nodes of $G$ with no dead-ends paths of length greater than $1$ between them and consider the \coa $C$-graphs $H^G_{xy}$ and $K^G_{xy}$ defined in \S \ref{coa.decomp}; proposition \ref{prop:grapheldecomp} tell us that $f_{xy}=((C_{\sigma_G(1)},\dots,K^G_{xy},\dots,C_{\sigma_G(n)}),\alpha_{xy})$ is a morphism in $\FSm{\Opprop{C}}/{\mathbf{v}}$
\begin{equation}\label{comp.morp}
\lmorp{f_{xy}}{G}{H^G_{xy}.}
\end{equation}

A morphism $\morp{(\{K_i\}_{i=1}^m, \alpha)}{G}{H}$ is an isomorphism if and only if $\alpha$ is a bijection and in that case
all $K_i$'s are $C$-corollas; such an isomorphism uniquely define an isomorphism of
(not completely ordered) $C$-graph from $G$ to $H$ preserving the port order: $\alpha$ determines the image of the nodes and the $K_i$ determines the image
of edges (note that edges which are not ports of a node are necessarily ports of $G$, thus their image is determined by the requirement
that the port order is preserved).

This assignment establishes a bijection between the set of isomorphisms from $G$ to $H$ in $\comcat{w_!{\Opprop{C}}}{{\str{s}}}$ and the set of
isomorphisms of $C$-graphs from $G$ to $H$ preserving the port order.

\subsubsection{The PROP for operadic \texorpdfstring{$I$}{I}-families}\label{sec.feypar}

Let $\morp{\ofam{F}}{I}{\Oper}$ be an operadic $I$-family.

Consider the operadic Grothendieck construction $\gco \ofam{F}$. The PROP $\FSm{(\gco \ofam{F})}$ has the following description:\\
\textbf{Objects:} are finite sequences $\{(i_k,c_k)\}^n_{k=1}$ where $i_k\in I$ and $c_k\in \ofam{F}(i_k)$ for every $1\leq k \leq n$.\\
\textbf{Morphisms:} given two objects $\{(i_k,c_k)\}^n_{k=1}$ and $\{(j_h,d_h)\}^m_{h=1}$ a morphism between them are 
triples $(\alpha, \{f_k\}^n_{k=1}, \{p_h\}^m_{h=1})$ where:
\begin{itemize}
    \item[-] $\morp{\alpha}{\fcar{n}}{\fcar{m}}$ is a function;
    \item[-] $\morp{f_k}{i_k}{j_{\alpha{k}}}$ is a morphism in $I$ for every $1\leq k \leq n$;
    \item[-] $p_h$ is an operation in  $\ofam{F}(d_h)(\{f_k(c_k)\}_{k\in \alpha^{-1}(h)};c_h)$ for 
    every $1\leq l \leq m$.
\end{itemize}

There is a canonical morphism of operads $\morp{\mathrm{colim}}{\gco (\ofam{F})}{\varinjlim \ofam{F}}$ which produces a morphism of PROPs
\[
\morp{\mathrm{colim}}{\FSm{(\gco \ofam{F})}}{\FSm{(\varinjlim \ofam{F})}}
\]
For every $i\in I$ let $\morp{e_i}{\ofam{F}(i)}{\varinjlim \ofam{F}}$ be the $i$-morphism defining the colimit.

Let $E=\varinjlim (\Clf \circ\ofam{F})$ be the set of colours of $\varinjlim \ofam{F}$. For every $e\in E$ the comma category 
$\comcat{\mathrm{colim}}{e}$ has the following description:\\
\textbf{Objects:} are couples $(\{(i_k,c_k)\}_{k\in \fcar{n}}, o)$ where $\{(i_k,c_k)\}_{k\in \fcar{n}}$ is an object in $\FSm{(\gco \ofam{F})}$ and
$o$ is an operation in $(\varinjlim \ofam{F})(\{e_{i_1}(c_{i_1})\}^n_{i=1};e)$.

\textbf{Morphisms:} a morphism between $(\{(i_k,c_k)\}^n_{k=1}, o_1)$ and $(\{(j_h,d_h)\}^n_{h=1}, o_2)$ is a morphism  $(\alpha, \{f_k\}_{k=1}^n, \{p_h\}_{l=1}^m)$ in 
$\FSm{(\gco \ofam{F})}$ such that $\omega^*_\alpha (o_2 \circ (\{e_{j_h}(p_h)\}_{h=1}^m))=o_1$.

\begin{rmk}\label{rmkpostipar} 
    When $I$ is a poset we can omit the collection $\{f_k\}^n_{k=1}$ from the description of a morphism, since it is completely determined by the source and the target. 
\end{rmk}
\subsubsection{Example I: Free PROPs generated by a bicollection}\label{sec:freeprop}
As a trivial example consider the morphism of operads $\arr{\Val{C}}{\Opprop{C}}$ which induces the free-forgetful adjunction (\ref{eq:freeprop}).
For every $\str{c}\in \Val{C}$ the comma category $\comcat{w_!i}{\str{c}}$ is equivalent to 
$\mathbb{G}_{\mathrm{ord},\str{c}}$, the category of \coa (i.e. completely ordered acyclic,\S \ref{apx:completely ordered graph}) $C$-graph of $C$-valence $\str{c}$ 
and isomorphisms of ordered graphs between them.
For every $C$-coloured bicollection $A\in \cat{M}^{\Val{C}}$
\begin{equation}\label{eq:secondfreeprop}
F_{\Opprop{C}}A(\str{c})\cong 
\underset{G\in \mathbb{G}_{\mathrm{ord},\str{c}}}{\varinjlim}
\bigotimes_{v\in \verx{G}} A(\arity{v}).
\end{equation}
\subsubsection{Example II: Free PROP generated by a symmetric bicollection}\label{sec:freesigmaprop}
Consider now the morphism of operads $\morp{\eta}{\Sigma \Val{C}}{\Opprop{C}}$ (\S \ref{sec.opprop}) which induces adjunction (\ref{eq:freesigmaprop}); for every $C$-valence $\str{v}\in \Val{C}$, the comma category
$\comcat{w_!\eta}{\str{v}}$ the category $\comcat{w_!\eta}{\str{v}}$ is the the maximal subgroupoid of the slice $\comcat{w_!{\Opprop{C}}}{\str{v}}$; 
it is equivalent to $\mathbb{G}_{\mathrm{port},\str{v}}$, the category with acyclic $C$-graph with a port order of $C$-valence $\str{v}$ as objects and  isomorphisms of $C$-graphs that preserves the port order as morphisms.
 
Let $\mathbb{G}_{\str{v}}$ be the groupoid of $C$-graph with residue $\str{v}$ and isomorphisms between them;
the functor $\morp{\pi}{\mathbb{G}_{\mathrm{port,\str{v}}}}{\mathbb{G}_{\mathrm{\str{v}}}}$ that forgets the order on the ports is a discrete opfibration with all fibers isomorphic to $\Aut(\str{v})$, the set of automorphisms of $\str{v}$ in $\Sigma\Val{C}$.

It follows that for every $C$-coloured symmetric $\cat{M}$-bicollection $A\in \cat{M}^{\Sigma \Val{C}}$ the free PROPs generated by $A$ has $\str{v}$-component 
\[
n_!A(\str{v})= \coprod_{G\in \pi_0(\mathbb{G}_{\str{c}})}\Aut(\str{v})\underset{\Aut(G)} \otimes \left( \bigotimes_{g\in \verx{G}} A(\arity{g}) \right). 
\]

\subsubsection{Example III: Change of colours}\label{sec.changeofcl}
For every $C,D\in \Set$ and every function $\morp{f}{C}{D}$, there is a morphism of operads
$\morp{\widetilde{f}}{\Opprop{C}}{\Opprop{D}}$ sending an operation (i.e. a \coa $C$-graph) $(G,\bar{\lambda},\tau,\sigma)$ to
$(f(G),\bar{\lambda},\tau,\sigma)$ (the same graph with the labelling changed).

This produces an adjunction between the categories of algebras
\[
\adjpair{\widetilde{f}_!}{\widetilde{f}^*}{\EPropfc{\Mm}{C}}{\EPropfc{\Mm}{D}}
\]

Note that when $f$ is injective the morphism $\widetilde{f}$ is fully-faithful, thus $f_!$ is a fully faithful functor (corollary \ref{cor:ff}).
 
To describe $\widetilde{f}_!$ it will be useful to give the following definition
\begin{defi}
A (\coa) $f$-graph is a couple $(G,\{l_u\}_{u\in N_G})$ where $G$ is a (\coa) graph, and $l_u$
is a $C$-labelling on $\resd{u}$, such that $fl_u$ and $fl_t$ coincide on $\cie{u}{t}$.
\end{defi}

\noindent In other words $f$-graphs are graphs whose ports are labelled by $C$ and whose inner edges are labelled by $C\underset{D}\times C$. Isomorphisms of (\coa) $f$-graphs are defined in the obvious way.

Given a $C$-coloured PROP $P$ and a $D$-valence $\mathbf{v}$
\[
\widetilde{f}_!(P)(\mathbf{v}) \cong \varinjlim_{(\mathbf{u},g) \in w_!\widetilde{f}/\mathbf{v}} P(\mathbf{u}).
\] 

The comma category $\comcat{w_!\widetilde{f}}{\mathbf{v}}$ has the following description:\\
\textbf{Objects:} according to the definition, the objects are couples $(\{\mathbf{u}_i\}_{i=1}^n,G)$ where $\{\mathbf{u}_i\}_{i=1}^n$ is a sequence of $C$-valences (for some $n\in \N$) and $G$ is a \coa $D$-graph $G$ of
arity $(f(\mathbf{u_1}),\dots,f(\mathbf{u_n}); \mathbf{v})$. Alternatively these objects can be described as \coa $f$-graphs with residue $\mathbf{v}$.\\
\textbf{Morphisms:} given two $f$-graphs $(G,\{l_i\}_{i\in N_G})$ and $(H,\{k_{j}\}_{j\in N_H})$,with $\card{N_G}=n$ and $\card{N_H}=m$, an element in $\widetilde{f}/\mathbf{v}(G,H)$
is a couple $(\{K_i\}_{i=1}^{m},\alpha)$ and the $\{K_i\}_{i=1}^m$ is a sequence of \coa $C$-graphs such that $\{f(K_i)\}_{i=1}^m$ is insertable over $H$
and $\morp{\alpha}{\fcar{n}}{\fcar{m}}$ is a permutation for that insertion over $H$ such that $G=H\circ_{\alpha} (f(K_1),\dots,f(K_m))$.

\begin{prop}\label{mate.prop}
Given a function $\morp{f}{C}{D}$, in the commutative diagram of adjunctions
\[
\xymatrix{\EOperfc{\Mm}{C} \ar@<3pt>[r]^{w_!} \ar@<-3pt>[d]_{f_!}& \EPropfc{\Mm}{C} \ar@<3pt>[d]^{f_!} \ar@<3pt>[l]^{w^*}\\
          \EOperfc{\Mm}{D} \ar@<3pt>[r]^{w_!} \ar@<-3pt>[u]_{f^*} & \EPropfc{\Mm}{D} \ar@<3pt>[l]^{w^*} \ar@<3pt>[u]^{f^*}}
\]
the mate $ w_! f^* \Rightarrow f^* w_!$ is an isomorphism.
\end{prop}
\begin{proof}
For every $\ope{O}\in \EOperfc{\Mm}{D}$ and every $C$-valence $\ioval{\mathbf{a}}{\mathbf{b}}\in \Val{C}$ of valence $(n,m)$ the 
$\ioval{\mathbf{a}}{\mathbf{b}}$-component of the considered mate at $\ope{O}$ coincides with the isomorphism
\[ 
 w_!f^*\ope{O}\ioval{\mathbf{a}}{\mathbf{b}}\cong \coprod_{\morp{\alpha}{\fcar{n}}{\fcar{m}}} \bigotimes_{i=1}^{m}  \ope{O} \ioval{f(\covf{\mathbf{a}}{\alpha}{i})}{f(b_i)}\cong
 \]
 \[
 \cong\coprod_{\morp{\alpha}{\fcar{n}}{\fcar{m}}} \bigotimes_{i=1}^{m} \ope{O}\ioval{\covf{f(\mathbf{a})}{\alpha}{i}}{f(\mathbf{b})_i}\cong f^*w_!\ope{O}\ioval{\mathbf{a}}{\mathbf{b}}.
\]

\end{proof}
\noindent 
The following corollary can also be proven directly.

\begin{cor}\label{mate.cor}
Suppose that 
\[
\diagc{A}{B}{C}{D}{i}{f}{h}{g}{}
\]
is a push-out diagram in $\Set$ such that $i$ is an injective function; consider the commutative diagram of adjunctions
\[
\xymatrix{\EOperfc{\Mm}{A} \ar@<3pt>[r]^{f_!} \ar@<-3pt>[d]_{i_!} & \EPropfc{\Mm}{C} \ar@<3pt>[d]^{h_!} \ar@<3pt>[l]^{f^*}\\
          \EOperfc{\Mm}{B} \ar@<3pt>[r]^{g_!} \ar@<-3pt>[u]_{i^*} & \EPropfc{\Mm}{D}. \ar@<3pt>[l]^{g^*} \ar@<3pt>[u]^{h^*}}
\]
The mate $f_! i^* \Rightarrow h^* g_!$ is an isomorphism.
\end{cor}
\begin{proof}
This is an immediate consequence of 
\cite[Proposition B.26]{C14} and Proposition \ref{mate.prop}.
\end{proof}

\subsubsection{Example IV: Push-outs of PROPs}\label{sec.podiag}
Let $\dPo$ be the poset with three elements $O,A,B$ such that $O<B$ and $O<A$ (as in (\ref{pocat})).
Colimits over $\dPo$-diagram in $\Mm$ are just push-outs in $\Mm$.

From \S\ \ref{sec.diagalg} (see also Remark \ref{rmkBV}) we know that the colimits adjunction
\[
 \adjpair{\varinjlim}{\mathrm{const}}{\EPropfc{\Mm}{C}^{\dPo}}{\EOperfc{\Mm}{C}}
\]
is isomorphic to the adjunction 
\[
 \adjpair{c_!}{c^*}{\EPropfc{\Mm}{C}^{\dPo}}{\EOperfc{\Mm}{C}}
\]
induced by the morphism of operads $\morp{c}{\Opprop{C}\BVt \dPo}{\Opprop{C}}$.

So, given a $\dPo$-diagram of $C$-coloured PROPs $D\in \EPropfc{\Mm}{C}^{\dPo}$, the value of the push-out of $D$ at a $C$-valence $\mathbf{v}\in \Val{C}$
is
\begin{equation}\label{eq.colim1}
 c_!(D)\cong(\varinjlim D)(\mathbf{v})\cong \varinjlim_{(\mathbf{u},g)\in \comcat{\FSm{(\Opprop{C}\BVt {\dPo})}} {\mathbf{v}}} D(\mathbf{u}).
\end{equation}
It will be convenient to denote the PROPs $D(A),D(B)$ and $D(O)$ by $D_A$,$D_B$ and $D_O$.

The index category $\comcat{\FSm{(\Opprop{C}\BVt {\dPo})}} {\mathbf{v}}$ has the following description:\\
\textbf{Objects} are couples $(\mathbf{u},g)$; $\mathbf{u}=\{(s_i,p_i)\}_{i=1}^{n}$ is a sequence of couples where
$s_i\in \Val{C}$ and $p_i\in {\dPo}$ for every $i\in \fcar{n}$, while $g$ is a \coa
$C$ graph of arity $(s_1,\dots,s_n; \mathbf{v})$; alternatively we can describe this objects as couples $(g,M)$ where $g$ is a \coa $C$-graph
of valence $\mathbf{v}$ and $\morp{M}{N_g}{\{A,O,B\}=\Ob{\dPo}}$ is a function of sets (to recover the previous description set
$p_i=M(\sigma_g(i))$ for every $i\in \fcar{n}$).  

It will be convenient to give a name to this kind of objects:
\begin{defi}
A \emph{$\dPo$-marking} for a (\coa) $C$-graph $G$ is a function $\morp{M}{N_G}{\Ob{\dPo}}$.\\
A \emph{$\dPo$-marked (\coa) $C$-graph} (or \coa $\dPo$-$C$-graph for short) is a couple $(G,M)$ where $G$ is a (\coa) $C$-graph $M$ is a $\dPo$-marking for $G$.
The marking of a $\dPo$-marked \coa $C$-graph $G$ will be denoted by $M_G$.
\end{defi}
\begin{figure}
 \begin{tikzpicture}[scale=1]
 \node[verto] (o) at (0.5,-1.7) {3};
    \incheckpt{o}{2}{0.4}{0.5};
    \outcheckpt{o}{2}{0.5}{-0.4};
 \node[vertb] (b) at (3.5,-1) {1};
    \incheckpt{b}{2}{0.4}{0.5};
    \outcheckpt{b}{1}{0.3}{-0.5};  
 \node[verta] (a) at (0.5,0) {2};
    \incheckpt{a}{3}{0.4}{0.5};
    \outcheckpt{a}{2}{0.3}{-0.5};  
 \inports{1.5,1}{5}{0.75};
 \outports{1.5,-3}{3}{1};
 \connectntoinp{a}{1/2/a,2/1/c,3/3/a};
 \connectnodes{a}{o}{1/1/c,2/2/b};
 \connectntoop{o}{1/1/d,2/2/d};
 \connectntoinp{b}{4/1/d,5/2/d};
 \connectntoop{b}{1/3/b};
 \end{tikzpicture}
\caption{A graphical representation of a $\dPo$-marked $C$-graph of valence $(a,c,a,d,d;d,d,b)$: nodes marked by $A$ have the shape of a down-pointing triangle,
 nodes marked by $B$ have the shape of a right-pointing triangle,
 nodes marked by $O$ have the shape of a diamond.
}\label{fig.pcgraph}
\end{figure}
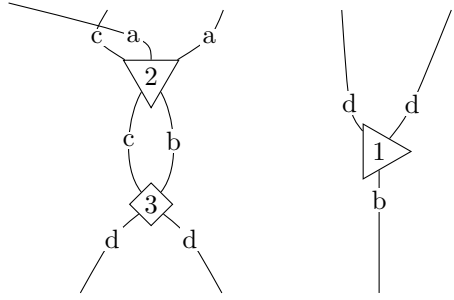
\noindent \textbf{Morphisms:} 
observe that graph insertion is defined for \coa $\dPo$-$C$-graphs as well:

\begin{defi}
Given a \coa $\dPo$-$C$-graph $G$ and a node $s\in N_G$ a \coa $\dPo$-$C$-graph $K$ is called \emph{insertable over $G$ at $s$} if 
so is the underlying graph of $K$ and $M_K(x)\leq M_G(s)$ for every $x\in N_K$.
The \coa $C$-graph $G\circ K$ inherits a $\dPo$-marking from $G$ and $K$.

A sequence of \coa $\dPo$-$C$-graphs $\{K_i\}_{i=1}^m$ with $m=\card{N_G}$
is called \emph{insertable over $G$} if $K_i$ is insertable at $\sigma_G(i)$ for
every $i\in \fcar{m}$. 
If $\morp{\alpha}{\fcar{n}}{\fcar{m}}$ is a permutation for the insertion of $\{K_i\}_{i=1}^m$ over $G$, the resulting \coa $C$-graph
\[
 G\circ_{\alpha} (K_1,\dots,K_m)
\]
inherits an obvious $\dPo$-marking.   
\end{defi}

\noindent Suppose two \coa $\dPo$-$C$-graphs $G$ and $H$ of valence $\mathbf{v}$ are given and set $n=\card{N_G}$, $m=\card{N_H}$. 
As one might expect, an element of $(\comcat{\FSm{(\Opprop{C}\BVt {\dPo})}} {\mathbf{v}})(G,H)$  is a couple
$(\{K_i\}_{i=1}^m,\alpha)$ where $\{K_i\}_{i=1}^m$ is a sequence of \coa $\dPo$-$C$-graphs
insertable over $H$ and $\morp{\alpha}{\fcar{n}}{\fcar{m}}$ is a permutation for the insertion of $\{K_i\}_{i=1}^m$ over $H$ 
such that $G=H\circ_{\alpha} (K_1,\dots,K_m)$.

The functor $\morp{w_!(c)}{\comcat{\FSm{(\Opprop{C}\BVt {\dPo})}} {\mathbf{v}}}{\comcat{\FSm{\Opprop{C}}}{\mathbf{v}}}$ forgets the $\dPo$-marking.

\begin{defi}
    A morphism $\morp{f}{G}{H}$ in $\comcat{\FSm{(\Opprop{C}\BVt {\dPo})}}{\mathbf{v}}$ is said to be \emph{graph-preserving} if $w_!(c)(G)=w_!(c)(H)$ and $w_!(c)(f)$ is equal to the identity.
\end{defi}
\noindent Informally, a morphism is graph-preserving if and only if it changes only the marking on the nodes from the source to the target.
    
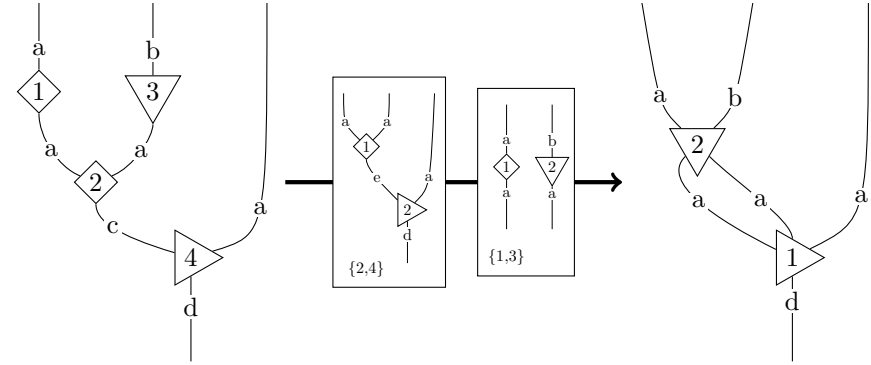
\begin{figure}
\begin{tikzpicture}
\node (s) at (-4,0) 
{
\begin{tikzpicture}
\inports{0,2.5}{3}{0.75}
\outports{0.5,-2.5}{1}{0.5}
\node[verto] (1) at (-1.5,1.2) {1};
   \incheckpt{1}{1}{0.5}{0.5};
   \outcheckpt{1}{1}{0.5}{-0.6};
\node[verta] (3) at (0,1.2) {3};
   \incheckpt{3}{1}{0.5}{0.5};
   \outcheckpt{3}{1}{0.5}{-0.6};
\node[verto] (2) at (-0.75,0) {2};
   \incheckpt{2}{2}{0.5}{0.2};
   \outcheckpt{2}{1}{0.5}{-0.5};
\node[vertb] (4) at (0.5,-1) {4};
   \incheckpt{4}{2}{1}{0.3};
   \outcheckpt{4}{1}{0.5}{-0.6};
\connectntoinp{1}{1/1/a};
\connectntoinp{3}{2/1/b};
\connectntoinp{4}{3/2/a};
\connectntoop{4}{1/1/d};
\connectnodes{1}{2}{1/1/a};
\connectnodes{3}{2}{1/2/a};
\connectnodes{2}{4}{1/1/c};         
\end{tikzpicture}
};
\node (t) at (4,0) 
{
\begin{tikzpicture}
\inports{0,2.5}{3}{0.75}
\outports{0.5,-2.5}{1}{0.5}
\node[verta] (2) at (-0.75,0.5) {2};
   \incheckpt{2}{2}{0.5}{0.5};
   \outcheckpt{2}{2}{0.5}{-0.5};
\node[vertb] (1) at (0.5,-1) {1};
   \incheckpt{1}{3}{0.5}{0.5};
   \outcheckpt{1}{1}{0.5}{-0.5};
\connectntoinp{1}{3/3/a};
\connectntoinp{2}{1/1/a,2/2/b};
\connectntoop{1}{1/1/d};
\connectnodes{2}{1}{1/1/a,2/2/a};         
\end{tikzpicture}
};
\draw[->,ultra thick] (s.east) --
node{
\begin{tikzpicture}[scale=0.6, every node/.style={transform shape}]
\node[fill=white,draw,-] at (1.5,0) {
\begin{tikzpicture}
\inports{0,1.5}{2}{0.5};
\outports{0,-1.5}{2}{0.5};
\node[verta] (2) at (0.5,0) {2};
   \incheckpt{2}{1}{0.5}{0.5};
   \outcheckpt{2}{1}{0.5}{-0.5};
\node[verto] (1) at (-0.5,0) {1};
   \incheckpt{1}{1}{0.5}{0.5};
   \outcheckpt{1}{1}{0.5}{-0.5};
\connectntoinp{1}{1/1/a};
\connectntoinp{2}{2/1/b};
\connectntoop{1}{1/1/a};
\connectntoop{2}{1/2/a};
\node at (-0.5,-2) {\{1,3\}};
\end{tikzpicture}
}; 
\node[fill=white, draw,-] at (-1.5,0) {
\begin{tikzpicture}
\inports{0,2}{3}{0.5};
\outports{0.4,-2}{1}{0.5};
\node[vertb] (2) at (0.4,-0.7) {2};
   \incheckpt{2}{2}{0.5}{0.5};
   \outcheckpt{2}{1}{0.5}{-0.5};
\node[verto] (1) at (-0.5,0.7) {1};
   \incheckpt{1}{2}{0.5}{0.5};
   \outcheckpt{1}{1}{0.5}{-0.5};
\connectntoinp{1}{1/1/a,2/2/a};
\connectntoinp{2}{3/2/a};
\connectntoop{2}{1/1/d};
\connectnodes{1}{2}{1/1/e};
\node at (-0.5,-2) {\{2,4\}};
\end{tikzpicture}
};       
\end{tikzpicture}
}
(t.west);
\end{tikzpicture}
\caption{A morphism}
\end{figure}

We can rewrite (\ref{eq.colim1}) as
\[
 (\varinjlim D)(\mathbf{v})\cong \varinjlim_{(g,M)\in \comcat{\FSm{(\Opprop{C}\BVt {\dPo})}} {\mathbf{v}}} (\bigotimes_{k\in N_g}D_{M(k)}(\resd{k}))
\]

To get a feeling on how the right-hand side looks like, if $(g,M)$ is the $\dPo$-$C$-graph of Figure \ref{fig.pcgraph} then
\[
 \bigotimes_{k\in N_g}D_{M(k)}(\resd{k})=D_b(\ioval{d,d}{B})\otimes D_A(\ioval{c,a,a}{c,b})\otimes D_O(\ioval{c,b}{d,d}).
\]

The unit $\ntra{l}{D}{c^*c_!(D)}$ has also an explicit description;
for every $X\in \dPo$ and $\mathbf{v}\in \Val{C}$,  let $C_{X,\mathbf{v}}$ be the unique untwisted $\dPo$-$C$-corolla with $C$-valence $\mathbf{v}$ and with its node marked by $X$;
the $\mathbf{v}$-component of $l_X$ is the canonical morphism
\begin{equation}\label{}
 \lmorp{\iota_{C_{X,\mathbf{v}}}}{D_X(\mathbf{v})}{\varinjlim_{(g,M)\in \comcat{(\FSm{\Opprop{C}\BVt \dPo})}{\mathbf{v}} } (\bigotimes_{k\in N_g}D_{M(k)}(\resd{k}))}
\end{equation}

\section{Push-out of a PROP along a morphism of operads}\label{sec.poalongop}
We are now ready to prove our main theorem about push-out of PROPs:
\begin{thm}\label{main.po}
Let $\morp{a}{U}{V}$ be a morphism of $\Mm$-operads injective on the colours and fully faithful and let $\morp{f}{w_!(U)}{P}$ be a morphism of PROPs.  
In the push-out diagram
\[
 \diagc{w_!(U)}{P}{w_!(V)}{R}{w_!(a)}{f}{b}{g}{}
\]
the morphism $b$ is injective on colours and fully-faithful.
\end{thm}
\begin{proof}
Since $\morp{\Clf}{\EProp{\Mm}}{\Set}$ is a bifibration and $a$ is injective on colours, then $b$ is injective on colours.
Let $D=\Cl{R}$ and $C=\Cl{P}$; the two diagrams
\begin{equation}\label{diagpofc}
 \diagc{g_!a_!w_!(U)}{b_!(P)}{g_!w_!(V)}{R}{g_!(\trc{w_!(a)})}{b_!(\trc{f})}{\trc{b}}{\trc{g}}{}
\quad
\diagc{w_! g_! a_!(U)}{b_!(P)}{w_!g_!(V)}{R}{w_!g_!(\trc{a})}{b_!(\trc{f})}{\trc{b}}{\trc{g}}{}
\end{equation}
are isomorphic and they are both push-outs in $\EPropfc{\Mm}{D}$.

Since $b$ is injective on colours the unit of $(b_!,b^*)$ is an isomorphism (\emph{cf.} \S\ \ref{sec.changeofcl}), therefore to prove that $\trf{b}$ is an isomorphism
it is sufficient to prove that $b^*(\trc{b})$ is an isomorphism. 

Let $\morp{\mathbb{D}}{\dPo}{\EPropfc{\Mm}{D}}$ be the diagram represented by the upper-left corner of the right diagram in (\ref{diagpofc}), so
$\mathbb{D}_A=w_!g_!(V)$, $\mathbb{D}_O=w_!g_!a_!(U)$ and $\mathbb{D}_B=b_!(P)$;
it is actually a mixed $\dPo$-diagram of algebras in $\Mm$ (\S\ \ref{sec.diagalg}) with hierarchy functor $\kappa$ represented by the following inclusions of operads 
\[
 \diagc{\Opop{C}}{\Opprop{C}}{\Opop{D}}{\Opprop{D},}{b}{w}{w}{b}{}
\]

in other words $\mathbb{D}$ is an algebra in $\Mm$ for the operad $\mathcal{G}=\fbv{\mathcal{\Opprop{D}}}{\dPo}{\kappa}$.
 
To prove that $b^*(\trc{b})$ is an isomorphism is equivalent to check that for every $\mathbf{v}\in \Val{C}$
\[
 \lmorp{(\trc{b})_{\mathbf{v}}}{b_!(P)(\mathbf{v})}{R(\mathbf{v})}
\]
is an isomorphism.

From \S\ \ref{sec.podiag} we know that this map is isomorphic to
\begin{equation}\label{first.pomorp}
 \lmorp{\iota_{C_{B,\mathbf{v}}}}{b_!(P)(\mathbf{v})}{\varinjlim_{g\in \comcat{\FSm{\mathcal{G}}}{\mathbf{v}}} \bigotimes_{k\in N_g}\mathbb{D}_{M_{g(k)}}(\resd{k}).}
\end{equation}

Observe that $\trf{g_!(\trc{w_!(a)})}$ is and isomorphism; in fact it is isomorphic to $\morp{f_!w_!(\trf{a})}{f_!w_!(U)}{f_!a^*w_!(V)}$ (\emph{cf.} Corollary \ref{mate.cor}), 
which is an isomorphism since $\trf{a}$ is an isomorphism by assumption. 

The category $\comcat{\FSm{\mathcal{G}}}{\mathbf{v}}$ is the full subcategory of $\comcat{\FSm{\Opprop{D}\BVt \dPo}}{\mathbf{v}}$ that contains every \coa $\dPo$-$D$-graphs $G$ such that:
\begin{itemize}
 \item[-] every node marked by $A$ or $O$ has exactly one export;
 \item[-] every node marked by $O$ or $B$ has all the port labelled by elements of $C$. 
\end{itemize}
 
For every \coa $\dPo$-$D$-graph $G$ in $\comcat{\FSm{\mathcal{G}}}{\mathbf{v}}$ and every $u\in N_G$ marked by $O$ let $G_u$ be the 
\coa $\dPo$-$D$-graph that has the same underlying \coa $D$-graph and the same $\dPo$-marking except on $u$, where $M_{G_u}(u)=A$.
Let $\morp{w_{G,u}}{G}{G_u}$ be the unique graph-preserving morphism from $G$ to $G_u$;
Let $(\comcat{\FSm{\mathcal{G}}}{\mathbf{v}})^+\cong (\comcat{\FSm{\mathcal{G}}}{\mathbf{v}})[W^{-1}]$ be the small category obtained from $\comcat{\FSm{\mathcal{G}}}{\mathbf{v}}$
by formally inverting the morphisms in
\[
 W=\{\morp{w_{G,u}}{G}{G_u} \mid G\in \comcat{\FSm{\mathcal{G}}}{\mathbf{v}},\quad u\in N_G, M_{G}(u)=O\}.
\]
The diagram defining the colimit on the right in expression (\ref{first.pomorp}) can be extended to $(\comcat{\FSm{\mathcal{G}}}{\mathbf{v}})^+$
in a unique way since every morphism in $W$ is sent to an isomorphism in $\Mm$ and clearly
\[
 \varinjlim_{g\in \comcat{\FSm{\mathcal{G}}}{\mathbf{v}}} \bigotimes_{k\in N_g}\mathbb{D}_{M_{g(k)}}(\resd{k})\cong \varinjlim_{g\in (\comcat{\FSm{\mathcal{G}}}{\mathbf{v}})^+} \bigotimes_{k\in N_g}\mathbb{D}_{M_{g(k)}}(\resd{k})
\]
Therefore, to prove that (\ref{first.pomorp}) is an isomorphism, it is sufficient to prove that $C_{B,\mathbf{v}}$ is final in $(\comcat{\FSm{\mathcal{G}}}{\mathbf{v}})^+$.

Note that the canonical functor (which basically forget the $\dPo$-marking)
\[
 \lmorp{c_{\mathbf{v}}}{\comcat{\FSm{\mathcal{G}}}{\mathbf{v}}}{\comcat{\FSm{\Opprop{D}}}{\mathbf{v}}}
\]
is faithful and send the element of $W$ to isomorphisms.
Therefore there is a canonical faithful functor
\[
 \lmorp{c^+_{\mathbf{v}}}{(\comcat{\FSm{\mathcal{G}}}{\mathbf{v}})^+}{\comcat{\FSm{\Opprop{D}}}{\mathbf{v}}}
\]
which sends $C_{B,\mathbf{v}}$ to the corolla $C_{\mathbf{v}}$. Since $C_{\mathbf{v}}$ is final in $\comcat{\FSm{\Opprop{D}}}{\mathbf{v}}$, to check that $C_{B,\mathbf{v}}$ is final it is sufficient to check that there exists a morphism from $G$ to 
$C_{B,\mathbf{v}}$ for every $G\in (\comcat{\FSm{\mathcal{G}}}{\mathbf{v}})^+$.

Let $G$ be \coa $\dPo$-$D$-graph $G$ in $(\comcat{\FSm{\mathcal{G}}}{\mathbf{v}})^+$, if $G$ has nodes marked by $O$ we can 
compose graph-preserving morphisms (changing the markings from $O$ to $A$) to get a morphism $\morp{q}{G}{G'}$ where $G'$ is
a \coa $\dPo$-$D$-graph without nodes marked by $O$; thus we can restrict ourselves to the case in which $G$ has no nodes marked by $O$.

Suppose $G$ has no nodes marked by $O$; note that
since $G$ has $D$-valence $\mathbf{v}$ (which is actually a $C$-valence), there can not be ports labelled by elements of $D\backslash C$; let $n$ be the number of inner 
edges of $G$ labelled by elements in $D\backslash C$; we claim that it is not restrictive to suppose that $n=0$. 

In fact suppose $n\geq 0$ and $e$ is an inner edge of $G$; then $e$ has to be an export of a node $x$ and the input of a node $y$, so
both $x$ and $y$ has to be marked by $A$. As a consequence, $\cie{x}{y}=\{e\}$ (since $x$ has $e$ has unique output) and $x$ and $y$
has no other dead-ends paths between them. Let $H^G_{xy}$ be the \coa graph $D$-graph defined in \S\ \ref{coa.decomp}; the 
$\dPo$-marking of $G$ induces a natural $\dPo$-$D$-graph structure on $H^G_{xy}$.
The morphism (\ref{comp.morp}) in $\comcat{(\FSm{\Opprop{D}})}{\mathbf{v}}$ (obtained from proposition \ref{prop:grapheldecomp}) can be lifted to a morphism in $\dPo$-$D$-graph
\[
 \lmorp{\widetilde{f}_{xy}}{G}{H^G_{xy}}
\]
the $\dPo$-$D$-graph $H^G_{xy}$ has less (inner) edges labelled by elements of $D\backslash C$ than $G$ and (no nodes marked by $O$).
Iterating this process we can find a morphism $\morp{m}{G}{H}$ such that $H$ has no edges labelled by elements of $D\backslash C$.

Suppose then that $G$ has no edges labelled by elements of $D\backslash C$ and no nodes marked by $O$; let $\widetilde{G}$ be the 
$\dPo$-$D$-graph which as the same underlying \coa $D$-graph but all nodes marked by $B$.
There is a morphism $\morp{l}{G}{\widetilde{G}}$ in $(\comcat{\FSm{\mathcal{G}}}{\mathbf{v}})^+$ such that $c^+_{\mathbf{v}}(l)$ is the identity, obtained as composition of maps in $W$ (to change all the $A$-marking in $O$-marking) and  a graph-preserving map (changing all the $O$-marking
in $B$-marking). 

Now $\widetilde{G}$ has all the nodes marked by $B$, therefore the unique morphism $\arr{c^+_{\mathbf{v}}(\widetilde{G})}{C_{\mathbf{v}}}$ (in
$\comcat{(\FSm{\Opprop{D}})}{\mathbf{v}}$) lifts to a unique marking preserving morphism $\arr{\widetilde{G}}{C_{B,\mathbf{v}}}$; this last morphism
composed with $l$ gives the desired morphism from $G$ to $C_{B,\mathbf{v}}$. 
\end{proof}

\begin{rmk}
It is not true, in general, that the push-out of a fully-faithful inclusion of PROPs is again fully-faithful; in fact there are
$\dPo$-$D$-graphs,
as the one in Figure \ref{obstructiongr}, that can not be reduced to $\dPo$-$D$-graphs with all the node marked by $B$.
 
\begin{figure}[ht]
\begin{tikzpicture}
 \node[verta] (a) at (0,1.2) {1};
 \node[vertb] (b) at (-0.7,0) {2};
 \node[verta] (d) at (0,-1.2) {4};
 \inports{0,2.4}{1}{0};
 \outports{0,-2.4}{1}{0};
 \incheckpt{a}{1}{0}{0.5};
 \outcheckpt{a}{2}{0.5}{-0.3};
 \incheckpt{b}{1}{0}{0.3};
 \outcheckpt{b}{1}{0}{-0.3};
 \incheckpt{c}{1}{0}{0.3};
 \outcheckpt{c}{1}{0}{-0.3};
 \incheckpt{d}{2}{0.7}{0.3};
 \outcheckpt{d}{1}{0}{-0.7};
 \connectnodes{a}{b}{1/1/x};
 \connectnodes{a}{d}{2/2/d};
 \connectnodes{b}{d}{1/1/y};
 \connectntoinp{a}{1/1/z};
 \connectntoop{d}{1/1/y};
\end{tikzpicture} 
\caption{}\label{obstructiongr}
\end{figure}

\end{rmk}
\noindent The following two corollaries are immediate since the inclusions \[\ECfProp{\Mm} \hookrightarrow \EProp{\Mm}\text{ and }\EAfProp{\Mm} \hookrightarrow \EProp{\Mm}\] 
preserve colimits (they are left adjoints).

\begin{cor}\label{cor.po1}
Let $\morp{a}{U}{V}$ be a morphism of $\Mm$-operads injective on the colours and fully faithful and let $\morp{f}{w_!(U)}{P}$ be a morphism of 
constant-free (augmentation-free) PROPs.  
In the push-out diagram of constant-free (augmentation-free) PROPs
\[
 \diagc{w_!(U)}{P}{w_!(V)}{R}{w_!(a)}{f}{b}{g}{}
\]
the morphism $b$ is injective on colours and fully-faithful. 
\end{cor}

\section{Locally presentable operadic families}\label{sec.locpresop}
This section gives sufficient conditions on a operadic family $\ofam{F}$ to guarantee that the total category
$\ToBi{\ofam{F}}{\Mm}$ is locally presentable (when $\Mm$ is locally presentable).

Along the way, we also prove that the adjunction (\ref{freeforgfib}) between $\ofam{F}$-collections and $\ofam{F}$-algebras is monadic and that it is finitary if $\ofam{F}$ is.
For the whole section $\Mm$ will be a bicomplete symmetric monoidal category whose tensor product commutes with colimits
in each variable. 

\subsection{\texorpdfstring{$\ofam{F}$}{F}-Collections}
A discrete (coloured) operad is a (coloured) operad whose unique operations are the identities. The full subcategory of $\Oper$ spanned
by discrete coloured operads is isomorphic to $\Set$;
this inclusion of $\Set$ into $\Oper$ correspond to the section of the bifibration $\morp{\Clf}{\Oper}{\Set}$ which associate to
every set $C$ the initial $C$-coloured operad.
 
A discrete $\mathcal{C}$-family is an operadic $\mathcal{C}$-family which takes values in discrete operads; equivalently, it
is a $\mathcal{C}$-diagram in $\Set$. The category of $\ofam{F}$-algebras of a discrete family $\ofam{F}$ 
is isomorphic to the category of $\ofam{F}$-collection.

\begin{lemma}\label{lambda.dir}
Suppose $I$ is a $\lambda$-directed poset and let $\morp{D}{I}{\Set}$ be an $I$-diagram in $\Set$ with colimit $L$.
For every $l\in L$ the comma category $\comcat{(\gc D)}{l}$ is a $\lambda$-filtered poset. 
\end{lemma}
\begin{proof}
For every $i\in I$ let $\morp{\iota_i}{D(i)}{L}$ be the map defining $L$ as colimit.
The category $\gc D$ is a poset: its objects are couples $(i,x)$ where $i\in I$ and $x\in D(i)$; $(i,x)\leq (j,y)$ if and only if 
$i\leq j$ and $d_{ij}(x)=y$, where $d_{ij}$ is the unique arrow $\morp{d_{ij}}{D(i)}{D(j)}$ in the diagram $D$.

Given a non-empty set $A$ of cardinality less than $\lambda$ and a collection of objects $\{(i_a,x_a)\}_{a\in A}$ in $\comcat{(\gc D)}{l}$,
fix an element $a\in A$;
for every $b\in A$, since $e_{i_a}(x_a)=e_{i_b}(x_b)=l$ there must exist $(j_{b},y_b)$ bigger than $(i_a,x_a)$ and $(j_b,x_b)$.
Now the collection $\{j_b\}_{b\in A}$ must have an upper bound $z$ in $I$; furthermore $d_{i_b,z}(i_b)=d_{j_b,z}(y_b)=d_{i_a,z}(x_a)$
for every $b\in A$,
thus $(z,d_{i_a,z}(x_a))$ is an upper bound for the collection $\{(i_a,x_a)\}_{a\in A}$.  
\end{proof}

\begin{prop}\label{col.small}
Let $\lambda$ be a regular cardinal, $\mathcal{C}$ be a category and $\ofam{F}$ a discrete $\mathcal{C}$-family. Suppose $c\in \mathcal{C}$ and $v\in \Mm$ are $\lambda$-small in their respective categories and let $s\in \Cl{\ofam{F}(c)}$; then
$(c,\iota_s(v))$ is small in $\ToBi{\ofam{F}}{\Mm}$.   
\end{prop}

\begin{proof}
Consider a small $\lambda$-directed poset $I$ and let $\morp{D}{I}{\ToBi{\ofam{F}}{\Mm}}$ an $I$-diagram.
Let $\morp{\mathbf{D}}{I}{\Oper}$ be its associated operadic $I$-family.
Let $L=\varinjlim (\pi_{\ofam{F}}D)\in \mathcal{C}$; then $\varinjlim D\cong (L,X)$ where $X$ is the $\ofam{F}(L)$ algebra associated to
the left Kan extension of $\morp{D}{\FSm{\gco \mathbf{D}}}{\Mm}$ along the canonical morphism
$\morp{u}{\FSm{\gco \mathbf{D}}}{\FSm{\ofam{F}(L)}}$
\[
 \xymatrix{\FSm{\gco \mathbf{D}} \ar[r]^D \ar[d]^{u} & \Mm\\
           \FSm{\ofam{F}(L)} \ar[ur]_{X}& .}
\]
A morphism $\morp{\mathbf{f}}{(c,\iota_s(v))}{(L,X)}$ is just a couple $(f,g)$ where
$\morp{f}{c}{L}$ and $\morp{g}{v}{X(\ofam{F}(f)(s))}$.
Now
$X(\ofam{F}(f)(s))=\varinjlim_{((i,y),\alpha)\in \comcat{u}{\ofam{F}(f)(s)}} D(i)(y)$  
thus $g$ factor via some $D(i)(y)$, since $\comcat{u}{\ofam{F}(f)(s)}$ is $\lambda$-directed by Lemma \ref{lambda.dir}; 
since $c$ is $\lambda$-small it is not restrictive to assume that $f$ factors via $\pi_{\ofam{F}}D(i)$;
it follows that $(f,g)$ factors via $D(i)$.
\end{proof}
\noindent 
Given a discrete $\cat{C}$-family $\ofam{F}$ and a $\ofam{F}$-algebra $(c,x)\in \AlgF{\ofam{F}}{\Mm}$, recall that we denote by
 $\morp{\iota_x}{\Mm^{\ofam{F}(c)}}{\Mm}$ the left adjoint of the projection on the $x$-component. For every $c\in \cat{C}$ we will denote by $\ast_a$ the initial object in 
$\Alg{\ofam{F}(a)}{\Mm}\cong \Mm^{\ofam{F}}$. Recall the definition of (strong) generator from  \cite[\S 0.6]{AR94}.

\begin{prop}\label{strong.gen}
Let $\mathcal{C}$ be a category and let $\ofam{F}$ be a discrete $\mathcal{C}$-family. Suppose that $\mathcal{A}$ is a strong generator for $\cat{C}$, $\mathcal{B}$ is a strong generator for $\Mm$ and
$G$ is a weakly initial set of colours for $\ofam{F}$ (def. \ref{def.weakin}). Then the set
\[
\mathcal{G}=\{(a,\ast_a) \mid a\in \mathcal{A}\}\cup \{(c,\iota_x(b)) \mid (c,x)\in G,\ b\in \mathcal{B} \}
\]
is a strong generator for $\ToBi{\ofam{F}}{\Mm}$.
\end{prop}

\begin{proof}
We first prove that $\mathcal{G}$ is a generator.\\
Suppose let $(z,Z)$ and $(w,W)$ be two objects of $\ToBi{\ofam{F}}{\Mm}$ and let $(f,\alpha),(f,\alpha')\colon (z,Z) \rightrightarrows (w,W)$ be
two distinct morphism between them. If $f\neq f'$ then there exist $a\in \mathcal{A}$ and $\morp{k}{a}{z}$ such that $fk\neq f'k$, thus $(f,\alpha)(k,\ast)\neq (f',\alpha)(k,\ast)$.

On the other hand, if $f=f'$ then there must exist $s\in \ofam{F}(W)$ such that $\morp{\alpha_s}{Z(s)}{f^*(W)(s)}$ is different
from $\morp{\alpha'_s}{Z(s)}{f^*(W)(s)}$; thus there exist $b\in \mathcal{B}$ and $\morp{\beta}{b}{Z(s)}$ such that 
$\alpha_s\beta\neq \alpha'_s\beta$. It follows that $(f,\alpha)(\id_z,\iota_s\beta)\neq (f,\alpha')(\id_z,\iota_s\beta)$.

We are left to prove that $\mathcal{G}$ is a strong generator.
Let $(z,Z)\in\ToBi{\ofam{F}}{\Mm}$ and let $\morp{(i,\gamma)}{(k,K)}{(z,Z)}$ be a subobject of $(z,Z)$; it follows that  $i$ is a subobject
of $z$ and $\gamma$ is a subobject of $i^*(Z)$. Suppose that for every $(d,y)\in \mathcal{G}$ every morphism $\morp{(j,\delta)}{(d,y)}{(z,Z)}$ factors via $(i,\gamma)$.

Given that for every $a\in \mathcal{A}$ and every morphism $\morp{m}{a}{z}$ there is a morphism $\morp{(m,\ast)}{(a,\ast_a)}{(z,Z)}$,
$i$ has to be an isomorphism.

For every colour $s\in \ofam{F}(z)$ there is a $(c,x)\in G$ and a $\morp{q}{c}{z}$ such that $q(x)=s$. For every $b\in \mathcal{B}$
each morphism $\morp{p}{b}{i^*(Z(s))}$ produces a morphism $\morp{(q,\iota_x p)}{(c,\iota_x(b))}{(z,Z)}$, that, by hypothesis has to factor
via $(i,\gamma)$; this implies that $\gamma_{i^{-1}(s)}$ is an isomorphism for every $s\in \ofam{F}(z)$, that is $\gamma$ is an isomorphism.    
\end{proof}

\begin{cor}
Let $\lambda$ be a regular cardinal. Suppose that $\mathcal{C}$ and $\Mm$ are $\lambda$-locally presentable, and $\morp{\ofam{F}}{\mathcal{C}}{\Set}$ is a discrete
$\mathcal{C}$-family with a weakly initial set of colours, then the category of $\ofam{F}$-collection in $\Mm$ is $\lambda$-locally presentable.
\end{cor}
\begin{proof}
The category $\ToBi{\ofam{F}}{\Mm}$ is cocomplete and has a strong generator formed by  $\lambda$-small objects by Propositions 
\ref{strong.gen} and \ref{col.small},
 thus it is locally presentable by \cite[Theorem 1.20]{AR94}. 
\end{proof}

\subsection{Finitary operadic families}
Now consider a general operadic $\mathcal{C}$-family $\ofam{F}$; suppose that $\mathcal{C}$ is cocomplete.

\begin{lemma}
The adjunction \ref{freeforgfib} is monadic. 
\end{lemma}
\begin{proof}
This is a straightforward generalisation of Proposition 3.3 \cite{C14} (where $\mathcal{C}=\Set$).
\end{proof}

\begin{defi}\label{def.finfam}
An operadic $\mathcal{C}$-family $\ofam{F}$ is \emph{finitary} if it preserves filtered colimits. 
\end{defi}

\begin{lemma}
Suppose that $\ofam{F}$ is finitary, then the adjunction \ref{freeforgfib} is finitary.
\end{lemma}
\begin{proof}
Take $I$ a directed poset and a $I$-diagram $\morp{D}{I}{\ToBi{\ofam{F}}}{\Mm}$ and let $d=\pi_\ofam{F} D$. 
Consider the operadic
Grothendieck construction $\gco{(\ofam{F}d)}$, according to prop. \ref{prop.diagop2} its algebras in $\Mm$ are the $I$-diagram in 
$\AlgF{\ofam{F}}{\Mm}$ that lies over $d$.
 
There is a commutative diagram of operads
\[
 \xymatrix{ \gco (\disf{\ofam{F}}d)\ar[r]\ar[d]^-{\disf{l}} & \gco{(\ofam{F}d)} \ar[d]^-{l}\\
            \varinjlim (\disf{\ofam{F}}d)\ar[r] & \varinjlim (\ofam{F}d). 
}
\]

$D$ is a $\gco{(\ofam{F}d)}$-algebra, so it can be represented by a strong monoidal functor 
$\morp{\widehat{D}}{\FSm{\gco{(\ofam{F}d)}}}{\Mm}$; the colimit of $D$ is then represented by the left Kan extension of $\widehat{D}$ along $w_!l$ (\emph{cf.} appendix \ref{sec: alg as Kan extensions});
similarly the colimit of $\disf{D}$ is represented by $\mathrm{Lan}_{w_!\disf{l}} \hat{D}$, the left Kan extension of $\widehat{D}\widehat{\Clf}$ along $w_!\disf{l}$.

All we have to do is to prove that the canonical natural transformation
\[
 \ntra{\nu}{\mathrm{Lan}_{\disf{l}} (\disf{\hat{D}})}{\disf{(\mathrm{Lan}_l \hat{{D}})}}
\]
is a natural isomorphism. 

Let $C=\varinjlim (\Clf\circ \ofam{F})$ be the set of colours of $\varinjlim (\disf{\ofam{F}}d)$ (and $\varinjlim (\ofam{F}d)$). 
It is sufficient to prove that the morphism
\[
 \lmorp{\nu_c}{\varinjlim_{(a,f)\in\comcat{\FSm{\disf{l}}}{c}} D(a)}{\varinjlim_{(a,f)\in\comcat{(\FSm{l})}{c}} D(a)}
\]
is an isomorphism for every $c\in C$; this morphism is induced by the inclusion of the indexing categories
\[
 \morp{\iota_c}{\comcat{\FSm{\disf{l}}}{c}}{\comcat{(\FSm{l})}{c}}
\]
thus it is sufficient to show that $\comcat{\FSm{\disf{l}}}{c}$ is final in $\comcat{(\FSm{l})}{c}$.
We use the description of $\comcat{(\FSm{l})}{c}$ given in \S\ \ref{sec.feypar} (Remark \ref{rmkpostipar} therein).
For every $i\in I$ let $\morp{e_i}{\ofam{F}(i)}{\varinjlim \ofam{F}}$ be the $i$-morphism defining the colimit; for every $s\leq t$ in $I$
let $\morp{d_{st}}{\ofam{F}(s)}{\ofam{F}(t)}$ be the unique such arrow in the diagram $D$.

The objects of $\comcat{(\FSm{l})}{c}$ in the image of $i_c$ are the ones of the type $((i,x),\id_c)$, for some $i\in I$ and $x\in \ofam{F}(i)$ (such that $e_i(x)=d$); the morphisms in the image of $\iota_c$ are the ones of the form $\morp{(\id_{\fcar{1}},\id)}{((i,x),\id_c)}{((i',x'),\id_c)}$ (with $i\leq i'$).
where $f$ is a morphism in $I$.

To prove that $\comcat{\FSm{\disf{l}}}{c}$ is final in $\comcat{(\FSm{l})}{c}$ we have to show that for every 
$y=(\{i_k,c_k\}^n_{k=1},o)$ in $\comcat{(\FSm{l})}{c}$ the category $\comcat{y}{(\comcat{\FSm{\disf{l}}}{c})}$ is non-empty and connected.

The operation $o$ belongs to $(\varinjlim \ofam{F})(\{e_{i_k}(c_k)\}^n_{i=1};c)$; since $I$ is filtered there exists
$s\in I$ such that $s\geq i_k$ for every $1\leq k\leq n$, a colour $c'\in \ofam{F}(s)$ such that $e_s(c')=c$ and an operation $o'\in \ofam{F}(s)(\{d_{si_k}(c_k)\}^n_{k=1}; c')$ such that $e_s(o')=o$ (\emph{cf.} \ref{rmk:filtcol}).

Therefore there is a morphism
\[
 \morp{(\ast,o')}{y}{((s,c'),\id_{c'})}
\]
which defines an object in $\comcat{y}{(\comcat{\FSm{\disf{l}}}{c})}$.
To show that $\comcat{y}{(\comcat{\FSm{\disf{l}}}{c})}$ is connected suppose we have another object
\[
 \morp{(\ast,o'')}{y}{((t,c''),\id_{c''});}
\]
 we can find $v\in I$ such that $v\geq s$ and $v\geq t$ and $d_{sv}(o')=d_{sv}(o'')=o'''$;
$(\ast,o''')$ is an object in $\comcat{y}{(\comcat{\FSm{\disf{l}}}{c})}$ connected to both $(\ast,o')$ and $(\ast,o'')$. 
\end{proof}

\begin{rmk}\label{rmk:filtcol}
We remark that, in the proof of the previous lemma we used an explicit description the filtered colimits in $\Oper$, namely we have used the fact that the forgetful functor $\morp{U}{\Oper}{\EMGraph{\Set}}$ reflects filtered colimits; for a proof of this fact and the description of filtered colimits in $\EMGraph{\Set}$ we refer the reader to \cite[App. A]{C14}. 
\end{rmk}
\begin{cor}\label{loc.pres.par}
Let $\mathcal{C}$ and $\Mm$ be $\lambda$-locally presentable categories for a regular cardinal $\lambda$. Suppose $\ofam{F}$ is a finitary operadic $\mathcal{C}$-family
with a weakly initial set of colours, then $\ToBi{\ofam{F}}{\Mm}$ is $\lambda$-locally presentable.  
\end{cor}
\begin{proof}
Under our assumptions $\ToBi{\disf{\ofam{F}}}{\Mm}$, the category of $\ofam{F}$-collection in $\Mm$, is locally presentable and adjunction
\ref{freeforgfib} is monadic and finitary, therefore $\ToBi{\ofam{F}}{\Mm}$ is $\lambda$-locally presentable by \cite[Satz 10.3]{GU71}.  
\end{proof}

\appendix
\section{Operadic morphisms}\label{sec: alg as Kan extensions}
The goal of this section is to give a more explicit description of the left adjoint of the adjunction between the categories of algebras induced by a morphism of operads. 

The considerations done in this section should be put in the broadest context presented in \cite{MT08}, from which this section took inspiration.

In this section we will make use of enriched coends and left Kan extensions, we refer the reader to \cite{Ke82} for the notation.

Fix a bicomplete closed symmetric monoidal category $\Mm$.\\ 
Suppose $\morp{u}{\prp{A}}{\prp{B}}$ is a morphism of coloured $\Mm$-PROPs and let $(\cat{M}\otimes,\Un)$ be a cocomplete symmetric $\Mm$-algebra (\S \ref{sec:Malgebras}).
Consider $\cat{M}^{\prp{A}}$, the category of $\Mm$-functors from (the underlying symmetric monoidal $\Mm$-category of) $\prp{A}$ to 
$\cat{M}$; identifying $\Alg{\prp{A}}{\cat{M}}$ with the category of strong monoidal $\Mm$-functors from $\prp{A}$ to $\cat{M}$ and
symmetric monoidal $\Mm$-natural transformations between them there is a natural functor $\Alg{\prp{A}}{\cat{M}}\to \cat{M}^{\prp{A}}$ which forgets the monoidal structure; in between there is the category $\LaxM{\prp{A}}{\cat{M}}$ of (lax) symmetric monoidal $\Mm$-functor and symmetric monoidal $\Mm$-natural transformations, of which $\Alg{\prp{A}}{\cat{M}}$ is a full subcategory
\[
\Alg{\prp{A}}{\cat{M}} 
\hookrightarrow
\LaxM{\prp{A}}{\cat{M}} 
\longrightarrow
\cat{M}^{\prp{A}}. 
\]
Restriction and Left Kan extension along $u$ define an adjunction
\begin{equation}\label{eq:rest left kan}
\xymatrix{\cat{M}^{\prp{A}} \ar@<3pt>[r]^-{u_!} & \ar@<3pt>[l]^-{u^*} \cat{M}^{\prp{B}}.} 
\end{equation}

For every $G\in \cat{M}^{\prp{B}}$, a symmetric (strong) monoidal structure on $G$ induces a symmetric (strong) monoidal structure on $u^*G$
in a natural way; in other words $u^*$ extends to functors $\morp{u^*}{\LaxM{\prp{B}}{\cat{M}}}{\LaxM{\prp{A}}{\cat{M}}}$ and
$\morp{u^*}{\Alg{\prp{B}}{\cat{M}}}{\Alg{\prp{A}}{\cat{M}}}$.

The behaviour of $u_!$ on monoidal functors is more subtle; for every $F\in \cat{M}^{\prp{A}}$ and $\str{b}\in \prp{B}$, the $b$-component of $u_!F$ is isomorphic to the coend:
\[
u_!F(\str{b})= \LKan{u} F (\str{b}) \cong\gc^{\str{a}\in \prp{A}} \prp{B}(u(\str{a}),\str{b})\otimes F(\str{a}).
\]

Suppose that $F$ is endowed with a symmetric monoidal structure $(\phi,\psi)$;
from the commutative diagrams
\[
\begin{tikzpicture}[scale=0.9]
\matrix(a)[matrix of math nodes,
row sep=4em, column sep=4em,
text height=1.5ex, text depth=0.25ex]
{\prp{B}\times\prp{B} & \prp{A}\times\prp{A} & \cat{M}\times \cat{M} \\
    \prp{B} & \prp{A} & \cat{M} \\};
\path (a-1-2) edge[->] node[name=m1,pos=.5,label=$u\times u$] {} (a-1-1);
\path (a-2-2) edge[->] node[name=m2,pos=.5,label=below:$u$] {} (a-2-1);
\path (a-2-2) edge[->] node[name=m4,pos=.5,label=below:$F$] {} (a-2-3);
\path (a-1-2) edge[->] node[name=m3,pos=.5,label=$F\times F$] {} (a-1-3);
\path (a-1-2) edge[->] node[pos=.5,label=right:$\boxtimes$] {} (a-2-2);
\path (a-1-1) edge[->] node[pos=.5,label=left:$\boxtimes$] {} (a-2-1);
\path (a-1-3) edge[->] node[pos=.5,label=right:$\otimes$] {} (a-2-3); 
\path (m1) -- node[name=c1,pos=.2] {} node[name=c2,pos=.8] {} (m2);
\path (m3) -- node[name=c3,pos=.3] {} node[name=c4,pos=.7] {} (m4);
\path (c1) --
node[pos=.5] {$\cong$}
(c2);
\path (c3) edge[double distance=2pt,-implies]
node[pos=.6,right=2pt,xshift=-2pt] {$\phi$}
(c4);

\matrix(b)[matrix of math nodes,
row sep=2em, column sep=4em,
text height=1.5ex, text depth=0.25ex] at (7.5,0)
{ & \mathbf{1} &  \\
    \prp{B} &  & \cat{M} \\
    & \prp{A} & \\ };
\path (b-1-2) edge[->] node[name=d1,scale=0.7,pos=.5,label=$\Un$] {} (b-2-1);
\path (b-1-2) edge[->] node[name=e1, scale=0.7,pos=.5,label=$\Un$] {} (b-2-3);
\path (b-3-2) edge[->] node[name=d2,scale=0.7,pos=.5,label=below:$u$] {} (b-2-1);
\path (b-3-2) edge[->] node[name=e2,scale=0.7,pos=.5,label=below:$F$] {} (b-2-3);
\path (b-1-2) edge[->] node[scale=0.7,pos=.5,label=left:$\Un$] {} (b-3-2);
\path (d1) --
node[pos=.5] {$\cong$}
(d2);
\path (e1) -- node[name=f1,pos=.3] {} node[name=f2,pos=.7] {} (e2);
\path (f1.west) edge[double distance=2pt,-implies]
node[pos=.6,label=left:$\psi$] {}
(f2.west);
\end{tikzpicture}
\]
we get a natural transformation
\begin{equation}\label{eq: monoidal str1}
\widetilde{\phi} \colon \otimes (u_!F\times u_! F)\cong  \LKan{u\boxtimes u} (\otimes(F\times F)) \rightarrow (u_! F)\boxtimes 
\end{equation}
and a morphism
\begin{equation}\label{eq: monoidal str2}
\widetilde{\psi} \colon \Un \longrightarrow u_! F([])
\end{equation}
that define a (lax) symmetric monoidal structure on $u_!F$. This assignment is natural in $F$ and the functor $f_!$ sends symmetric monoidal $\Mm$-natural transformation to symmetric monoidal $\Mm$-natural transformation, i.e. it extends to a functor $\morp{f_!}{\LaxM{\prp{A}}{\Mm}}{\LaxM{\prp{B}}{\Mm}}$ which is left adjoint to $u^*$.

Hence we have the following commutative diagram:
\begin{equation}\label{eq: big alg as kan diagram}
\xymatrix{
    \Alg{\prp{A}}{\cat{M}} 
    \ar@<3pt>[d]^-{}
    &
    \Alg{\prp{B}}{\cat{M}} 
    \ar@<3pt>[d]^-{}
    \ar@<3pt>[l]^-{u^*}
    \\
    \LaxM{\prp{A}}{\cat{M}} 
    \ar@<3pt>[d]^-{}
    \ar@<3pt>[r]^-{u_!} &
    \LaxM{\prp{B}}{\cat{M}} 
    \ar@<3pt>[d]^-{}
    \ar@<3pt>[l]^-{u^*}
    \\
    \cat{M}^{\prp{A}} 
    \ar@<3pt>[r]^-{u_!} &
    \cat{M}^{\prp{B}} 
    \ar@<3pt>[l]^-{u^*}
}
\end{equation}
The functor $u_!$ does not always restrict to $\morp{u_!}{\Alg{\prp{A}}{\cat{M}}}{\Alg{\prp{B}}{\cat{M}}}$; in fact, even when the monoidal structure on $A$ is strong the morphisms  
(\ref{eq: monoidal str1}) and (\ref{eq: monoidal str2}) are not isomorphisms in general, i.e. the induced monoidal structure on $u_!A$ is not strong; but this is the only obstruction that does not allow
us to restrict $u_!$.

To understand for which $u$ this obstruction vanishes, it is useful to look closer to 
(\ref{eq: monoidal str1}) and (\ref{eq: monoidal str2}).

For every $\str{p},\str{q}\in \prp{B}$ set 
\[
\gfun{W_{\str{p}}}{\opc{\prp{A}}}{\Mm}{\str{s}}{\prp{B}(u(\str{s}),\str{p})},
\  
\gfun{W_{\str{p},\str{q}}} {\opc{\prp{A}}\times \opc{\prp{A}}}{\Mm}{(\str{s},\str{t})}{\prp{B}(u(\str{s}),\str{p})\otimes \prp{B}(u(\str{t}),\str{q});}
\]
there are natural transformations
\begin{equation}\label{eq: weights kan ext}
\begin{tikzpicture}[baseline=(current  bounding  box.center)]
\matrix(a)[matrix of math nodes,
row sep=4em, column sep=4em,
text height=1.5ex, text depth=0.25ex]
{\opc{\prp{A}}\times \opc{\prp{A}} & \Mm \\
    \opc{\prp{A}} & \\};
\path (a-1-1) edge[->] node[pos=.5,above,scale=0.8] {$W_{\str{p},\str{q}}$} node[pos=.2,name=m1] {}  (a-1-2);
\path (a-1-1) edge[->] node[pos=.5,name=m2,scale=0.8,left] {$\opc{\boxtimes}$} (a-2-1);
\path (a-2-1) edge[->] node[pos=.6,name=m3,below=10pt,scale=0.7] {$W_{\str{p}\boxtimes \str{q}}$} (a-1-2);
\path (m1) -- node[pos=.1,name=n1] {} node[pos=.75,name=n2] {} (a-2-1);
\path (n1) 
edge[double distance=2pt,-implies]
node[pos=.6,above=2pt,xshift=-2pt] {$\alpha$}
(n2);
\end{tikzpicture}
\ \ \ 
\begin{tikzpicture}[baseline=(current  bounding  box.center)]
\matrix(a)[matrix of math nodes,
row sep=4em, column sep=4em,
text height=1.5ex, text depth=0.25ex]
{\mathbf{1} & \Mm \\
    \opc{\prp{A}} & \\};
\path (a-1-1) edge[->] node[pos=.5,above,scale=0.8] {$\Un$} node[pos=.5,name=m1] {}  (a-1-2);
\path (a-1-1) edge[->] node[pos=.5,name=m2,scale=0.8,left] {$[]$} (a-2-1);
\path (a-2-1) edge[->] node[pos=.6,name=m3,below=10pt,scale=0.7] {$W_{[]}$} (a-1-2);
\path (m1) -- node[pos=.1,name=n1] {} node[pos=.75,name=n2] {} (a-2-1);
\path (n1) 
edge[double distance=2pt,-implies]
node[pos=.6,above=3pt,xshift=-3pt] {$\lambda$}
(n2);
\end{tikzpicture}
\end{equation}
induced by the monoidal structure of $\prp{A}$, $\prp{B}$ and $u$.

The $(\str{p},\str{q})$-component of $\widetilde{\phi}$ is
isomorphic to
\[
\gc^{\str{s},\str{t}\in \prp{A}} \prp{B}(u(\str{s}),\str{p})\otimes \prp{B}(u(\str{t}),\str{q})\otimes F(\str{s})\otimes F(\str{t})
\rightarrow
\gc^{\str{z}\in \prp{A}} \prp{B}(u(\str{s}),\str{p}\boxtimes \str{q})\otimes F(\str{z})
\]
that, if $F$ is strong monoidal, is isomorphic to the map between weighted colimits
\[
\lmorp{\alpha^*}
{
    \Wcolim{\cat{A}\times \cat{A}}{W_{\str{p},\str{q}}}{(F\circ\boxtimes)}
}
{
    \Wcolim{\cat{A}}{W_{\str{p}\boxtimes \str{q}}}{F}
}
\]  

Similarly (\ref{eq: monoidal str2}) is isomorphic to the map between weighted colimits
\[
\Un \cong \Wcolim{}{\Un}{F([])} \longrightarrow
\Wcolim{\prp{A}}{W_{[]}}{F}
\]
induced by $\lambda$ in (\ref{eq: weights kan ext}).

This leads to the following definition:

\begin{defi}\label{def:operadic}(\emph{cf.} \cite[\S 3, def. 9]{MT08})
 A morphism $\morp{u}{\ope{A}}{\ope{B}}$ between $\Mm$-PROPs (or more in general, between small symmetric monoidal $\Mm$-categories) is \emph{operadic}\index{operadic morphism} if  $\alpha$ and $\lambda$ in diagram (\ref{eq: weights kan ext}) exhibit  
    \begin{itemize}
        \item[-] $W_{\str{p}\boxtimes\str{q}}$ as the left Kan extension of $W_{\str{p},\str{q}}$ along $\opc{\boxtimes}$ and
        \item[-] $W_{[]}$ as the left Kan extension of $\Un$ along $[]$.    
    \end{itemize}
    for every $\str{p},\str{q}\in \prp{B}$. 
\end{defi}

\noindent The previous discussion can be subsumed in the following proposition: 
\begin{prop}\label{prop: algebraic kan trick}
    Suppose that $\morp{u}{\prp{A}}{\prp{B}}$ is an operadic morphism of $\Mm$-PROPs. 
    For every cocomplete symmetric monoidal $\Mm$-category $\cat{M}$, the functor $u_!$ in (\ref{eq: big alg as kan diagram}) restricts to a left adjoint for $\morp{u^*}{\Alg{\prp{A}}{\cat{M}}}{\Alg{\prp{B}}{\cat{M}}}$. 
\end{prop}

\begin{proof}
    Under our hypothesis, for every $\Mm$-functor $\morp{F}{\cat{A}}{\cat{M}}$
    and every $\str{p},\str{q}\in \prp{B}$ the canonical maps $\morp{\alpha^*}{\Wcolim{\cat{A}\times \cat{A}}{W_{\str{p},\str{q}}}{(F\circ\boxtimes)}}{\Wcolim{\cat{A}}{W_{\str{p}\boxtimes\str{q}}}{F}}$ and $\morp{\lambda^*}{\Un\otimes F([])}{\Wcolim{\prp{A}}{W_{[]}}{F}}$
    are isomorphisms by the dual of \cite[Theorem 4.63]{Ke82}.
    
    In particular, if $F$ is a strong monoidal functor, (\ref{eq: monoidal str1}) and (\ref{eq: monoidal str2}) are isomorphisms.
\end{proof} 
\begin{rmk}
    The operadic condition of definition \ref{def:operadic} for a strong monoidal morphism $\morp{u}{\prp{A}}{\prp{B}}$ can be formulated in a different way (\emph{cf.} \cite[\S 1.1.2]{KW14}). Since $\prp{A}$ and $\prp{B}$ are symmetric monoidal, the functor category $\Mm^{\opc{\prp{A}}}$ (resp. $\Mm^{\opc{\prp{B}}}$) has a symmetric monoidal structure given by the Day convolution. To ask that $\morp{u}{\prp{A}}{\prp{B}}$ is equivalent to ask that the restriction functor
    $\morp{u^*}{\Mm^{\opc{\prp{B}}}}{\Mm^{\opc{\prp{A}}}}$ is strong monoidal (with respect to the Day convolution).
\end{rmk}
\noindent For every $\str{p}\in \prp{B}$ of length $n$ let $\bigotimes_{i=1}^n W_{p_i}\colon \prod_{i=1}^n \opc{\prp{A}} \to \Mm$ be the composition of 
\[\morp{\prod_{i=1}^n W_{p_i}}{\prod_{i=1}^n \opc{\prp{A}}}{\prod^{n}_{i=1} \Mm}\] with the $n$-fold tensor product 
$\morp{\otimes}{\prod^{n}_{i=1} \Mm}{\Mm}$. There is a canonical natural transformation $\ntra{\lambda^{\str{p}}}{\bigotimes_{i=1}^n W_{p_i}}{W_{\str{p}}\opc{\boxtimes}}$.  Operadicity can also be formulated in the following way.

\begin{prop}\label{prop: different operadic}
A morphism of of $\Mm$-PROPs $\morp{u}{\prp{A}}{\prp{B}}$ is operadic if and only if for every $\str{p}\in \prp{B}$ the natural transformation
$\lambda^{\str{p}}$ exhibits $W_{\str{p}}$ as the left Kan extension of $\bigotimes_{i=1}^n W_{p_i}$ along $\opc{\boxtimes}$, i.e. if
and only if  
the canonical map:
\begin{equation}\label{eq:operadic for props}
\gc^{\str{s}_1,\dots,\str{s}_n\in \prp{A}} \prp{A}(\str{z},\str{s}_1\boxtimes\dots\boxtimes\str{s}_n)\otimes 
      \prp{B}(u(\str{s}_1),p_1)\otimes \dots \otimes \prp{B}(u(\str{s}_n),p_n) \longrightarrow \prp{B}(u(\str{z}),\str{p})
\end{equation}
is an isomorphism for every $\str{z}\in \prp{A}$.
\end{prop}
\begin{proof}
 Suppose $u$ is operadic.
We are going to prove that $W_{\str{p}}$ is the left Kan extension of $\bigotimes_{i=1}^n W_{p_i}$ along $\opc{\boxtimes}$
for every $\str{p}\in \prp{B}$ by induction on the length of $\str{p}$.

If $\str{p}=[]$ then $\lambda^{[]}$ coincides with $\lambda$ thus the statement follows from the hypothesis that $u$ is operadic.
If $\card{\str{p}}=1$ the statement is trivial.
Suppose that $\str{p}=(p_1,\dots,p_n)$ and that the statement is true for every object of $\prp{B}$ of length smaller than $n$; let $\str{p}'=(p_2,\dots,p_n)$; then
in the diagram  
\[
\begin{tikzpicture}[scale=0.5]
\matrix(a)[matrix of math nodes,
row sep=3.7em, column sep=7em,
text height=1.5ex, text depth=0.25ex]
{\opc{\prp{A}}\times \prod_{i=1}^n\opc{\prp{A}} & \Mm \\
    \opc{\prp{A}} \times \opc{\prp{A}} & \\
     \opc{\prp{A}} & \\};
\path (a-1-1) edge[->] node[name=m1,pos=.22] {} 
                       node[name=l1,pos=.5,style={scale=0.8},yshift=10pt] {$\bigotimes_{i=1}^n W_{p_i}$} (a-1-2);
\path (a-2-1) edge[bend right=10,->] node[name=m2,pos=.3] {} 
                                     node[name=l2,pos=.3,style={scale=0.8},below,xshift=22pt] {$W_{p_1}\otimes W_{\str{p}'}$} (a-1-2);
\path (a-3-1) edge[bend right,->] node[name=m3,pos=.3] {} 
				  node[name=l3,pos=.5,style={scale=0.8},xshift=8pt,yshift=-11pt] {$W_{\str{p}}$} (a-1-2);
\path (a-1-1) edge[->] node[name=n1,pos=.5,style={scale=0.8},xshift=-22pt] {$\id \times \opc{\boxtimes}$} (a-2-1);
\path (a-2-1) edge[->] node[name=n2,pos=.5,style={scale=0.8},xshift=-12pt] {$\opc{\boxtimes}$} (a-3-1);
\path (m1) -- node[name=c1,pos=.1] {} node[name=c2,pos=.7] {} (m2);
\path (m2) -- node[name=c3,pos=.2] {} node[name=c4,pos=.8] {} (m3);
\path (c1) edge[double distance=2pt,-implies]
node[pos=.6,left=2pt,yshift=2pt,style={scale=0.8}] {$id\otimes \lambda^{\str{p}'}$}
(c2);
\path (c3) edge[double distance=2pt,-implies]
node[pos=.6,left=2pt,yshift=2pt,style={scale=0.8}] {$\alpha_{\str{p_i},\str{p}'}$}
(c4);
 \end{tikzpicture}
\]
the top triangle is a left Kan extension by inductive hypothesis and the bottom triangle is a left Kan extension since $u$ is operadic; it follows that the outer triangle exhibits $W_{\str{p}}$ as the left Kan extension of $\bigotimes_{i=1}^n W_{p_i}$ along
$\opc{\boxtimes}$. 

Conversely suppose that $\lambda^{\str{p}}$ exhibits $W_{\str{p}}$ as left Kan extension for every $\str{p}\in \prp{B}$; since 
$\lambda$ in diagram (\ref{eq: weights kan ext}) is isomorphic to $\lambda^{[]}$ we only need to check the first condition of definition \ref{def:operadic}. 

Given $\str{p},\str{q}\in \prp{B}$ of length $n$ and $m$ respectively, we have a commutative diagram:
\[
\begin{tikzpicture}[scale=0.5]
\matrix(a)[matrix of math nodes,
row sep=3.7em, column sep=8em,
text height=1.5ex, text depth=0.25ex]
{\prod_{i=1}^n\opc{\prp{A}}\times \prod_{i=1}^m\opc{\prp{A}} & \Mm \\
    \opc{\prp{A}} \times \opc{\prp{A}} & \\
     \opc{\prp{A}} & \\};
\path (a-1-1) edge[->] node[name=m1,pos=.15] {} 
                       node[name=l1,pos=.5,style={scale=0.8},yshift=10pt] {$\bigotimes_{i=1}^n W_{p_i}\otimes \bigotimes_{j=1}^m W_{q_j}$} (a-1-2);
\path (a-2-1) edge[bend right=10,->] node[name=m2,pos=.3] {} 
                                     node[name=l2,pos=.3,style={scale=0.8},below,xshift=22pt] {$W_{\str{p},\str{q}}$} (a-1-2);
\path (a-3-1) edge[bend right,->] node[name=m3,pos=.3] {} 
				  node[name=l3,pos=.5,style={scale=0.8},xshift=8pt,yshift=-11pt] {$W_{\str{p}\boxtimes \str{q}}$} (a-1-2);
\path (a-1-1) edge[->] node[name=n1,pos=.5,style={scale=0.8},xshift=-22pt] {$\opc{\boxtimes} \times \opc{\boxtimes}$} (a-2-1);
\path (a-2-1) edge[->] node[name=n2,pos=.5,style={scale=0.8},xshift=-12pt] {$\opc{\boxtimes}$} (a-3-1);
\path (m1) -- node[name=c1,pos=.1] {} node[name=c2,pos=.7] {} (m2);
\path (m2) -- node[name=c3,pos=.2] {} node[name=c4,pos=.8] {} (m3);
\path (c1) edge[double distance=2pt,-implies]
node[pos=.6,left=2pt,yshift=2pt,style={scale=0.8}] {$\lambda^{\str{p}} \otimes \lambda^{\str{q}}$}
(c2);
\path (c3) edge[double distance=2pt,-implies]
node[pos=.6,left=2pt,yshift=2pt,style={scale=0.8}] {$\alpha_{\str{p},\str{q}}$}
(c4);
 \end{tikzpicture}
\]
The composition of $\lambda^{\str{p}} \otimes \lambda^{\str{q}}$ with $\alpha_{\str{p},\str{q}}$ is isomorphic to $\lambda^{\str{p}\boxtimes \str{q}}$, thus the outer triangle and the top triangle are left Kan extensions; this implies that the
bottom triangle is a left Kan extension as well; hence $u$ is operadic.
\end{proof}
\noindent 
Our next goal is to show that every morphism of $\Mm$-PROPs generated by a morphism of $\Mm$-operads is operadic.

More in details, suppose $\morp{u}{\ope{O}}{\ope{P}}$ is a morphism of $\Mm$-operads; by abuse of notation let $\morp{u}{w_!\ope{O}}{w_!\ope{P}}$
be the induced map of PROPs.

Recall the description of the $\Mm$-PROP $w_!\ope{P}$ we gave in \S \ref{operandcat}; as a shorthand, for every $\str{s},\str{t}\in w_!\ope{P}$ such that $\card{\str{s}}=n$ and $\card{\str{t}}=m$ and every function $\morp{f}{\fcar{n}}{\fcar{m}}$
we will denote by $\ope{P}(\str{s},\str{t})_f$ the tensor product $\bigotimes_{i=1}^m \ope{P}(\str{s}_{f^{-1}(i))},t_i)$, i.e. the 
component of $w_!\ope{P}(\str{s},\str{t})$ determined by $f$.

Consider $\str{p},\str{q}\in w_!\ope{P}$ and $\str{z}\in w_!\ope{O}$ with lengths $\card{\str{p}}=m_1$, $\card{\str{q}}=m_2$ and  $\card{\str{z}}=n$.

The weights we are considering have the following expressions:
\[
W_{\str{p}\boxtimes \str{q}}(\str{z})\cong w_!\ope{P}(u(\str{z}),p\boxtimes q)\cong \coprod_{\morp{f}{\fcar{n}}{\fcar{m_1}\sqcup\fcar{m_2}}} (\bigotimes_{i=1}^{m_1} \ope{P}(u(\str{z}_{f^{-1}(i)}),p_i))\otimes (\bigotimes_{j=1}^{m_2} \ope{P}(u(\str{z}_{f^{-1}(j)}),q_j))\cong
\]
\begin{equation}\label{eq:big wanna coend}
\cong \coprod_{\morp{f}{\fcar{n}}{\fcar{m_1}\sqcup\fcar{m_2}}} \ope{P}(u(\str{z}_{f^{-1}(m_1)}),\str{p})_{f_1} \otimes
\ope{P}(u(\str{z}_{f^{-1}(m_1)}),\str{q})_{f_2}
\end{equation}
where $f_i$ is the restriction of $f$ to $f^{-1}(m_i)$ for every $i\in \{1,2\}$.\\
For every $\str{s},\str{t}\in w_!\ope{O}$ with lengths $\card{s}=k_1$ and $\card{t}=k_2$:
\begin{equation}\label{eq: operad weight}
W_{\str{p},\str{q}}(\str{s},\str{t})\cong w_!\ope{P}(u(\str{s}),p)\otimes w_!\ope{P}(u(\str{t}),q)\cong
\end{equation}
\[ 
\underset{{\morp{b}{k_2}{m_2}}}{\coprod_{\morp{a}{k_1}{m_1}}} \ope{P}(u(\str{s}),p)_a\otimes \ope{P}(u(\str{t}),q)_b\cong
\]\[\cong\underset{{\morp{b}{k_2}{m_2}}}{\coprod_{\morp{a}{k_1}{m_1}}}\left(\bigotimes^{m_1}_{i=1} \ope{P}(u(\str{s}_{a^{-1}(i)}),p_i)\right)\otimes \left(\bigotimes^{m_2}_{j=1}\ope{P}(u(\str{t}_{b^{-1}(j)}),q_j)\right)
\]
and
\begin{equation}\label{eq: weight of weights}
w_!\ope{O}(\str{z},\str{s}\otimes \str{t})\cong \coprod_{\morp{g}{\fcar{n}}{\fcar{k_1}\sqcup \fcar{k_2}}} \left(\bigotimes^{k_1}_{i=1}\ope{O}(\str{z}_{g^{-1}(i)},s_i)\right)\otimes \left(\bigotimes^{k_2}_{j=1}\ope{O}(\str{z}_{g^{-1}(j)},t_j)\right)\cong
\end{equation}
\[
\cong \coprod_{\morp{g}{\fcar{n}}{\fcar{k_1}\sqcup \fcar{k_2}}} \ope{O}(\str{z_{g^{-1}(k_1)}},\str{p})_{g_1} \otimes \ope{O}(\str{z_{g^{-1}(k_2)}},\str{p})_{g_2}
\]

\begin{lemma}\label{lemma:main lemma kan ext}
    Let $\morp{u}{\ope{O}}{\ope{P}}$ be a map of $\Mm$-operads. The induced map of PROPs $\morp{u}{w_!\ope{O}}{w_!\ope{P}}$ is operadic. 
\end{lemma}

\begin{proof}
    We start by showing that for every $\str{p},\str{q}\in w_!\ope{O}$ diagram (\ref{eq: weights kan ext})
    exhibits $W_{\str{p}\boxtimes\str{q}}$ as the left Kan extension $\LKan{\opc{\boxtimes}} W_{\str{p},\str{q}}$.
    Since $\Mm$ is cocomplete as a $\Mm$-category, for every $\str{z}\in w_!\ope{O}$ the left Kan extension  $\LKan{\opc{\boxtimes}} W_{\str{p},\str{q}}$
    has $\str{z}$-component:
    \[
      \LKan{\opc{\boxtimes}} W_{\str{p},\str{q}}(\str{z})\cong \gc^{\str{s},\str{t}\in w_!\ope{O}\times w_!\ope{O}} w_!\ope{O}(\str{z},\str{s}\otimes \str{t}) \otimes 
      w_!\ope{P}(u(\str{s}),p)\otimes w_!\ope{P}(u(\str{t}),q)
    \]  
    According to (\ref{eq: operad weight}) and (\ref{eq: weight of weights})
    $w_!\ope{O}(\str{z},\str{s}\boxtimes \str{t})\otimes w_!\ope{P}(u(\str{s}),\str{p})\otimes w_!\ope{P}(u(\str{t}),\str{q})$ is isomorphic to
    \begin{equation}\label{eq: big coend expression}
      \coprod_{\morp{g}{\fcar{n}}{\fcar{k_1}\sqcup \fcar{k_2}}}
      \underset{{\morp{b}{k_2}{m_2}}}{\coprod_{\morp{a}{k_1}{m_1}}}
      \ope{O}(\str{z}_{g^{-1}(k_1)},\str{p})_{g_1} \otimes \ope{O}(\str{z}_{g^{-1}(k_2)},\str{p})_{g_2}\otimes
      \ope{P}(u(\str{s}),p)_a\otimes \ope{P}(u(\str{t}),q)_b
    \end{equation}
    The $\Mm$-natural transformation $\alpha$ in diagram (\ref{eq: weights kan ext}) induces a canonical natural transformation $\ntra{\widetilde{\alpha}}{\LKan{\opc{\otimes}} W_{p,q}}{W_{\str{p}\boxtimes\str{q}}}$.
    It is sufficient to show that for every $\str{z}$ in $w_!\ope{O}$ the component map:
    \[
    \widetilde{\alpha}_{\str{z}}\colon \gc^{\str{s},\str{t}\in w_!\ope{O}\times w_!\ope{O}} w_!\ope{O}(\str{z},\str{s}\boxtimes \str{t}) \otimes 
    w_!\ope{P}(u(\str{s}),p)\otimes w_!\ope{P}(u(\str{t}),q) \longrightarrow w_!\ope{P}(u(z),\str{p}\boxtimes\str{q}).
    \]
    is an isomorphism. 
    By the universal property of coends, this is equivalent to show that the induced map
    \begin{equation}
    \morp{\widetilde{\alpha}_{\str{z}}^*}{\Mm(W_{\str{p}\boxtimes\str{q}}(\str{z}),X)}{\Enr{\Mm} \mathrm{ExtNat}\left(w_!\ope{O}(\str{z},-\boxtimes-)\otimes W_{\str{p,q}}(-\boxtimes-),X\right)}
    \end{equation}
    is a bijection for every $X\in \Mm$; here $\Enr{\Mm} \mathrm{ExtNat}\left(w_!\ope{O}(\str{z},-\boxtimes-)\otimes W_{\str{p,q}}(-\boxtimes-),X\right)$ denotes the set of extra natural $\Mm$-transformations from $\morp{w_!\ope{O}(\str{z},-\boxtimes-)\otimes W_{\str{p,q}}(-\boxtimes-)}{\opc{w_!\ope{O}}\times\opc{w_!\ope{O}}\times w_!\ope{O}\times w_!\ope{O}}{\Mm}$ to $X$ (see \cite[\S 1.7]{Ke82}).
    
    We are going to describe $\widetilde{\alpha}_{\str{z}}^*$ and its inverse $\beta$ explicitly.
    
    Consider a morphism $\morp{h}{W_{\str{p}\boxtimes \str{q}}}{X}$; for every $\str{s},\str{t}\in w_!\ope{O}$ the component 
    $\morp{\alpha_{\str{z}}^*(h)_{\str{s},\str{t}}}{w_!\ope{O}(\str{z},\str{s}\boxtimes \str{t})\otimes W_{p,q}(\str{s},\str{t})}{X}$
    of the extra-natural transformation $\alpha_{\str{z}}^*(h)$ has the following description: for every $\morp{g}{\fcar{n}}{\fcar{k_1}\sqcup \fcar{k_2}}$ and every $\morp{a}{\fcar{k}_1}{\fcar{m}_1}$, $\morp{b}{\fcar{k}_2}{\fcar{m}_2}$ the restriction to the $(g,a,b)$-component of
    the coproduct (\ref{eq: big coend expression}) is:
    \begin{equation}
    \xymatrix{
        \ope{O}(\str{z},\str{s}\boxtimes \str{t})_g\otimes \ope{P}(u(\str{s}),\str{p})_a\otimes \ope{P}(u(\str{t}),\str{q})_b
        \ar[d]^-{\textit{ composition after applying $u$}}\\
        \ope{P}(u(\str{z}_{g^{-1}(k_1)}),\str{p})_{a\circ g_1} \otimes \ope{P}(u(\str{z}_{g^{-1}(k_2)}),\str{p})_{b\circ g_2}
        \ar[d]^-{\textit{ tensor in  $w_!\ope{P}$}}\\
        \ope{P}(u(\str{z}),\str{p}\boxtimes\str{q})_{(a\sqcup b) \circ g}
        \ar[d]^-{h}\\
        X.
    } 
    \end{equation}
    For every extra-natural transformation $\gamma\in \Enr{\Mm} \mathrm{ExtNat}\left(w_!\ope{O}(\str{z},-\boxtimes-)\otimes W_{\str{p,q}}(-\boxtimes-),X\right)$ and for every $\morp{f}{\fcar{n}}{\fcar{m}_1\sqcup \fcar{m}_2}$ the map $\morp{\beta(\gamma)}{W_{\str{p}\boxtimes\str{q}}(\str{z})}{X}$ 
    is defined on the $f$-component of the coproduct (\ref{eq:big wanna coend}):
    
    \begin{equation}
    \xymatrix{
        \smxylabel{\Un\otimes \ope{P}(\str{z},\str{p}\boxtimes \str{q})_{f}}  
        \ar[d]^-{v\otimes \id}\\ 
        \smxylabel{\ope{O}(\str{z},\str{z}_{f^{-1}(m_1)}\otimes \str{z}_{f^{-1}(m_2)})_{\omega_{\widehat{f}}}
            \otimes \ope{P}(u(\str{z}_{f^{-1}(m_1)}),\str{p})_{f_1}\otimes \ope{P}(u(\str{z}_{f^{-1}(m_2)},\str{q})_{f_2}}
        \ar[d]^-{\gamma}\\
        X
    }
    \end{equation}
    where $v$ is the tensor product of the unit
    morphisms of every component of \[\ope{O}(\str{z},\str{z}_{f^{-1}(m_1)}\boxtimes \str{z}_{f^{-1}(m_2)})_{\omega_{\widehat{f}}}
    \cong \bigotimes_{i=1}^n \ope{O}(z_{\omega_{\widehat{f}}^{-1}(i)},z_{\omega_{\widehat{f}}^{-1}(i)}).\]
    
    It is easy to check that $\alpha^*_{\str{z}}\beta(h)=h$ for every $h\in \Mm(W_{\str{p}\boxtimes\str{q}}(\str{z}),X)$.\\
    To check that for every $\beta\alpha^*_{\str{z}}=\id$ it is sufficient to observe that for every \[\gamma\in \Enr{\Mm} \mathrm{ExtNat}(w_!\ope{O}(\str{z},-\boxtimes-)\otimes W_{\str{p,q}}(-\boxtimes-),X),\] every $\str{s},\str{t}\in w_!\ope{O}$ and every $\morp{g}{\fcar{n}}{\fcar{k_1}\sqcup \fcar{k_2}}$ and every $\morp{a}{\fcar{k}_1}{\fcar{m}_1}$, $\morp{b}{\fcar{k}_2}{\fcar{m}_2}$ the following diagram commutes
    \begin{center}
        \begin{tikzpicture}[every node/.style={scale=0.8}]
        \node (a) at (0,0) 
        {$\ope{O}(\str{z}_{g^{-1}(k_1)},\str{s})_{g_1} \otimes \ope{O}(\str{z}_{g^{-1}(k_2)},\str{t})_{g_2}
            \otimes \Un \otimes \ope{P}(\str{s},\str{p}))_a \otimes \ope{P}(\str{s},\str{q}))_b$};
        \node (b) at (0,-2)
        {$\ope{O}( \str{z}_{g^{-1}(k_1)},\str{s})_{g_1} \otimes \ope{O}(\str{z}_{g^{-1}(k_2)},\str{t})_{g_2}
            \otimes \ope{O}(\str{z},\str{z}_{g^{-1}(k_1)}\otimes \str{z}_{g^{-1}(k_2)})_{\omega_{\widehat{g}}} \otimes \ope{P}(\str{s},\str{p}))_a \otimes \ope{P}(\str{s},\str{q}))_b$}; 
        \node (c) at (-5,-4)
        {
            $\ope{O}( \str{z},\str{s}\boxtimes \str{t})_g \otimes \ope{P}(\str{s},\str{p})_a \otimes \ope{P}(\str{s},\str{q}))_b$
        };
        \node (d) at (3,-5)
        {
            $\ope{O}(\str{z},\str{z}_{g^{-1}(k_1)}\otimes \str{z}_{g^{-1}(k_2)})_{\omega_{\widehat{(a\sqcup b)g}}} \otimes \ope{P}(\str{z}_{g^{-1}(k_1)},\str{p})_{ag_1} \otimes \ope{P}(\str{z}_{g^{-1}(k_2)},\str{q})_{bg_2}$
        };
        \node (x) at (-2,-7) {$X$};
        \path (a) edge[->] (b);
        \path (b) edge[->] (c);
        \path (b) edge[->] (d);
        \draw (d) edge[->] node[midway,below,xshift=15ex] {$\gamma_{\str{z}_{g^{-1}(k_1)},\str{z}_{g^{-1}(k_2)},\omega_{\widehat{(a\sqcup b)g}},ag_1,bg_2}$} (x);
        \draw (c) edge[->] node[midway,below,xshift=-5ex] {$\gamma_{\str{s},\str{t},g,ag,b}$} (x);
        \draw (a.west) edge[bend right=40,->] node[midway,above,xshift=-3pt] {$\cong$} ($(c.north west)+(1,0)$);
        \node[gray] (txt) at (-0.5,-4) {\textit{extra-naturality of $\gamma$}};
        \end{tikzpicture}
    \end{center}
    The composition of the inverse of the isomorphism on the top left with the descending chain of arrows on the right is equal to $\alpha_{\str{z}}^*\beta(\gamma)_{\str{s},\str{t},g,ag,b}$; this shows that $\alpha_{\str{z}}^*\beta(\gamma)=\gamma$.
    
    To show that $W_{[]}$ is the left Kan extension $\LKan{[]}{\Un}$ notice that the map induced by $\lambda$ (diagram \ref{eq: weights kan ext})
    \[
    \lmorp{\widetilde{\lambda}_{\str{z}}}{\LKan{[]}{\Un}(\str{z})}{W_{[]}(\str{z})}
    \]
    is isomorphic to
    $
    \morp{u_\str{z}}{w_!\ope{O}(\str{z},[])}{w_!\ope{P}(\str{z},[])}
    $,
    which is an isomorphism for every $\str{z}\in w_!\ope{O}$.
\end{proof}

\begin{cor}\label{cor:Lan}
    For every morphism of $\Mm$-operads 
    $\morp{u}{\ope{O}}{\ope{P}}$ and every algebra $A\in \Alg{\ope{O}}{\cat{M}}$ the $\ope{P}$-algebra $\morp{u_!A}{w_!\ope{P}}{\Mm}$ is the left Kan extension of $A$ along $w_!(u)$; in particular:
    \[
    u_!A(c)\cong \gc^{\str{s}\in w_!\ope{O}}
    \ope{P}(u(\str{s});c)\otimes A(s_1)\otimes\dots \otimes A (s_{\card{\str{s}}})
    \]
    for every colour $c$ of $\ope{P}$.  
\end{cor}
\begin{proof}
    It is a direct application of Proposition \ref{prop: algebraic kan trick}; in fact Lemma \ref{lemma:main lemma kan ext} shows that $u$ satisfies all the required hypotheses.
\end{proof}

\begin{cor}\label{cor:ff}
    For every fully faithful morphism of $\Mm$-operads $\morp{u}{\ope{O}}{\ope{P}}$
    the induced functor $\morp{u_!}{\Alg{\ope{O}}{\cat{M}}}{\Alg{\ope{P}}{\cat{M}}}$ is fully faithful.
\end{cor}
\begin{proof}
 If $u$ is fully-faithful, $w_!(u)$ is fully-faithful thus the unit of the adjunction $(u_!,u^*)$ is a natural isomorphism, i.e.
 $u_!$ is fully-faithful.
\end{proof}

\begin{prop}\label{prop:operadic operads}
A $C$-coloured $\Mm$-PROPs is in the image of $w_!$ if and only if the unique morphism of $C$-coloured PROPs
$\morp{i}{\Str{C}}{\prp{P}}$ is operadic. 
\end{prop}
\begin{proof}
Suppose that $i$ is operadic; to show that $\prp{P}$ is in the image of $w_!$ it is sufficient to prove that the counit $\morp{\eta_{\prp{P}}}{w_!w^*\prp{P}}{\prp{P}}$ is an isomorphism. 

According to Proposition \ref{prop: different operadic}, the morphism $i$ is operadic if and only if for every $\str{c},\str{d}\in \Str{C}$, with length $n$ and $m$ respectively, the $\str{c}$-component of the map induced by $\lambda^{\str{d}}$
\begin{equation}\label{eq: lambda of a discrete}
 \gc^{\str{e_1},\dots,\str{e_n}\in \Str{C}} \Str{C}(\str{c},\str{e_1}\boxtimes \dots \boxtimes \str{e_m})\otimes 
\bigotimes_{i=1}^m \prp{P}(\str{e_i},d_i) \longrightarrow \prp{P}(\str{c},\str{d})
\end{equation}
is an isomorphism.

The left hand side is isomorphic to:
\[
 \underset{(\str{c}\to \str{e})\in \comcat{\str{c}}{\boxtimes}} \varinjlim \prp{P}(\str{e_i},d_i)
\]
The comma category $\comcat{\str{c}}{\boxtimes}$ (where $\morp{\boxtimes}{\Str{C}^{\times n}}{\Str{C}}$ denotes the $m$-fold tensor product) is a simply connected groupoid and its set of connected components is isomorphic to $\pi_0(\comcat{\str{c}}{\boxtimes})\cong \Set(\fcar{n},\fcar{m})$.
Under these identifications the morphism (\ref{eq: lambda of a discrete}) becomes
\[
 \underset{\morp{f}{\fcar{n}}{\fcar{m}}}\coprod \prp{P}(\str{c}_{f^{-1}(i)},d_i) \longrightarrow \prp{P}(\str{c},\str{d})
\]
We leave to the reader to check that this map coincides with the $(\str{c},\str{d})$-component of the counit of $(w_!,w^*)$ at $\prp{P}$ (\emph{cf.} formula (\ref{eq:eq free prop oper})).

If $\prp{P}$ is in the image of $w_!$, the morphism $i$ is in the image of $w_!$, thus $i$ is operadic by lemma \ref{lemma:main lemma kan ext}. 
\end{proof}

\subsubsection{Extensions of \texorpdfstring{$\Set$}{Set}-Operads}\label{sec:extset}
In ordinary categories (i.e. in $\Set$-categories) all weighted colimits can be computed as ordinary colimits indexed by the category
of elements of the considered weight; it follows that (point-wise) left Kan extension along a functor $f$ are computed object-wise 
as colimits indexed by the comma category of $f$ over the considered object.
 
In particular, given a morphism of $\Set$-operads $\morp{u}{\ope{O}}{\ope{P}}$ and a strong monoidal functor $\morp{A}{w_!\ope{O}}{\cat{M}}$
(i.e. an $\ope{O}$-algebra in $\cat{M}$), for every $p\in \Cl{\ope{P}}$:
\begin{equation}
 u!A(p)\cong \underset{\{\arr{u(\str{o})}{p}\}\in 
\comcat{u}{p}}\varinjlim A(o_1)\otimes\dots\otimes A(o_{\card{o}})
\end{equation}
It is thus useful to analyse the structure of the comma category $w_!\comcat{u}{p}$ to get information over $u_!$.
Examples were presented in \S \ref{sec.operadj}.

\subsection{Feynman Categories}\label{sec:fey cat}
Kaufmann and Ward defines in \cite{KW14} another class of symmetric monoidal categories that serve to model algebraic objects, namely the class of Feynman categories.

An instance of operadic functor (def. \ref{def:operadic}) appears in the formulation of the hierarchy condition for Feynman categories (\emph{c.f.} 1.1.2 in \emph{loc.cit.}).

In this section we would like to compare the notion of coloured operad and Feynman categories, showing that they are equivalent and that the former can be regarded as a strictification of the latter. The equivalence between Feynman categories and coloured operad has also been proved independently by Batanin, Kock and Weber and will appear in \cite{BKW15}.

We begin by defining what we call a ``pre-Feynman category", that should be regarded as a non-strict version of a PROP.

The results of this section should also indicate that the requirement that the symmetric monoidal category underlying a PROP is strict monoidal is more a simplification than a restriction.
      
\subsubsection{Pre-Feynman Categories}
Fix a bicomplete closed symmetric monoidal category $(\Mm,\otimes,\Un)$.

For every category $\cat{C}$ we denote by $\Cor{\cat{C}}$ its \emph{core}, i.e. the  maximal subgroupoid of $\cat{C}$; for a $\Mm$-enriched category $\cat{W}$, $\Cor{\cat{W}}$ will stand for the core of its underlying (un-enriched) category.

Recall that for every category $\cat{C}$, the symmetric monoidal category underlying the PROP $k_!(\cat{C})$ can be characterised as the free symmetric monoidal category generated by $\cat{C}$.

\begin{defi}    
A \emph{(small) pre-Feynman $\Mm$-category}\index{pre-Feynman category} is a triple $(\Mfg, \Mf, \iota)$ where $\Mfg$ is an (un-enriched) groupoid, 
$\Mf$ is a (small) symmetric monoidal $\Mm$-category and $\iota$ is a fully-faithful functor $\morp{\iota}{\Mfg}{\Cor{\Mf}}$ such that the induced strong monoidal $\Mm$-functor $\morp{\adjsm{\iota}}{\fsm{\Mfg}}{\Mf}$ is essentially surjective, that is every object $x\in \Mf$ is isomorphic to $\iota g_1\otimes \dots \otimes \iota{g_n}$ for some $n\in \N$ and some $g_1,\dots,g_n\in \Mfg$.

A \emph{morphism of pre-Feynman $\Mm$-categories}\index{morphism! of pre-Feynman categories} $\morp{f}{(\Mfg,\Mf,\iota)}{(\Mfg',\Mf',\iota')}$ is a triple $(v_f,m_f,\alpha_f)$ where $\morp{v_f}{\Mfg}{\Mfg'}$
is a morphism of groupoid, $\morp{m_f}{\Mf}{\Mf'}$ is a strong monoidal $\Mm$-functor and $\ntra{\alpha_f}{\adjsm{\iota'}\fsm{(v_f)}}{m_f \adjsm{\iota}}$ is an invertible monoidal $\Mm$-natural transformation.
\[
 \diagcn{\fsm{\Mfg}}{\fsm{\Mfg'}}{\Mf}{\Mf'}{\adjsm{\iota}}{\fsm{(v_f)}}{\adjsm{\iota'}}{m_f}{\alpha_f}
\]
\end{defi}
\begin{rmk}
Note that $\alpha$ is completely determined by its behaviour on $\Mfg$, thus we could have defined $\alpha$ as an invertible $\Mm$-natural transformation from $\iota' v_f$ to $\iota m_f$.
\end{rmk}
\noindent Composition of morphisms of pre-Feynman $\Mm$-categories is defined in the evident way.

We will denote by $\EpreFey{\Mm}$ the $2$-category where $0$-cells are pre-Feynman $\Mm$-categories, $1$-cells are morphisms between them
and $2$-cells are described in the following way:
given  $\morp{f,g}{(\Mfg,\Mf,\iota)}{(\Mfg',\Mf',\iota')}$ a $2$-cell $f\Rightarrow g$ is
 a couple $(\gamma,\eta)$ where $\ntra{\eta}{m_f}{m_g}$ is a monoidal $\Mm$-natural transformation and $\ntra{\gamma}{\iota' v_f}{\iota' v_g}$ is an $\Mm$-natural transformation
such that $(\eta\iota) \alpha_f= \alpha_g \gamma$ (figure \ref{fig:2cell}).
\begin{figure}[th!]
\begin{tikzpicture}[scale=1.2]
\matrix(a)[matrix of math nodes,
row sep=5em, column sep=7em,
text height=1.5ex, text depth=0.25ex]
{\Mfg&\Mfg'\\
 \Mf& \Mf'\\};
\path[bend left=20,->](a-1-1) edge node[pos=.5,above] {\footnotesize $v_f$} (a-1-2);
\path[bend left=25,gray, ->](a-1-1) edge node[pos=.4,name=d1]{} node[pos=.5,name=c1]{} (a-2-2);
\path[bend right=27,->](a-1-1) edge node[pos=.5,above] {\footnotesize $v_g$} (a-1-2);
\path[bend left=20,->](a-2-1) edge node[pos=.45,name=b1]{} node[pos=.6,above] {\footnotesize $m_f$} (a-2-2);
\path[bend right=27,->](a-2-1) edge node[pos=.4,name=b2]{} node[pos=.45,below] {\footnotesize $m_g$} (a-2-2);
\path (b1) edge[double distance=2pt,shorten <=2pt,shorten >=2pt,-implies] node[fill=white,inner sep=0pt,pos=.5,rounded corners=2pt,right,xshift=4pt]{\footnotesize $\eta$} (b2);
\path (d1) edge[bend right=20,double distance=2pt,shorten <=0.4cm,shorten >=0.3cm,-implies] node[fill=white,inner sep=0pt,pos=.3,rounded corners=2pt,above,yshift=2pt,xshift=-6pt]{\footnotesize $\alpha_f$} (a-2-1);
\path[bend right=27,gray,->](a-1-1) edge node[pos=.5,name=c2]{} node[pos=.5,name=d2]{} (a-2-2);
\path[->](a-1-1) edge node[pos=.5,left] {\footnotesize $\iota$} (a-2-1);
\path[->](a-1-2) edge node[pos=.5,right] {\footnotesize $\iota'$}(a-2-2);
\path (c1) edge[shorten <=0pt,shorten >=0pt,double distance=2pt,-implies] node[pos=.7,above,yshift=3pt]{\footnotesize $\gamma$} (c2);
\path  (d2) edge[bend left=10,double distance=2pt,shorten <=0.1cm,shorten >=0.05cm,-implies] node[fill=white,inner sep=0pt,pos=.3,rounded corners=2pt,above,yshift=0pt,xshift=-6pt]{\footnotesize $\alpha_g$}(a-2-1);
\end{tikzpicture}
\caption{A $2$-cell in $\preFey$.}\label{fig:2cell}
\end{figure}

A $1$-cell $(v_f,m_f)$ is an equivalence in the $2$-category $\EpreFey{\Mm}$ if and only if $m_f$ is an equivalence of $\Mm$-categories.

\begin{defi}
    For a pre-Feynman $\Mm$-category $(\Mfg,\Mf,\iota)$ and a symmetric monoidal $\Mm$-category $\cat{W}$ the \emph{category of $\Mf$-algebras in $\cat{W}$} is the category of strong symmetric monoidal functor from $\Mf$ to $\cat{W}$ (and monoidal natural transformations between them) and it is denoted by $\Alg{\Mf}{\cat{W}}$.
\end{defi}

\noindent Recall that a $C$-coloured $\Mm$-PROP can be defined as a triple as $(C,\prp{P},i)$ as in Remark \ref{rmk:flex prop}.

There is a $2$-functor:
\begin{equation}\label{eq:fun:prop->preFey}
\gfun{\FSmf}{\EProp{\Mm}}{\EpreFey{\Mm}}{(C,\prp{P})}{(\Cor{\prp{P}},\prp{P},i)}
\end{equation} 
where $\Cor{\prp{P}}=\Cor{k^*\prp{P}}$ is the \emph{core of} $k^*(\prp{P})$ and $i$ is the canonical inclusion $\morp{i}{\Cor{\prp{P}}}{\prp{P}}$. $T$ is defined on $1$-cells and $2$-cells in the evident way.

There is also a $2$-functor going in the opposite direction:
\begin{equation}\label{eq:fun:preFey->prop}
\gfun{\Foff}{\EpreFey{\Mm}}{\EProp{\Mm}}{(\Mfg,\Mf,\iota)}{(\Ob{\Mfg},\EndP{\Mm}(\mathbf{G}))}
\end{equation}
where $\EndP{\Mm}(\mathbf{G})$ is the endomorphism $\Mm$-PROP for the collection $\mathbf{G}=\{\iota(g)\}_{g\in \Ob{\Mfg}}$ (\S \ref{sec:endomorphism PROP}).

\begin{prop}\label{prop:eq preFey and PROPs}
 The $2$-functors $\FSmf$ and $\Foff$ defines a biequivalence of $2$-categories.
\end{prop}
\begin{proof}
  It is straight-forward to check that $ST$ is isomorphic to the identity on $\EProp{\Mm}$.
  Thus it is sufficient to exhibit an equivalence of pre-Feynman categories $\morp{\kappa}{TS(\Mf)}{\Mf}$
  pseudo natural in $\Mf=(\Mfg,\Mf,\iota)\in \EpreFey{\Mm}$. 
  
  For every $(\Mfg,\Mf,\iota)\in \EpreFey{\Mm}$ we have two morphisms of pre-Feynman categories, represented by the left and the right square of the
  following diagram
  \[
  \xymatrix{ \Mfg \ar[r] \ar[d] & \Mfg \ar[d]_-{} \ar@{}[dl]^(.3){}="a"^(.7){}="b" \ar@2{->}_-{\id} "a";"b" 
	      \ar@{}[dr]^(.3){}="x"^(.7){}="y" \ar@2{->}^-{\id} "x";"y" & \ar[l] \Mfg \ar[d]\\
	      \FSmf\Foff(\Mf) \ar[r] & \EndP{\Mm}(\Mf) & \ar[l] \Mf
	    }
  \]
  where $\EndP{\Mm}(\Mf)$ is the $\Mm$-PROP associated to $\Mf$ (or equivalently the strictification of $\Mf$, \S \ref{sec:endomorphism PROP}, \ref{apx:strictification}).
  Both squares are equivalence of pre-Feynman $\Mm$-categories and are pseudo natural in $\Mf$; choosing an inverse for the right square (for every $\Mf$)
  we get the desired pseudo natural equivalence $\ntra{\kappa}{\FSmf\Foff}{\id_{\EpreFey{\Mm}}}$.
\end{proof}
\begin{rmk}
    Note that it also follows that for every symmetric monoidal $\Mm$-category $\cat{W}$ and every pre-Feynman category $\Mf$ there is an equivalence of categories $\Alg{\Mf}{\cat{W}}\cong \Alg{S(\Mf)}{\cat{W}}$ and conversely, for every $\Mm$-PROP $\prp{P}$ the categories of algebras $\Alg{\prp{P}}{\cat{W}}$ and  $\Alg{T(\prp{P})}{\cat{W}}$ are equivalent. 
\end{rmk}
\subsubsection{Feynman Categories and Operads}

\begin{defi}(\cite{KW14} def. 1.1)
A \emph{(small) Feynman $\Mm$-categories}\index{Feynman category} is a pre-Feynman category $(\Mfg, \Mf, \iota)$ such
the functor $\morp{\adjsm{\iota}}{k_!\Mfg}{\Mf}$ is operadic.
\end{defi}
\begin{rmk}

\noindent Feynman categories are defined in \cite{KW14} in the case in which $\Mm=\Set$; their definition is formulated in a different way than the above one, but it is equivalent (at least for small Feynman categories).
In fact, in definition 1.1 of \emph{loc.cit.} a Feynman category is defined to be triple $(\Mfg, \Mf, \iota)$ as above satisfying three condition: the isomorphism condition, the hierarchy condition and the size condition. 

The size condition requires that the comma category $\comcat{\Mf}{\iota{g}}$ is essentially small for every $g\in \Mfg$;
Since we are only considering \emph{small} Feynman categories the size condition is automatically satisfied.

The fact that the hierarchy condition is equivalent to the operadicity of $\adjsm{\iota}$ under the isomorphism condition is clearly explained in \S 1.1.2 of \emph{loc.cit.}.

Thus it is enough to prove that, under the assumption that $(\Mfg, \Mf, \iota)$ is a Pre-Feynman category, the operadicity of $\adjsm{\iota}$ implies the isomorphism condition.

The isomorphism condition requires that $\morp{\adjsm{\iota}}{k_!\Mfg}{\Cor{\Mf}}$ is an equivalence of groupoids; since $(\Mfg, \Mf, \iota)$ is a pre-Feynman categories, it is sufficient to show that $\adjsm{\iota}$ is fully faithful on the isomorphisms.  
 
For every $\str{g},\str{h}\in \Str{\Ob{\cat{G}}}$ with length $\card{\str{g}}=n$ and $\card{\str{h}}=n$, the requirement that $\adjsm{\iota}$
is operadic implies that the canonical map
\begin{equation}\label{eq:operadic for fey}
\gc^{\str{s}_1,\dots,\str{s}_m\in k_!\cat{G}} k_!\cat{G}(\str{g},\str{s}_1\boxtimes\dots\boxtimes\str{s}_n)\otimes 
      \Mf(\adjsm{\iota}(\str{s}_1),\iota{h_1})\otimes \dots \otimes \Mf(\adjsm{\iota}(\str{s}_m),\iota{h_m}) \longrightarrow \Mf(\adjsm{\iota}(\str{g}),\adjsm{\iota}(\str{h}))
\end{equation}
is an isomorphism. Since $\comcat{\str{g}}{\boxtimes^m}$ is a simply connected groupoid with set of connected components isomorphic to $\Set(\fcar{n},\fcar{m})$, (\ref{eq:operadic for fey}) becomes  
\[
 \underset{\morp{f}{\fcar{n}}{\fcar{m}}}\coprod 
      \Mf(\adjsm{\iota}(\str{g}_{f^{-1}(1)}),\iota{h_1})\otimes \dots \otimes \Mf(\adjsm{\iota}(\str{g}_{f^{-1}(m)}),\iota{h_m})
      \longrightarrow \Mf(\adjsm{\iota}(\str{g}),\adjsm{\iota}(\str{h}))
\]
In the case $\Mm=\Set$, this implies that the set of isomorphism in $\Mf(\adjsm{\iota}(\str{g}),\adjsm{\iota}(\str{h}))$ is non-empty only if $n=m$ and in that case it canonically isomorphic to
\[
 \underset{f\in \Sigma_n}\coprod 
      \Cor{\Mf}(\iota(g_{f^{-1}(1)}),\iota(h_1))\otimes \dots \otimes \Cor{\Mf}(\iota(g_{f^{-1}(m)}),\iota(h_m)); 
\]
in view of the fact that $\iota$ is fully faithful on isomorphism this implies that $\morp{\adjsm{\iota}}{k_!\cat{G}}{\Cor{\Mf}}$ is
fully faithful; thus the isomorphism condition is satisfied.
\end{rmk}
 
The \emph{$2$-category of Feynman $\Mm$-category} $\EFey{\Mm}$ is defined to be the full $2$-subcategory of $\EpreFey{\Mm}$ spanned by all Feynman $\Mm$-categories. We will denote by $\morp{z}{\EFey{\Mm}}{\EpreFey{\Mm}}$ the canonical $2$-inclusion.

\begin{prop}
    The biequivalence of Proposition \ref{prop:eq preFey and PROPs} restricts to a biequivalence between $\EOper{\Mm}$ and $\EFey{\Mm}$.
\end{prop}
\begin{proof}
    We begin by proving that for every $\Mm$- operad $\mathcal{O}$ the pre-Feynman $\Mm$-category $T(\ope{O})$ is in the essential image of $z$; in fact, the morphism $\arr{k_!\Cor{\ope{O}}}{w_!\ope{O}}$ is operadic by Lemma \ref{lemma:main lemma kan ext}.
    
    We only need to prove that for every Feynman $\Mm$-category $(\Mf,\Mfg,\iota)$, $S(\Mf)$ is in the essential image of $w$;
    first notice that $(k_!\Mfg,\Mfg,\eta_{\Mfg})$ also a Feynman $\Mm$-category (here $\eta$ is the counit of the adjunction $(k_!,k^*)$); 
    since $\morp{\adjsm{\iota}}{k_!\Mfg}{\Mfg}$ is operadic the induced morphism
    $\morp{S(\adjsm{\iota})}{k_!\Mfg \cong S(k_!\Mfg)}{\Mf}$
    is also operadic; the morphism $S(\adjsm{\iota})$ is a morphism of $\Ob{\Mfg}$-coloured $\Mm$-PROPs; the unique morphism $\Ob{\Mfg}$-coloured $\Mm$-PROPs $\Str{\Ob{G}}\to k_!\Mfg$ is operadic, thus the composition
    \[
    \Str{\Ob{G}}\longrightarrow k_!\Mfg \overset{S(\adjsm{\iota})}\longrightarrow
        S(\Mf)
    \]
    is also operadic; it follows that $S(\Mf)$ is in the essential image of $w_!$ by Proposition \ref{prop:operadic operads}.       
\end{proof}
\noindent 
To summarise, we have a commutative diagram of $2$-functors 
\[
\xymatrix{\EOper{\Mm} \ar@<3pt>[r]^-{\FSmf} \ar[d]_-{w_!} & \ar@<3pt>[l]^-{\Foff} \EFey{\Mm} \ar[d]^-{z}\\ 
    \EProp{\Mm} \ar@<3pt>[r]^-{\FSmf} & \ar@<3pt>[l]^-{\Foff} \EpreFey{\Mm}
}
\]
where the horizontal arrows are biequivalences.
This shows that every small Feynman category can be ``strictified'' into a small coloured operad
with an equivalent category of algebras.

\subsection{Strictification}\label{apx:strictification}
\newcommand{\smon}[1]{\mathrm{str}(#1)}
\newcommand{\brak}[1]{\langle #1 \rangle}
Every monoidal $\Mm$-category $(\cat{W},\otimes,\Un_{\cat{W}})$ is equivalent to a strict monoidal $\Mm$-category.
Following Joyal and Street (\cite{JS91}), we recall how this strictification can be constructed;
we are going to define a slightly more general construction that associate to each collection of objects $\cat{W}$ a strict monoidal $\Mm$-category; we will recover a strictification of $\cat{W}$ as a special case.

Let $C$ be a (possibly large) set and let $\mathbf{X}=\{X_c\}_{c\in C}$ be a collection of objects in $\cat{W}$ indexed by $C$. 
for every $\str{c}\in \Str{C}$ define the object $\brak{\str{c}}\in \cat{W}$
by induction on the length of $\str{c}$ by the requirement that
$\langle[]\rangle=\Un_{\cat{W}}$, $\brak{(c)}=X_c$ and $\brak{\str{c}\ast (X_c)}=\brak{\str{c}}\otimes X_c$ for every $w\in W$ and $\str{w}\in \smon{\cat{W}}$.
We define a strict monoidal $\Mm$-category $\EndP{\Mm}(\mathbf{X})$; the set of objects of $\EndP{\Mm}(\mathbf{X})$ is $\Str{C}$ and for every 
for every $\str{c},\str{d}\in \Str{C}$ we set $\smon{W}(\str{c},\str{d})=\cat{W}(\brak{\str{c}},\brak{\str{d}})$
and $\str{c}\otimes \str{d}=\str{s}\ast \str{t}$. 

On the morphisms the tensor product is defined to be:
\[
\xymatrix{\cat{W}(\brak{\str{a}},\brak{\str{c}})\otimes \cat{W}(\brak{\str{b}},\brak{\str{d}}) \ar[r]^-{\otimes} &
    \cat{W}(\brak{\str{a}}\otimes \brak{\str{b}},\brak{\str{c}}\otimes \brak{\str{s}}) \ar[r]^-{\cong} &
    \cat{W}(\brak{\str{a}\ast \str{b}},\brak{\str{c} \ast \str{d}})}
\]
where the last isomorphism is defined by composition with the unique isomorphisms $\arr{\brak{\str{a}\ast \str{b}}}{\brak{\str{a}}\otimes \brak{\str{b}}}$,
$\arr{\brak{\str{c}}\otimes \brak{\str{d}}}{\brak{\str{c}\ast \str{d}}}$ obtained from the associator.
Composition and identities in $\EndP{\Mm}(
\mathbf{X})$ are defined from the ones in $\cat{W}$ in the obvious way.
With this tensor product $\EndP{\Mm}(
\mathbf{X})$ is clearly strict monoidal.

There is a strong monoidal fully faithful $\Mm$-functor $\morp{R_{
        \mathbf{X}}}{\EndP{\Mm}(
    \mathbf{X})}{\cat{W}}$ sending each $\str{c}\in \Str{C}$ to $\brak{\str{c}}$.

Furthermore, if every objects in $\cat{W}$ is isomorphic to a tensor product of elements of $\mathbf{X}$ then $R_{
    \mathbf{X}}$ is essentially surjective and thus an equivalence of monoidal $\Mm$-categories.

In particular if we take $C=\Ob{\cat{W}}$ and $\mathbf{W}=\{w\}_{w\in \cat{W}}$
then $\EndP{\Mm}(\mathbf{W})$ is a strict monoidal $\Mm$-category equivalent to $\cat{W}$ via $R_{\mathbf{W}}$; this particular strict monoidal $\Mm$-category will be denoted by $\EndP{\Mm}(\cat{W})$. 

Note that if $\cat{W}$ is symmetric (as a monoidal $\Mm$-category) then $\EndP{\Mm}(\mathbf{X})$ has a natural symmetric structure that makes the functor $R_{\mathbf{X}}$ into a symmetric monoidal functor.
In particular $\EndP{\Mm}(\mathbf{W})$ is a symmetric strict monoidal category (but not strict symmetric) equivalent to $\cat{W}$.

\section{Graphs}\label{ch:graph}
In this appendix we review the different notions of graph that are used in this work. 
Graphs are defined in many different ways trough-out the literature (see for example \cite{JK11},\cite{BB13},\cite{KW14}).
For the present work, we need to work with directed graphs which are open; informally this means that we do not require that the extremes of each edge end into a node. We chose to work with the definition of (directed) graph given by Kock in \cite{Ko14}; other choices would have been possible and the reader is free to adapt our definitions to the definition of graph that he prefers.

In part \ref{apx:insertion} we recall the operation of graph insertion, that was essential for defining the operad for PROPs.

\subsection{Directed graphs}
A graph will be for use the following amount of data:
\begin{defi}(\cite{Ko14})
A (directed) graph\index{graph (directed)} $G=(A,N,I,O,s,t,p,q)$ is a diagram in $\fSet$ of the form 
\begin{equation}\label{eq:pres.graph}
 \xymatrix{A & \ar[l]_-{s} I \ar[r]^p & N & \ar[l]_-{q} O \ar[r]^t & A}
\end{equation}
such that $s$ and $t$ are injective maps. The set $A$ is called the \emph{set of edges}\index{edge (of a graph)}, while $N$ is the \emph{set of nodes (or nodes)}\index{node!of a graph}.
The intersection $I\cap O$ form the set of \emph{inner edges}\index{edge (of a graph)!inner} of $G$.
A \emph{morphism of graphs}\index{morphism!of graphs} from $G=(A,N,I,O,s,t,p,q)$ to $G'=(A',N',I',O',s',t',p',q')$ is just a collection of maps $(a,n,i,o)$ that make
the following diagram commute
\begin{equation}\label{morp.graph}
 \xymatrix{A\ar[d]^a & \ar[l]_-{s} I\ar[d]^i \ar[r]^p & N\ar[d]^n & \ar[l]_-{q} O \ar[d]^o \ar[r]^t & A \ar[d]^a \\
           A' & \ar[l]_-{s'} I \ar[r]^{p'} & N' & \ar[l]_-{q'} O \ar[r]^{t'} & A'.}
\end{equation}

A morphism of graphs is called \emph{etale}\index{morphism!of graphs!etale}\index{etale!morphism} if the two central squares in (\ref{morp.graph}) are pullbacks.
\end{defi}

\noindent The category of graphs and morphisms between them will be denoted by $\GraK$.

\begin{rmk}
It is clear from the definition that
$\GraK$ is a full subcategory of the category $\Set^{\mathsf{G}}$ where $\mathsf{G}$ is the small category indexing the
diagram of shape (\ref{eq:pres.graph}). The category $\GraK$ has finite coproducts, which are preserved by these inclusions. 
\end{rmk}
\subsection{Residue and ports}
\begin{defi}\label{def:corolla}
For every $n,m\in \N$ let $C^n_m$ be the graph defined by the diagram:
\[
 \xymatrix{\fcar{n+m} & \ar[l] \fcar{n} \ar[r] & \fcar{1} & \ar[l] \fcar{m} \ar[r] & \fcar{n+m}}.
\]
A \emph{corolla}\index{corolla} is a graph $G$ isomorphic to $C^{n}_m$ for $(n,m)\in \N\times \N$; if such an isomorphism exists,
the couple $(n,m)$ is uniquely determined and is called the \emph{valence of $G$}. 
\end{defi}

\begin{defi}
The \emph{residue of a graph}\index{residue!of a graph} $G=(A,N,I,O,s,t,p,q)$ is the corolla
\[
 \xymatrix{A\backslash (I\cap O) & A\backslash O \ar[l]  \ar[r] & \ast & \ar[l] A\backslash I \ar[r] & A\backslash (I\cap O)}
\]
and will be denoted $\resd{G}$. The \emph{valence of $G$} is the valence of $\resd{G}$.
The sets $A\backslash (O\cap I)$, $A\backslash O$ and $A\backslash I$ will be called the \emph{set of ports of $G$}\index{port!of a graph}, the \emph{set of inports of $G$}\index{inport!of a graph} and the \emph{set of exports of $G$}\index{export!of a graph}, we
will denote them by $\port{G}$,$\iport{G}$ and $\oport{G}$ respectively.

For every node $x\in N_G$ the \emph{residue of $x$}\index{residue!of a node} is the corolla:
\[
 \xymatrix{p^{-1}(x)\coprod q^{-1}(x)& \ar[l] p^{-1}(x) \ar[r] & \{x\} & \ar[l] q^{-1}(x) \ar[r] & p^{-1}(x)\coprod q^{-1}(x)}
\]
The \emph{valence of $x$}\index{valence!of a node} is the valence of $\resd{x}$. 

We also define the \emph{set of ports of $x$}\index{ports!of a node}, the \emph{set of inports of $x$}\index{inport!of a node} and and the \emph{set of exports of $x$}\index{export!of a node} as 
 $\port{x}=\port{\resd{x}}$, $\iport{x}=\iport{\resd{x}}$, $\oport{x}=\oport{\resd{x}}$ respectively.
\end{defi}

\subsection{Paths}
\begin{defi}
Let $n$ be a positive integer. An \emph{open linear graph of length $n$}\index{linear graph!open} is a graph of the form
\[
 \xymatrix{\fcar{n+1} & \ar[l]_-{s} I \ar[r]^p & \fcar{n} & \ar[l]_-{q} O \ar[r]^t & \fcar{n+1}}
\]
such that 
\begin{itemize}
\item[-] for $\card{p^{-1}(v)\sqcup q^{-1}(v)}=2$ for every $v\in \fcar{n}$;
\item[-] $I\cup 0=\fcar{n+1}$;
\item[-] $\card{I\cap O}=\fcar{n+1}\backslash{1,n+1}$. 
\end{itemize}

An open linear graph of length $n$ is \emph{monotone}\index{linear graph!open!monotone} if $\card{p^{-1}(v)}=\card{q^{-1}(v)}=1$ for every $v\in \fcar{n}$
and the edge $1$ is the only inport.
An open linear graph of length $0$ is by definition a graph with no nodes and $\fcar{1}$ as set of edges. 
The \emph{monotone open linear graph of length $n$} is the open linear graph $L_n$ represented by
\[
 \xymatrix{\fcar{n+1} & \ar[l]_-{+0} \fcar{n} \ar[r]^{\id} & \fcar{n} & \ar[l]_{\id} \fcar{n} \ar[r]^-{+1} & \fcar{n+1}}.
\]
\end{defi}

\begin{defi}
Let $n$ be a positive integer. A \emph{dead-ends linear graph of length $n$}\index{linear graph!dead-ends} is a graph of the form
\[
 \xymatrix{\fcar{n} & \ar[l]_{\id} \fcar{n} \ar[r]^-{p} & \fcar{n+1} & \ar[l]_-{q} \fcar{n} \ar[r]^{\id} & \fcar{n}}
\]
such that $\leq\card{p^{-1}(v)\sqcup \card{q^{-1}(v)}}=2$ for every $v\in \fcar{n+1}\backslash\{1,n+1\}$ and 
$\card{p^{-1}(1)\sqcup q^{-1}(1)}=\card{p^{-1}(n)\sqcup q^{-1}(n)}=1$. 

An open linear graph of length $n$ is \emph{monotone}\index{linear graph!dead-ends!monotone} if $\card{p^{-1}(v)}$ and $\card{q^{-1}(v)}$ are smaller than $2$ for every $v\in \fcar{n}$
and the node $1$ has one export.

A dead-end linear graph of length $0$ is by definition a graph with set of nodes $\fcar{1}$ and no edges. 

The \emph{monotone dead-ends linear graph of length $n$} is the graph $P_n$ represented by
\[
 \xymatrix{\fcar{n} & \ar[l]_{\id} \fcar{n} \ar[r]^-{+1} & \fcar{n+1} & \ar[l]_-{+0} \fcar{n} \ar[r]^{\id} & \fcar{n}}
\]
\end{defi}

\begin{defi}
Let $n$ be a positive integer. A \emph{semi-open linear graph of length $n$}\index{linear graph!semi-open} is a graph of the form
\[
 \xymatrix{\fcar{n} & \ar[l]_{s} I \ar[r]^p & \fcar{n} & \ar[l]_-{q} O \ar[r]^{r} & \fcar{n}}
\]
such that 
\begin{itemize}
\item[-] for $\card{p^{-1}(v)\sqcup q^{-1}(v)}=2$ for every $v\in \fcar{n}\backslash\{1\}$;
\item[-] $\card{p^{-1}(1)\sqcup q^{-1}(1)}=1$;
\item[-] $I\cap O=\fcar{n}\backslash \{1\}$.
\end{itemize}
\end{defi}

\noindent A semi-open linear graph \emph{starts at a node} if $q^{-1}(1)=1$, otherwise it \emph{ends at a node}.
 
The \emph{monotone semi-open linear graph of length $n$ starting at a node} is the graph:
\[
 \xymatrix{\fcar{n} & \ar[l]_{+0} \fcar{n-1} \ar[r]^-{+1} & \fcar{n} & \ar[l]_-{\id} \fcar{n} \ar[r]^-{\id} & \fcar{n}}
\]
The \emph{monotone semi-open linear graph of length $n$ starting at an edge} is the graph:
\[
 \xymatrix{\fcar{n} & \ar[l]_{\id} \fcar{n} \ar[r]^-{\id} & \fcar{n} & \ar[l]_-{+1} \fcar{n-1} \ar[r]^-{+0} & \fcar{n}}
\]
\begin{defi}
Given a graph $G$ an \emph{open (dead-ends, semi-open) path of length $n$} in $G$ is a morphism from an open (dead-ends, semi-open) linear graph to $G$ which is
injective on nodes and edges. 
A monotone open (dead-ends) path of length $n$ in $G$ is an open (dead-ends) path with source $L_n$ (resp. $P_n$).

For an open path $p$ of length $n$ we will denote by $\stpt{p}$ and $\ept{p}$ the edges $p_A(1)$ and $p_A(n+1)$ respectively;
for a dead-ends path $p$ of length $n$ we will denote by $\stpt{p}$ and $\ept{p}$ the nodes $p_N(1)$ and $p_N(n+1)$;
for a semi-open path of length $n$ starting at a node (starting at an edge) we will denote by 
$\stpt{p}$ (resp. $\ept{p}$) the node $p_N(1)$ and by $\ept{p}$ (resp. $\stpt{p}$) the edge $p_A(1)$.
\end{defi}

\begin{defi}
Let $G$ be a graph $G$; two edges $x,y\in A_G$ are \emph{connected by an open path $p$}\index{connected!edges}  if $\stpt{p}=x$ and $\ept{p}=y$.  

Two nodes $v,u\in N_G$ are \emph{connected by a dead-ends path $p$}\index{connected!nodes} if $\stpt{p}=u$ and $\ept{p}=v$.

A node $v$ and an edge $x$ in $G$ are \emph{connected by a semi-open path $p$ starting at a node (ending at a node)} 
if $\stpt{p}=v$ and $\ept{p}=x$ ($\stpt{p}=x$ and $\ept{p}=v$).  
\end{defi}

\begin{defi}
A graph $G$ is \emph{connected}\index{connected!graph} if it can not be decomposed as a coproduct of smaller graph. 
\end{defi}

\noindent The following proposition is routine to check.
\begin{prop}
A graph $G$ is \emph{connected} if and only if the following two holds:
\begin{itemize}
 \item[-] every two edges $x,y$ in $G$ are connected by an open path;
 \item[-] every two nodes $v,u\in N_G$ are connected by a dead-ends path.
\end{itemize}
\end{prop}
\noindent
Every graph $G$ can be decomposed as the disjoint union of connected smaller graphs in a essentially unique way; the components
of this sum are called the \emph{connected components}\index{connected component!of a graph} of $G$.

There are different ways of defining rooted trees, the following is the one that we chose.

\begin{defi}\label{def:tree}
A connected graph $G$ is a \emph{(rooted) tree} if it has exactly one export $r$ and every edge $x$ in $G$ is connected to $r$ by exactly one monotone open path. 
\end{defi}

\begin{defi}\label{deficie}
For every $v,u\in N_G$ we will denote by $\cie{u}{v}$ (\emph{common inner edges})\index{common inner edge} the set of 
monotone dead-ends paths $p$ of length $1$ such that $p_N(1)=u$ and $p_N(2)=v$; that is all the inner edges that start at $u$ and
end at $v$. 
\end{defi}

\begin{defi}
For every $n\in \N$ we will denote by $W_n$ the graph represented by
\[
 \xymatrix{\fcar{n} & \ar[l]_{\id} \fcar{n} \ar[r]^{\id} & \fcar{n} & \ar[l]_{\id} \fcar{n} \ar[r]^{\id} & \fcar{n}}
\]  
\end{defi}

\begin{defi}
A graph $G$ is \emph{acyclic}\index{graph (directed)!acyclic}\index{acyclic!graph} if the set of morphisms $\Gr(W_n,G)$ is empty for every $n>0$. 
\end{defi}

\noindent We will only be interested in acyclic graphs. They can be pictured as in Figure \ref{firstgraph}: 
the circles are the nodes and the segments are the edges; the ports of a node correspond to the edges that has that node
as an extreme: the inports are the ones above the node while the exports are pictured below the node.  

Of course, similar pictures can be drawn for graphs which are not acyclic but then a direction on the edges (which in Figure 
\ref{firstgraph} is implicitly from the top to the bottom) should be specified to distinguish inports and exports.\\
We would also like to emphasise that, despite the pictures we are drawing, our graphs are not planar.

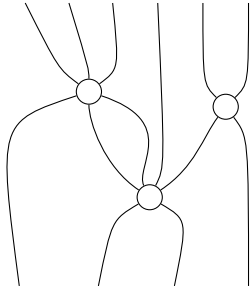
\begin{figure}[!ht]
\begin{tikzpicture}
\node[vert] (a) at (-0.6,0.7) {};
\node[vert] (b) at (0.2,-0.7) {};
\node[vert] (c) at (1.2,0.5) {};
\outports{0,-2}{4}{0.5}
\inports{0,2}{6}{0.3}
\incheckpt{a}{3}{0.2}{0.3}
\outcheckpt{a}{3}{0.6}{-0.4}
\incheckpt{b}{4}{0.2}{0.4}
\outcheckpt{b}{2}{-0.5}{-0.3}
\incheckpt{c}{2}{0.3}{0.3}
\outcheckpt{c}{2}{0.3}{-0.4}
\connectnodess{a}{b}{2/1,3/2}
\connectnodess{c}{b}{1/4}
\connectntoinps{a}{1/1,2/2,3/3}
\connectntoops{a}{1/1}
\connectntoinps{b}{4/3}
\connectntoops{b}{1/3,2/2}
\connectntoinps{c}{5/1,6/2}
\connectntoops{c}{2/4}
\end{tikzpicture} 
\caption{A graphical representation of a graph}\label{firstgraph}
\end{figure}

\subsection{Total orders and unshuffles}\label{not:total orders and unshuffles}
In the upcoming sections we are going to decorate graphs and their nodes with total orders on their ports;
we use this section to fix the notation for permutations, partitions and total orders on finite sets.

A \emph{total order}\index{total order} on a finite set $C$ of cardinality $n$ will be for us a bijection $\morp{\alpha}{\fcar{n}}{C}$.
 Given two total orders $\alpha,\beta$ on a finite set $C$, the \emph{twist}\index{twist!between orders} between $\alpha$ and $\beta$ is the permutation $\otw{\alpha}{\beta}=\beta^{-1}\alpha\cong \Sigma_n$.

Clearly a total order on a set $X$ restricts to a total order on each subset of $X$ in a unique way; the following definition
just express this concept coherently with our definition of total order.

\begin{defi}
Let $\alpha$ be a total order on a finite set $C$ of cardinality $n$ and $D$ be a subset of $C$ of cardinality $m$;
let $\morp{r_{\alpha}}{\fcar{m}}{\fcar{n}}$ be the unique order preserving function with image $\alpha^{-1}(D)$. Then the \emph{restriction
of $\alpha$ to $D$} is the total order $\ror{\alpha}{D}=\alpha r_{\alpha}$.  
\end{defi}

\begin{defi}\label{def:order insertion}
Let $C$ and $D$ be two finite sets, $\morp{\alpha}{\fcar{n}}{C}$ and $\morp{\beta}{\fcar{m}}{D}$ be two total order, and let $d$ be an element of $D$.
The \emph{insertion of $\alpha$ over $\beta$ in $d$}\index{insertion!of order} is the total order $\morp{\beta\circ_d \alpha}{\fcar{n+m-1}}{C\sqcup (D\backslash\{d\})}$ such that
\[
 \beta\circ_i \alpha(j)=
\begin{cases}
\beta(j) & \text{if } j\leq \beta^{-1}(d)  \\
\alpha(i) & \text{if }\beta^{-1}(d) \leq j < \beta^{-1}(d)+n\\
\beta(j-n+1) & \text{if }\beta^{-1}(d)+n \leq j
\end{cases}
\]
\end{defi}

\noindent Let $\morp{\alpha}{\fcar{n}}{\fcar{m}}$ be a function of sets; consider the set $\fcar{m}\times \fcar{n}$ endowed with the lexicographic order: $(a,b)\leq (c,d)$ if and only if $a<c$ or $a=c$ and $b\leq d$ (where $<$ is the usual order on the natural numbers).
Consider the function $\morp{(\alpha,\id)}{\fcar{n}}{\fcar{m}\times \fcar{n}}$. There exists a unique injective order preserving function $\morp{t_\alpha}{\fcar{n}}{\fcar{m}\times \fcar{n}}$. We define the \emph{unshuffling of $\alpha$}\index{unshuffling} as the unique permutation $\morp{\omega_{\alpha}}{\fcar{n}}{\fcar{n}}$
that makes the following diagram commute
\[
\xymatrix{\fcar{n} \ar[r]^{(\alpha,\id)} \ar[dr]_{\omega_{\alpha}} & \fcar{m}\times \fcar{n}\\
	& \fcar{n}.\ar[u]^{t_{\alpha}}}
\]   

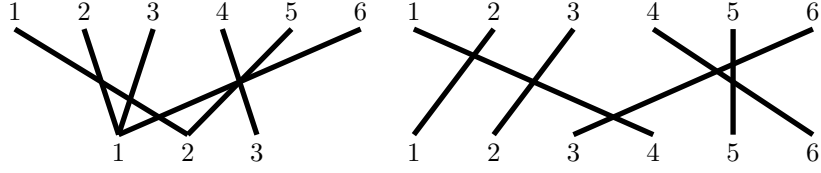
\begin{figure}
    \begin{tikzpicture}[scale=0.7]
    \coordinate (a) at (0,0);
    \coordinate (b) at (0,2);
    \centerp{a}{3}{1.3}{0}
    \centerp{b}{6}{1.3}{0}
    \conx{1/2,2/1,3/1,4/3,5/2,6/1}{b-}{}{a-}{}
    \foreach \x in {1,...,6}
    {
        \node[anchor=south] (ld) at (b-\x) {\x};
    }
    \foreach \x in {1,...,3}
    {
        \node[anchor=north] (ld) at (a-\x) {\x};
    }
    \coordinate (c) at (8,0);
    \centerp{c}{6}{1.5}{0}
    \upcopy{c-1,c-2,c-3,c-4,c-5,c-6}{2}
    \conx{1/4,2/1,3/2,4/6,5/5,6/3}{c-}{-up}{c-}{}
    \foreach \x in {1,...,6}
    {
        \node[anchor=south] (lu) at (c-\x-up) {\x};
        \node[anchor=north] (ld) at (c-\x) {\x};
    }
    \end{tikzpicture}
    \caption{A function $\morp{f}{\fcar{6}}{\fcar{3}}$ (on the left) and its unshuffling $\omega_f$ (on the right).}                                  
\end{figure}

\subsubsection*{Partitions and permutations by blocks}\label{sec:permutation} 
Let $n\in \N$ an \emph{ordered $m$-partition of $n$}\index{ordered partition} is a sequence of $m$ natural numbers $(k_1,\dots,k_m)\in \N^m$ such that 
$\sum_{i=1}^{m} k_i=n$;
Every function $\morp{\alpha}{\fcar{n}}{\fcar{m}}$ determines an ordered $m$-partition of $n$, namely $(\card{\alpha^{-1}(1)},\dots,\card{\alpha^{-1}(m)})$, called the \emph{partition associated to $\alpha$} and denoted by $p_{\alpha}$;
this assignment restricts to a bijection between the set of ordered maps from $\fcar{n}$ to $\fcar{m}$ and the set of ordered $m$-partitions of $n$, therefore we will abusively make no distinctions between these two sets.

Every $m$-partition $p$ of $n$, if we think to $\Sigma_{\card{p^{-1}(i)}}$ as the automorphism group of the set $p^{-1}(1)$, determines a subgroup $i_p\colon \prod_{i=1}^{m} \Sigma_{\card{p^{-1}(i)}} \hookrightarrow \Sigma_n$ that we will denote by $\Sigma_p$ and call the \emph{subgroup of $p$-permutations}; when the partition $p$ we are considering is evident from the context, given $\tau_i\in \Sigma_{\card{p^{-1}(i)}}$ for every $i\in \fcar{m}$ we will denote the image of $(\tau_1,\dots,\tau_m)$ by $[\tau_1\dots \tau_m]$.

For every $m$-partition $p$ of $n$  let let $\Sigma_{p}\backslash \Sigma_n$ be the right coset of $\Sigma_{p}$ (i.e. the set of orbits of $\Sigma_n$ with respect to the left action of $\Sigma_{p}$).

Every function $\morp{\alpha}{\fcar{n}}{\fcar{m}}$ can be decompose as $p_f$ pre-composed with a bijection $\sigma$, uniquely determined up to $p$-permutations; $\sigma$ can always be chosen to be the unshuffling of $f$, i.e. $f=p_f \omega_f$.
In other words, the function
\begin{equation}
 \gfun{d}{\Set(\fcar{n},\fcar{m})}{\underset{(k_1,\dots,k_m)\in \N^m}\coprod \left(\Sigma_{k_1}\times \dots \times \Sigma_{k_m}\backslash \Sigma_n\right)}{f}{[\omega_f]_{p_f}}
\end{equation}
is an isomorphism.

Given $p$ an ordered $m$-partition of $n$ and a permutation  $\sigma\in \Sigma_m$ we will denote by $\block{\sigma}\in\Sigma_n$ the unshuffling $\omega_{\sigma p}$; informally $\block{\sigma}$ is the permutation of $\fcar{n}$ which changes the order of the fibers of $p$ according to $\sigma$, leaving the order of the elements inside of each fiber unchanged.

\subsection{Orders on the ports}

\begin{defi}\label{ordertwist}
Given a graph $G$ and a node $v\in N_G$ a \emph{port-order on $v$}\index{port-order!on a node} is an etale morphism $\morp{\lambda}{C^n_m}{G}$
such that $f(1)=v$ (here $1$ stands for the unique node of $C^n_m$) for some $n,m\in \N$.    
\end{defi}

\noindent Note that such a $\lambda$ exists if and only if $C^n_m$ is isomorphic to $\resd{v}$ therefore $n,m$ are uniquely determined
by $v$. In fact, giving a port-order on $v$ is the same as giving a total order on the inports and a total order on the exports of $\resd{v}$. 

\begin{defi}
A \emph{covering order for a graph $G$}\index{covering order} is a collection $\{\lambda_v\}_{v\in N_G}$ such that $\lambda_v$ is a port-order on $v$.  

A \emph{port-order on a graph $G$}\index{port-order!on a graph} is a port-order on the unique node of $\resd{G}$.

A \emph{node order on $G$}\index{node order} is a bijection $\morp{\sigma}{\fcar{n}}{N_G}$, where $n=\card{N_G}$. 
\end{defi}

\subsection{Completely Ordered Coloured Graphs}\label{apx:completely ordered graph}
\begin{defi}\label{def:c-graph}
For every $C\in \Set$ a \emph{$C$-labelling}\index{labelling} for a graph $G$ is a function of sets $\morp{l}{A_G}{C}$. 

A \emph{$C$-coloured graph}\index{coloured graph} (or \emph{$C$-graph} for short) is a couple $(G,l)$ where $G$ is a graph and $l$ is a $C$-labelling for $G$.

Given two $C$-coloured graphs $(G,l)$ and $(G',l')$, a \emph{morphism of $C$-graphs}\index{morphism!of coloured graphs} is a morphism of graphs $\morp{f}{G}{G'}$ preserving
the $C$-labelling, i.e. $l'f=l$.  
\end{defi}

\noindent A $C$-labelling on $\resd{G}$ induces a $C$-labelling on $\resd{G}$. 

Given a function $\morp{f}{C}{D}$ and a $C$-graph $G=(G,l)$ we will denote by $f(G)$ the $D$-graph $(G,fl)$.

\begin{defi}
An $\emph{ordered $C$-graph}$\index{ordered graph} is a triple $(G,\bar{\lambda},\tau)$ where $G$ is a $C$-graph, $\bar{\lambda}$ is a covering order for $G$ and $\tau$ is a port-order on $G$. Given an ordered $C$-graph $G$, its covering order and port order
will be denoted by $\bar{\lambda}^G$ and $\tau_G$ respectively. 

A \emph{$C$-corolla}\index{$C$-corolla} is an ordered $C$-graph whose underlying graph is a corolla.

An \emph{isomorphism of ordered $C$-graphs} from $G$ to $H$ is an isomorphism of $C$-graphs $\morp{f}{G}{H}$ respecting the given orders, that is $\lambda^H_f(v) f=\lambda^G_v$ for every 
$v\in N_G$ and $\tau_H f=\tau_G$.  

A \emph{completely ordered $C$-graph}\index{completely ordered graph} (a c.o. $C$-graph, for short) is a couple $(G,\sigma)$ where $G$ is an ordered $C$-graph and $\sigma$ is a node order on $G$. Given a c.o. $C$-graph $G$, its node order will be denoted by $\sigma_G$.

An \emph{isomorphism of completely ordered $C$-graphs} from $G$ to $H$  is an isomorphism of ordered $C$-graphs such that $\sigma_H f=\sigma_G$.
\end{defi}

\noindent Completely ordered $C$-graphs can be pictured as in Figure \ref{secondgraph}. 
The labels on the edges are elements of $C$. The numbers in the nodes determines the node order. 
the circles are the nodes and the segments are the edges; the inports and exports of the graph are ordered from the left to the right;
the same holds for inports and exports of the nodes.  

\begin{figure}[!ht]
\begin{tikzpicture}[scale=1.2]
\node[vert] (a) at (-0.6,0.7) {2};
\node[vert] (b) at (0.2,-0.6) {1};
\node[vert] (c) at (0.6,0.8) {3};
\outports{0,-2}{4}{0.5}
\inports{0,2}{6}{0.3}
\incheckpt{a}{3}{0.2}{0.6}
\outcheckpt{a}{3}{0.6}{-0.4}
\incheckpt{b}{3}{0.4}{0.4}
\outcheckpt{b}{2}{0.5}{-0.3}
\incheckpt{c}{2}{0.3}{0.6}
\outcheckpt{c}{2}{0.3}{-0.4}
\connectnodes{a}{b}{2/1/c,3/3/d}
\connectnodes{c}{b}{1/2/a}
\connectntoinp{a}{1/1/a,2/2/b,4/3/c}
\connectntoop{a}{1/1/c}
\connectntoop{b}{1/3/d,2/2/d}
\connectntoinp{c}{3/1/c,6/2/d}
\connectntoop{c}{2/4/b}
\end{tikzpicture} 
\caption{A graphical representation of a c.o. acyclic $C$-graph}\label{secondgraph}
\end{figure}
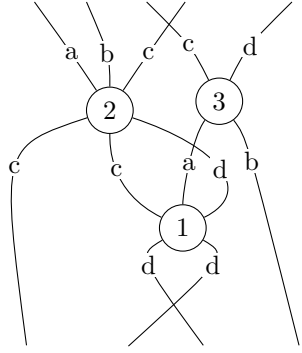

\begin{defi}\label{defpermgr}
Given a permutation $\gamma\in \Sigma_n$ and a completely ordered $C$-graph $G=(G,\sigma)$ with $n$ nodes, we will denote by 
$\gamma^*G$ the completely ordered $C$-graph $(G,\sigma \gamma)$. 
\end{defi}

\noindent The only difference between $G$ and $\gamma^*G$ is in the node order (permuted by $\gamma$).

\subsection{Twists}

Suppose a c.o. $C$-corolla $(D,\bar{\lambda}^D,\tau_D,\sigma_D)$ is given. Since $D$ has only one node $\ast$ there is no choice for $\sigma$ while the collection $\bar{\lambda}^D$ contains a unique element $\lambda_*$, that can be regarded as a port-order on $D$.

Consider now a c.o. $C$-graph $G$ (even though the $C$-labelling will not be relevant for this discussion) 
and two nodes $u,v\in N_G$. Let $l$ be the cardinality of the set of common inner edges $\cie{u}{v}$ (def. \ref{deficie});
the total ordering $\bar{\lambda}^G$ induces two different orders on $\cie{u}{v}$; the first is obtained by restriction from $\lambda_{u}^G$ considering
$\cie{u}{v}$ as a subset of the inports of $u$, let us denote it by $\morp{\gamma_u}{\fcar{l}}{\cie{u}{v}}$;  the second is obtained in the same way from $\lambda_{u}^G$ considering $\cie{u}{v}$ as a subset of
the exports of $v$, denote it by $\morp{\gamma_v}{\fcar{l}}{\cie{u}{v}}$.

\begin{defi}
Given a c.o. graph $G$ and two nodes $u,v\in N_G$ the \emph{twist between $u$ and $v$}\index{twist!between nodes}, denoted by $\tw{u}{v}$ is the twist (\S \ref{not:total orders and unshuffles}) $\tw{\gamma_u}{\gamma_v}\in \Sigma_l$ (where
$l=\card{\cie{u}{v}}$). 
\end{defi}

\noindent In the same way we can express the difference between the port order and the covering order. 

\begin{defi}
Given a c.o. graph $G$ and a node $v$ the set $l(v)=\iport{G}\cap \iport{v}$ comes with two total order: $\iota$ induced from 
the port order $\sigma_G$ and $\iota'$ induced from $\lambda^G_v$; the \emph{input-twist of $G$ in $v$}\index{input-twist} is $\tw{v}{G}_{in}=\tw{\iota'}{\iota}$.
Similarly, there are two induced order on $r(v)=\oport{G}\cap \oport{v}$: $\omega$ induced from $\sigma_G$ and $\omega'$ induced form
$\lambda_v^G$; the \emph{output-twist of $G$ in $v$}\index{output-twist} is $\tw{v}{G}_{out}=\tw{\omega'}{\omega}$.
\end{defi}

\begin{defi}
Given a c.o. $C$-corolla $(D,\bar{\lambda}^D,\tau_D,\sigma_D)$ of valence $(n,m)$, the \emph{twist}\index{twist!of a c.o. corolla} of $D$ is
$(\tw{\ast}{D}_{in},\tw{\ast}{D}_{out})\in \Sigma_n\times \Sigma_m$, where $\ast$ is the unique node of $D$.
\end{defi}

\begin{rmk}
Note that given a $C$-labelled corolla $D$ of valence $(n,m)$, a covering order $\bar{\lambda}^D$ and an element $\alpha\in \Sigma_n \times \Sigma_m$ there is a unique c.o. $C$-corolla structure $(D,\bar{\lambda}^D,\tau_D,\sigma_D)$ with twist $\alpha$
(up to unique isomorphisms of completely ordered graphs).
\end{rmk}
\begin{defi}\label{defi: untwisted and planar}
 A c.o. $C$-graph is \emph{untwisted}\index{untwisted graph} if $\tw{u}{v}$, $\tw{v}{G}_{in}$ and $\tw{v}{G}_{out}$ are equal to the identity for every $v,u\in N_G$.
\end{defi}
  
\subsection{Automorphisms of ordered graphs}
Completely ordered $C$-graphs have always trivial automorphism group. This is not true in general for ordered $C$-graphs, for example
the ordered graphs in Figure \ref{fixedpoint} and \ref{fixedpoint2} at page \pageref{fixedpoint}  have non-trivial automorphisms. However non-trivial automorphisms can only occur
 if the graph has components with no inports and exports, as the following proposition shows.
\begin{prop}\label{trivialauto}
Let $G$ be an ordered $C$-graph such that every node of $G$ is connected (via a semi-open path) to an inport or an export, then the automorphisms
group of $G$ is trivial. 
\end{prop}
\begin{proof}
It is easy to show that under our hypothesis every edge $e\in A_G$ is connected to an inport or an export. For every $v\in N_G$ let 
$d(v)$ be the minimal length of a semi-open path connecting $v$ to an inport or an export. Similarly, for every $e\in A_G$ let $\delta(e)$
be the minimal length of an open path connecting $e$ to an inport or an export.

Let $f$ be an automorphism of $G$. If $\delta(e)=0$ then $e$ is a port and since $f$ preserves the port-order $f(e)=e$.

We are going to prove by induction on $n\in \N$ that for every $v\in N_G$ ( $e\in A_G$) if $d(v)=n$ (resp. $\delta(e)=n$)
then $f(v)=v$ (resp. $f(e)=e$).

We have already checked the case $n=0$. Suppose that $n>0$ and that the statement holds for every $m<n$. 
Let $v\in A_G$ be such that $d(v)=n$ and let $p$ be a semi-open path of minimal length from $v$ to a port $r$, then $v$ has a port
$c$ (in the image of $p$) such that $\delta{c}<n$; by inductive hypothesis $f(e)=e$, thus $f(v)=v$.  
Let $e\in A_G$ be such that $\delta(e)=n$ and let
$p$ be an open path of minimal length from $v$ to a port $r$, then $e$ is a port of a node $u$ (in the image of $p$) such that $d(u)\leq n$; by inductive hypothesis
$f(u)=u$ and since $f$ preserves the covering orders $f(e)=e$.   
\end{proof}

\subsection{Graph Insertion}\label{apx:insertion}\index{insertion, graph}
In this section we are going to recall the graph insertion operation; 
intuitively this should encode a way of plug a $C$-graph into a node (of another $C$-graph) with the same valence.

Suppose two graphs $H$ and $G$, a node $v\in N_G$ and an isomorphism $\morp{\phi}{\resd{v}}{\resd{H}}$ are given (note that $\phi$
exists if and only if $H$ and $v$ have the same valence).
Via $\phi$ we can define the following sets:
\[
A'= A_H\!\!\!\underset{A_{\resd{v}}}\sqcup\!\!\! A_G,\ \  I'= I_H\!\!\!\underset{I_{\resd{v}}}\sqcup\!\!\! I_G,\ \ O'= O_H\!\!\!
\underset{O_{\resd{v}}}\sqcup\!\!\! O_G.\]
The insertion of $H$ in $v$
is the graph defined by the diagram:
\[
 \xymatrix{A' & \ar[l]_-{s_H \sqcup s_G} I' \ar[r]^-{p_H \sqcup p_G} & N_H \sqcup (N_G\backslash \{v\}) & \ar[l]_-{q_H \sqcup q_G} O' \ar[r]^-{t_H \sqcup t_G} & A'}
\]
and denoted by $G \circ_{\phi} H$.

Note that
\begin{itemize}
 \item[-] for every $u\in N_G\backslash\{v\}$ ( $u\in N_H$) the residue of $u$ as a node of $G$ (resp. of $H$) is canonically isomorphic to the residue of $u$ as a node of $G \circ_{\phi} H$.
 \item[-] the residue of $G\circ_{\phi} H$ is canonically isomorphic to the residue of $G$.
\end{itemize}

The set of $C$-valences (\S \ref{sec:valences and bicollections})\index{valence!$C$-} is isomorphic to $\widetilde{\Val{C}}$, the set of triple $(n,m,l)$ where $n,m\in \N$, and $l$ is a $C$-labelling for the corolla $C^n_m$;
a bijection is explicitly given by
\[
 \gfun{\phi}{\widetilde{\Val{C}}}{\Val{C}}{\mathbf{v}=(n,m,l)}{(\inva{\mathbf{v}},\outva{\mathbf{v}})}
\]
where $\inva{\mathbf{v}}=(l(1),\dots,l(n))$ and $\outva{\mathbf{v}}=(l(n+1),\dots,l(n+m))$.
From now on we will make no distinction between $\Val{C}$ and $\widetilde{\Val{C}}$.  

\begin{figure}[!ht]
\[
\mathbf{v}=(n,m,l) = \ioval{l(1),\dots,l(n)}{l(n+1),\dots,l(m)}=\ioval{\mathbf{v}_{in}}{\mathbf{v}_{out}}
\]
\caption{Representation of a $C$-labelling of $C_m^n$ as a $C$-valence.}
\end{figure}
 
Given a completely ordered $C$-graph $G$, each node $v\in N_G$ comes with an associated $C$-valence $\cval{v}=(n,m,l_v)$ where
$(n,m)$ is the valence of $v$ and $l_v$ is the composition $l_G \lambda_v^G$; this $C$-valence will be called the \emph{$C$-valence of $v$}\index{valence!$C$-!of a node}.

In the same way, the residue $\resd{G}$ come also with a $C$-valence associated: if $(n,m)$ is the valence of $G$ then the $C$-valence 
$(n,m,l_G \tau_G)$ is called the \emph{$C$-valence of $G$}\index{valence!$C$-!of a graph} and denoted by $\cval{G}$.

\begin{defi}\label{def:arity of a graph}
For every c.o. $C$-graph $G$ there exist unique $n\in \N$ and $s_0,s_1,\dots,s_n\in \Val{C}$ such that
\begin{itemize}
 \item[-] $\card{N_G}=n$;
 \item[-] $\cval{\sigma_G(i)}=s_i$ for every $i\in{1,\dots,n}$;
 \item[-] $\cval{G}=s_0$.
\end{itemize}

The $\Val{C}$-signature $(s_1,\dots,s_n;s_0)$ is called the \emph{arity of $G$}\index{arity!of a graph}.
\end{defi}

\begin{defi}
For every $C$-valence $\mathbf{c}$ let $C_{\mathbf{c}}$ be the unique (up to isomorphisms) untwisted $C$-corolla with valence $\mathbf{c}$.  
\end{defi}

\noindent Suppose now that $H$ and $G$ are c.o. $C$-graphs, $v$ is a node of $G$ and that $v$ and $H$ has the same $C$-valence $(n,m,l)$. This automatically induce and isomorphism $\morp{\phi}{\resd{v}}{\resd{H}}$
\[
 \xymatrix{\resd{v} \ar[r]^-{\lambda_v^G} & C^n_m \ar[r]^-{\tau_H} & \resd{H}}
\]
We will denote by $G\circ_v H$ the insertion $G\circ_{\phi} H$.

Note that, since the $C$-valences of $v$ and $H$ are assumed to be the same, the $C$-labellings of $G$ and $H$ induce a labelling
on $G\circ_{\phi} H$.
The port-order $\tau_G$ induces a port order on $G\circ_v H$ and the given covering orders for $G$ and $H$ induce a covering order for $G\circ_v H$.
Let $a,b$ be the number of nodes of $G$ and $H$ respectively. The only thing missing to endow $G\circ_v H$ with a completely ordered $C$-graph structure is a node order, i.e. an isomorphism
$\morp{\sigma_{G\circ_v H}}{\fcar{a+b-1}}{N_H \sqcup N_G\backslash\{v\}}$; let $k=\sigma_G^{-1}(v)$, then $\sigma_{G\circ_v H}$
is defined in the following way
\[
 \sigma_{G\circ_v H}(i)=\begin{cases}\sigma_G(i) & i< k\\
                               \sigma_H(i-k+1) & k \leq i < k+b\\
                               \sigma_G(i-b) & i\geq k+b. \end{cases}
\]
This complete the definition of the structure of completely ordered $C$-graph that $G\circ_v H$ inherits from $G$ and $H$; it will always be considered with this structure.

It is easy to check that the insertion of an untwisted corolla do not change the graph, that is
\[
 G\cong G\circ_v C_{\cval{v}}
\]
for every $v\in N_G$.

The proof of the following proposition is routine.

\begin{prop}\label{prop.gcomp}
Suppose that $G$ is a c.o. $C$-graph, $v$ is a node in $G$ and $H_1$ is a c.o. graph with the same $C$-valence as $v$:
\begin{itemize}
 \item[-] if $G$ and $H_1$ are acyclic then $G\circ_{v} H_1$ is acyclic;
 \item[-] if $u$ is a node of $G$ different from $v$ and $H_2$ is a c.o. graph with the same $C$-valence as $v$ then $(G\circ_v H_1)\circ_u H_2\cong (G\circ_u H_2)\circ_v H_1$; 
 \item[-] if $k$ is a node of $H_1$ and $H_2$ is a graph with the same $C$-valence as $k$ then $(G\circ_v H_1)\circ_k H_2\cong
 (G\circ_v (H_2\circ_k H_1)$.
 \item[-] if $H_1\cong \resd{v}$ then $G\circ_v H_1\cong G$;
 \item[-] let us endow $\resd{H_1}$ with the unique c.o. structure with trivial twist, then $\resd{H_1}\circ_{\ast} H_1=H_1$. 
\end{itemize}
\end{prop}

\subsubsection{Multiple insertion}\label{multins}
Let $(G,\sigma_G)$ be a c.o. $C$-graph with $\card{N_G}=m$; for every $\mu\in \Sigma_m$ we define $\mu^*(G)$ to be the
c.o. graph $(G,\sigma_G\mu)$, that is the same ordered $C$-graph with the node order permuted by $\mu$.
  
Suppose that a sequence of c.o. $C$-graph $\{K_i\}_{i=1}^m$ is given such that $\resd{K_i}=\resd{\sigma_G(i)}$ and $\card{N_{K_i}}=k_i$ for every $i\in \fcar{m}$, such a sequence is
called \emph{insertable over $G$}\index{insertable!sequence of graphs}; define the $C$-graph $G\circ (K_1,\dots,K_m)$ to be 
the iterated insertion:
\[
(\dots(((G\circ_{\sigma_G(1)} K_1)\circ_{\sigma_G(2)} K_2)\dots)\circ_{\sigma_G(m)} K_m)
\]
(note that the order in which we insert the $K_i$'s does not matter, in the light of Proposition \ref{prop.gcomp}).

The number of nodes of $G\circ (K_1,\dots,K_m)$ is $n=\sum_{i=1}^m k_i$.

\begin{defi}\label{permins}
Given a sequence of c.o. graphs $\{K_i\}_{i=1}^m$ a function $\morp{\alpha}{\fcar{n}}{\fcar{m}}$ is
a \emph{permutation for the insertion of $\{K_i\}_{i=1}^m$} if $\card{N_{K_i}}=\card{\alpha^{-1}(i)}$
for every $i\in \fcar{m}$. 
\end{defi}

\noindent In other words, every sequence of graphs $(K_1,\dots,K_m)$ determines an ordered $m$-partition $(\card{N_{K_1}},\dots,\card{N_{K_m}})$
and a permutation for $(K_1,\dots,K_m)$ is a function with the same associated partition (\S \ref{not:total orders and unshuffles}).

Given a sequence $\{K_i\}_{i=1}^m$ insertable over $G$ and a permutation $\alpha$ for the insertion of $\{K_i\}_{i=1}^m$, we define the c.o. $C$-graph $G\circ_{\alpha} (K_1,\dots,K_m)$
to be 
\[
\omega_{\alpha}^*(G\circ (K_1,\dots,K_m))
\]
where $\omega_{\alpha}$ is the unshuffling of $\alpha$ (\S \ref{not:total orders and unshuffles}).

An insertion can also be  regarded as a particular case of multiple insertion, in fact given two c.o. $C$-graph $G$ and $K$
such that $\cval{K}=\cval{\sigma_G(m)}$ for some $m\in \fcar{\card{N_G}}$
\[
 G\circ_{\sigma_G(m)} K \cong G\circ (C_{\sigma_G(1)},\dots, C_{\sigma_G(m-1)},K,C_{\sigma_G(m+1)},\dots, C_{\sigma_G(\card{N_G})});
\]
If $\alpha$ is a permutation for the insertion of $(C_{\sigma_G(1)},\dots, C_{\sigma_G(m-1)},K,C_{\sigma_G(m+1)},\dots, C_{\sigma_G(\card{N_G})})$
over $G$ we will simply write $G\circ_{\sigma_G(m),\alpha} K$ for
\[
 G\circ_{\alpha} (C_{\sigma_G(1)},\dots, C_{\sigma_G(m-1)},K,C_{\sigma_G(m+1)},\dots, C_{\sigma_G(\card{N_G})}).
\]
\subsection{Decomposition of Acyclic Graphs}\label{coa.decomp}

A completely ordered acyclic $C$-graph will often be called simply a \emph{\coa graph}\index{coa graph}.

We first notice that, given an acyclic graph $G$ and two distinct nodes $x,y\in N_G$, there are three possible scenarios:
\begin{itemize}
 \item[-] there are no dead-ends monotone paths between $x$ and $y$; in this case we will say that $x$ and $y$ are \emph{disconnected}\index{disconnected!nodes};
 \item[-] all monotone dead-ends paths $p$ between $x$ and $y$ are such that $p_N(1)=x$ (that is the path ``starts'' at $x$); in this case
we say that \emph{$x$ dominates $y$}; 
 \item[-] \emph{$y$ dominates $x$}.
\end{itemize}
\noindent We are now going to show that every \coa graph can be written as iterated insertions of certain simple \coa graphs 
(called composition graph) over a corolla.
\begin{defi}\label{def:graph comp}
A \emph{composition graph} is a \coa graph $G$ with exactly two nodes.

A \emph{horizontal composition graph} (or \emph{merging graph} or \emph{hc graph} for short) is a composition graph such that the its two nodes have no common inner edges.

A composition graph which is not a horizontal composition graph is called a \emph{vertical composition graph} (or \emph{vc graph} for short).

\end{defi}

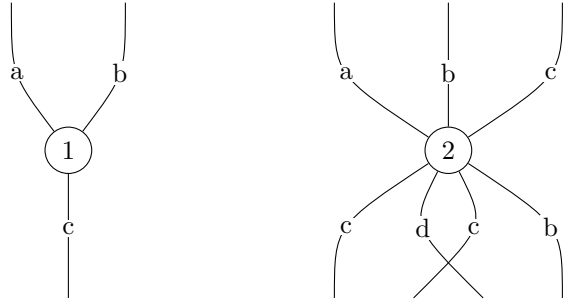
\begin{figure}[!ht]
\centering
 \begin{tikzpicture}
\node (c) at (-2.5,0) [vert] {1};
\newcommand{\inleac}{2}
\newcommand{\outleac}{1}
\foreach \x \y \z in {1/1/a,2/2/b}
   {
    \node[lea] (ilc\x) at ($(c)+(1.5*\x-0.75*\inleac -0.75,2)$) {};
    \node (imc\x) at ($(c)+(1.5*\y-0.75*\inleac-0.75,1)$) {};
    \draw (ilc\x) .. controls (imc\x) .. node[lea] {\z} (c);
   }
\foreach \x \y \z in {1/1/c}
   {
    \node[lea] (olc\x) at ($(c)+(1.5*\x-0.75*\outleac -0.75,-2)$) {};
    \node (omc\x) at ($(c)+(1.5*\y-0.75*\outleac-0.75,-1)$) {};
    \draw (olc\x) .. controls (omc\x) .. node[lea] {\z} (c);
   }
\node (d) at (2.5,0) [vert] {2};
\newcommand{\inlead}{3}
\newcommand{\outlead}{4}
\foreach \x \y \z in {1/1/a,2/2/b,3/3/c}
   {
    \node[lea] (ild\x) at ($(d)+(1.5*\x-0.75*\inlead -0.75,2)$) {};
    \node (imd\x) at ($(d)+(1.5*\y-0.75*\inlead-0.75,1)$) {};
    \draw (ild\x) .. controls (imd\x) .. node[lea] {\z} (d);
   }
\foreach \x \y \z in {1/1/c,2/3/c,3/2/d,4/4/b}
   {
    \node[lea] (old\x) at ($(d)+(1*\x-0.5*\outlead -0.5,-2)$) {};
    \node (omd\x) at ($(d)+(1*\y-0.5*\outlead-0.5,-1)$) {};
    \draw (old\x) .. controls (omd\x) .. node[lea] {\z} (d);
   }
\end{tikzpicture}\hspace{1cm}
\caption{A horizontal composition}
\end{figure}

\begin{figure}
\begin{tikzpicture}[scale=1.5]
\node[vert] (a) at (0,0.7) {1};
\node[vert] (b) at (0.5,-0.7) {2};
\outports{0,-2}{3}{0.5}
\inports{0,2}{5}{0.3}
\incheckpt{a}{3}{0.5}{0.3}
\outcheckpt{a}{3}{0.6}{-0.4}
\incheckpt{b}{4}{0.3}{0.4}
\outcheckpt{b}{2}{-0.5}{-0.3}
\connectnodes{a}{b}{2/3/a,3/2/b}
\connectntoinp{a}{1/1/c,3/3/b,4/2/a}
\connectntoop{a}{1/1/d}
\connectntoinp{b}{2/1/c,5/4/c}
\connectntoop{b}{1/3/a,2/2/d}
\end{tikzpicture}
\caption{A vertical composition graph}
\end{figure}
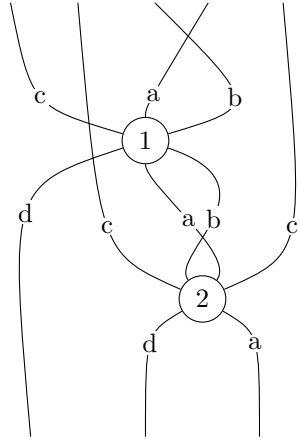

Let $G$ be a \coa $C$-graph an let $x,y\in N_G$ be two vertices with no dead-ends path of length greater than one between them
such that $x$ dominates $y$.
Suppose that $\sigma_G^{-1}(x)<\sigma_G^{-1}(y)$. Define $\iport{x,y}=\iport{y}\cap \cie{x}{y}$ and $\oport{x,y}=\oport{x}\cap \cie{x}{y}$.

Let $H^G_{xy}$ be the c.o. $C$-graph, whose underlying graph is
\[
\xymatrix{A_G\backslash \cie{x}{y} & \ar[l]^s I\backslash\iport{y} \ar[r]^p & N_G/\{x\sim y\} & \ar[l]^q O\backslash\iport{x} \ar[r]^t & A}
\]
and whose $C$-labelling, port order are obtained by restriction from the ones of $G$; the covering order is also induced
by restriction from $G$, except on the node $[x]$; let $i$ and $j$ the minimal elements in $\iport{x,y}$
and $\oport{x,y}$ respectively,  
(with respect to the covering order $\lambda^G$); on the node $[x]$ the order on 
$\iport{[x]}=(\iport{y}\backslash \iport{x,y})\cup \iport{x}$ (resp. $\oport{[x]}=(\oport{x}\backslash \oport{x,y})\cup \oport{y}$
is the restriction of the insertion (def.\ref{def:order insertion}) of the order on $\iport{x}$ over the order $\iport{y}$ in $i$ (resp. of the order on $\oport{y}$ over the order $\oport{x}$ in $j$).

The node order $\sigma_H$
is determined by the requirements that
\[          
\forall u,v\in N_H\backslash [x] \ \ \sigma_H^{-1}(u)\leq \sigma_H^{-1}(v)\ \Leftrightarrow \ \sigma_G^{-1}(u)\leq \sigma_G^{-1}(v)
\]
and 
\[
 \forall u\in N_H\backslash [x]\ \  \sigma_H^{-1}(u)< \sigma_H^{-1}([x]).
\]
\noindent Informally $H^G_{xy}$ is obtained by $G$ by shrinking $x$, $y$ and their common inner edges in one node.

Let $K^G_{xy}$ be the composition graph
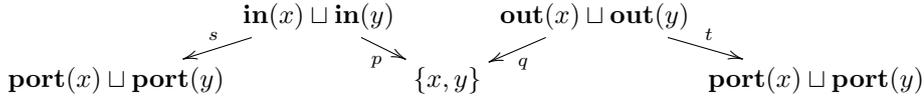
\[
\xymatrix@R=8pt@C=1pt{ & \ar[ld]_-{s} \iport{x}\sqcup \iport{y} \ar[rd]_-{p} & & \ar[ld]^-{q} \oport{x}\sqcup\oport{y}
\ar[rd]^-{t} & \\
\port{x}\sqcup\port{y} & & \{x,y\} &  &\port{x}\sqcup\port{y}}
\]
with $C$-colouring, covering order and node order induced from the ones of $G$; the residue of $K^G_{xy}$ has the same ports
of the node $[x]$ of $H_{xy}^G$ and we set the port order on $K^G_{xy}$ equal to the port order of $[x]$ in $H_{xy}^G$. 

Note that $K^G_{xy}$ is a horizontal composition if and only if $x$ and $y$ are disconnected.

Let $\morp{\alpha_{xy}}{\fcar{n}}{\fcar{n-1}}$ be the composition $\sigma_G \pi \sigma_H^{-1}$, where $\morp{\pi}{N_G}{N_H=N_G/\{x\sim y\}}$
is the quotient map.

The following proposition asserts that if we insert $K^G_{xy}$ over $H^G_{xy}$ at $[x]$ we get $G$ again. 
\begin{prop}\label{prop:grapheldecomp}
Let $G$ be a \coa $C$-graph an let $x,y\in N_G$ be two nodes with no monotone death-paths of length greater than one between them.
Then $H^G_{xy}$ is a \coa $C$-graph and 
\[
 G\cong H^G_{xy}\circ_{[x],\alpha_{xy}} K^G_{xy}.
\]
\end{prop}
\begin{proof}
The proof of the last formula is routine, thus we just prove that $H$ is acyclic.

Suppose there is a loop $\morp{l}{W_n}{H}$; if $[x]$ is not in the image of $l$ then  $l$ would define a loop in $G$
which is in contradiction with the fact that $G$ is acyclic; therefore $[x]$ is in the image of $l$; on the other hand this
would implies the existence of a monotone dead-ends path of length at least $2$ in $G$ between $x$ and $y$ and this is in contradiction with
our hypothesis. This implies that $H$ is acyclic.   
\end{proof}

\begin{prop}
Given a \coa $C$-graph $G$ one of the followings always occurs:
\begin{enumerate}[i.]
 \item $G$ has no dead-ends paths of length greater than $0$;
 \item there exist two distinct nodes $x,y\in N_G$ such that $\cie{x}{y}$ is non-empty and there are no monotone dead-ends path between them
  of length greater than $1$.
\end{enumerate}
\begin{proof}
Suppose that i. do not hold. Then there are two nodes $x,y\in N_G$ such that there exist a monotone death-path from $x$ to $y$. Let $p$ be a
monotone dead-ends path of maximal length $n$ between $x$ and $y$, it is not restrictive to suppose that $p_N(1)=x$; if $n=1$ then ii. holds. Otherwise take $y'$ to be $p_v(2)$, then $x$ and $y'$ are connected and there is no monotone dead-end path from $x$ to $y'$ of length greater than $1$, otherwise
there would be a monotone dead-end path from $x$ to $y$ of length bigger than $n$; therefore $x$ and $y'$ satisfy condition ii.  
\end{proof} 
\end{prop}

\printbibliography
\end{document}